\documentclass[final]{siamart0216}
\usepackage{amsfonts}

\usepackage{bm}
\usepackage{amsfonts,amsmath,amssymb}
\usepackage{mathrsfs,mathtools,stmaryrd,wasysym}
\usepackage{enumerate}
\usepackage{enumitem}
\usepackage{esint}
\usepackage{graphicx,psfrag}
\usepackage{hyperref}
\DeclareMathAlphabet{\mathpzc}{OT1}{pzc}{m}{it}

\usepackage{lineno}
\usepackage{pgf,tikz,pgfplots}
\usepackage{pstricks-add}

\DeclareFontEncoding{FMS}{}{}
\DeclareFontSubstitution{FMS}{futm}{m}{n}
\DeclareFontEncoding{FMX}{}{}
\DeclareFontSubstitution{FMX}{futm}{m}{n}
\DeclareSymbolFont{fouriersymbols}{FMS}{futm}{m}{n}
\DeclareSymbolFont{fourierlargesymbols}{FMX}{futm}{m}{n}
\DeclareMathDelimiter{\VERT}{\mathord}{fouriersymbols}{152}{fourierlargesymbols}{147}

\newcommand{\T}{\mathscr{T}}
\newcommand{\Sides}{\mathscr{S}}

\usetikzlibrary{arrows}

\newsiamremark{remark}{Remark}

\newcommand{\TheTitle}{A posteriori error estimates for semilinear optimal control problems}
\newcommand{\ShortTitle}{Error estimates for a semilinear optimal control problem}
\newcommand{\TheAuthors}{A. Allendes, F. Fuica, E. Ot\'arola, D. Quero}

\headers{\ShortTitle}{\TheAuthors}

\title{{\TheTitle}\thanks{AA is partially supported by CONICYT through FONDECYT project 1170579. EO is partially supported by CONICYT through FONDECYT Project 11180193.}}

\author{Alejandro Allendes\thanks{Departamento de Matem\'atica, Universidad T\'ecnica Federico Santa Mar\'ia, Valpara\'iso, Chile.
(\email{alejandro.allendes@usm.cl}).}
\and
Francisco Fuica\thanks{Departamento de Matem\'atica, Universidad T\'ecnica Federico Santa Mar\'ia, Valpara\'iso, Chile.
(\email{francisco.fuica@sansano.usm.cl}).}
\and
Enrique Ot\'arola\thanks{Departamento de Matem\'atica, Universidad T\'ecnica Federico Santa Mar\'ia, Valpara\'iso, Chile.
(\email{enrique.otarola@usm.cl}, \url{http://eotarola.mat.utfsm.cl/}).}
\and
Daniel Quero\thanks{Departamento de Matem\'atica, Universidad T\'ecnica Federico Santa Mar\'ia, Valpara\'iso, Chile.
(\email{daniel.quero@alumnos.usm.cl}).}}

\ifpdf
\hypersetup{
  pdftitle={\TheTitle},
  pdfauthor={\TheAuthors}
}
\fi

\date{Draft version of \today.}

\begin{document}

\maketitle

\begin{abstract}
We devise and analyze a reliable and efficient a posteriori error estimator for a semilinear control--constrained optimal control problem in two and three dimensional Lipschitz, but not necessarily convex, polytopal domains. We consider a fully discrete scheme that discretizes the state and adjoint equations with piecewise linear functions and the control variable with piecewise constant functions. The devised error estimator can be decomposed as the sum of three contributions which are associated to the discretization of the state and adjoint equations and the control variable. We extend our results to a scheme that approximates the control variable with piecewise linear functions and also to a scheme that approximates a nondifferentiable optimal control problem. We illustrate the theory with two and three--dimensional numerical examples.
\end{abstract}



\section{Introduction}\label{sec:Intro}

In this work we will be interested in the design and analysis of a posteriori error estimates for finite element approximations of a semilinear control--constrained optimal control problem: the state equation corresponds to a Dirichlet problem for a monotone, semilinear, and elliptic partial differential equation (PDE). To describe our control problem, for $d\in\{2,3\}$, we let $\Omega\subset\mathbb{R}^{d}$ be an open and bounded polytopal domain with Lipschitz boundary $\partial\Omega$. Notice that we do not assume that $\Omega$ is convex. Given a regularization parameter $\nu>0$ and a desired state $y_\Omega\in L^2(\Omega)$, we define the cost functional
\begin{equation}\label{def:cost_func}
J(y, u):=\frac{1}{2} \| y - y_\Omega \|^{2}_{L^{2}(\Omega)} + \frac{\nu}{2}\|u\|_{L^2(\Omega)}^2.
\end{equation} 
With these ingredients at hand, we define the \emph{semilinear elliptic optimal control problem} as: Find $\min J(y,u)$ subject to the \emph{monotone, semilinear, and elliptic PDE}
\begin{equation}\label{def:state_eq}
-\Delta y + a(\cdot,y)  =  u  \text{ in }  \Omega, \qquad
y  =  0  \text{ on }  \partial\Omega,
\end{equation}
and the \emph{control constraints}
\begin{equation}\label{def:box_constraints}
u \in \mathbb{U}_{ad},\quad \mathbb{U}_{ad}:=\{v \in L^2(\Omega):  \texttt{a} \leq v(x) \leq \texttt{b} \text{ a.e. } x \in \Omega \};
\end{equation}
the control bounds $\texttt{a},\texttt{b} \in \mathbb{R}$ are such that $\texttt{a} < \texttt{b}$. Assumptions on the function $a$ will be deferred until section \ref{sec:assumption}.

The analysis of error estimates for finite element approximations of semilinear optimal control problems has previously been considered in a number of works. The article \cite{MR1937089} appears to
be the first to provide error estimates for the distributed optimal control problem \eqref{def:cost_func}--\eqref{def:box_constraints}; notice that control constraints are considered. The authors of this work propose a fully discrete scheme on quasi--uniform meshes that discretizes the control variable with piecewise constant functions; piecewise linear functions are used for the discretization of the state and adjoint variables. In two and three dimensions and under the assumptions that $\Omega$ is convex, $\partial \Omega$ is of class $C^{1,1}$, and that the mesh--size is sufficiently small, the authors derive a priori error estimates for the approximation of the optimal control variable in the $L^2(\Omega)$-norm \cite[Theorem 5.1]{MR1937089} and the $L^{\infty}(\Omega)$-norm \cite[Theorem 5.2]{MR1937089}; the ones derived in the $L^2(\Omega)$-norm being optimal in terms of approximation. The analysis performed in \cite{MR1937089} was later extended in \cite{MR2350349} to a scheme that approximates the control variable with piecewise linear functions. The main result of this work reads as follows: $ h_{\T}^{-1} \| \bar u - \bar u_{\T} \|_{L^2(\Omega)} \rightarrow 0$ as $h_{\T} \downarrow 0$ \cite[Theorem 4.1]{MR2350349}, where $\bar u_{\T}$ denotes the corresponding finite element approximation of the optimal control variable $\bar u$. Under a suitable assumption, this result was later improved to 
\[
\| \bar u - \bar u_{\T} \|_{L^2(\Omega)}  \lesssim h_{\T}^{3/2};
\] 
see \cite[section 10]{MR3586845}. We conclude by providing a non-exhaustive list of extensions available in the literature: boundary optimal control \cite{MR2150243}, sparse optimal control \cite{MR3023751}, Dirichlet boundary optimal control \cite{MR2272157}, and state constrained optimal control \cite{MR2009948}.

While it is fair to say that the study of a priori error estimates for finite element solution techniques of semilinear optimal control problems is matured and well understood, the analysis of a posteriori error estimates is far from complete. An a posteriori error estimator is a computable quantity that depends on the discrete solution and data and is of primary importance in computational practice because of its ability to provide computable information about errors and drive the so-called adaptive finite element methods (AFEMs). The a posteriori error analysis for linear second--order elliptic boundary value problems and the construction of AFEMs and their convergence and optimal complexity have attained a mature understanding \cite{MR1885308,MR3076038,MR3059294}. To the best of our knowledge, the first work that provided an advance regarding a posteriori error estimates for linear and distributed optimal control problems is \cite{MR1887737}: the devised residual--type a posteriori error estimator is proven to yield an upper bound for the error \cite[Theorem 3.1]{MR1887737}. These results were later improved
in \cite{MR2434065} where the authors explore a slight modification of the estimator of \cite{MR1887737} and prove upper and lower error bounds which include oscillation terms \cite[Theorems 5.1 and 6.1]{MR2434065}. Recently, these ideas were unified in \cite{MR3212590}. In contrast to these advances the a posteriori error analysis for nonlinear optimal control problems is not as developed. To the best of our knowledge, the first work that provides an advance on this matter is \cite{MR2024491}. In this work the authors derive a posteriori error estimates for such a class of problems on Lipschitz domains and for nonlinear terms $a$ which are such that
\[
\partial a / \partial y(\cdot,y)\in W^{1,\infty}(-R,R), R > 0, 
\quad
a(\cdot,y)\in L^2(\Omega), y \in H^1(\Omega), 
\quad 
\partial a / \partial y \geq 0.
\] 
Under the assumption that estimate \eqref{eq:inequality_control_tilde} holds, the authors devise an error estimator that yields an upper bound for the corresponding error on the $H^1(\Omega)\times H^1(\Omega)\times L^2(\Omega)$--norm \cite[Theorem 3.1]{MR2024491}. We notice that no efficiency analysis is provided in \cite{MR2024491}. We conclude this paragraph by mentioning the approach introduced in \cite{MR1780911} for estimating the error in terms of the cost functional for linear/semilinear optimal control problems. This approach was later extended to problems with control constraints in \cite{MR2421327,MR2373479} and state constraints in \cite{MR2556843}.

In this work, we propose an a posteriori error estimator for the optimal control problem \eqref{def:cost_func}--\eqref{def:box_constraints} that can be decomposed as the sum of three contributions: one related to the discretization of the state equation, one associated to the discretization of the adjoint equation, and another one that accounts for the discretization of the control variable. This error estimator is different to the one provided in \cite{MR2024491}. On two and three dimensional Lipschitz polytopes, we obtain global reliability and efficiency properties. On the basis of the devised error estimator, we also design a simple adaptive strategy that exhibits, for the examples that we present, optimal experimental rates of convergence for all the optimal variables. We also provide numerical evidence that support the claim that our estimator outperforms the one in \cite{MR2024491}; see section \ref{sec:numer_results}. A few extensions of our theory are briefly discussed: piecewise linear approximation of the optimal control and sparse PDE-constrained optimization.

The outline of this paper is as follows. In section \ref{sec:nota_and_assum} we set notation and assumptions employed in the rest of the work. In section \ref{sec:semi_prob} we review preliminary results about solutions to \eqref{def:state_eq}. Basic results for the optimal control problem \eqref{def:cost_func}--\eqref{def:box_constraints} as well as first and second order optimality conditions are reviewed in section \ref{sec:pt_ocp}. The core of our work are sections \ref{sec:reliability} and \ref{sec:eff_estimator}, where we design an a posteriori error estimator for a suitable finite element discretization and show, in sections \ref{sec:reliability} and  \ref{sec:eff_estimator}, its reliability and efficiency, respectively. In section \ref{sec:extensions} we present a few extensions of the theory developed in previous sections. Finally, numerical examples presented in section \ref{sec:numer_results} illustrate the theory and reveal a competitive performance of the devised error estimator.


\section{Notation and assumptions}\label{sec:nota_and_assum}

Let us set notation and describe the setting we shall operate with.
\subsection{Notation}\label{sec:notation}

Throughout this work $d \in \{2,3\}$ and $\Omega\subset\mathbb{R}^d$ is an open and bounded polytopal domain with Lipschitz boundary $\partial\Omega$. Notice that we do not assume that $\Omega$ is convex. If $\mathscr{X}$ and $\mathscr{Y}$ are Banach function spaces, $\mathscr{X}\hookrightarrow \mathscr{Y}$ means that $\mathscr{X}$ is continuously embedded in $\mathscr{Y}$. We denote by $\mathscr{X}'$ and $\| \cdot \|_{\mathscr{X}}$ the dual and norm, respectively, of $\mathscr{X}$. The relation $\mathfrak{a} \lesssim \mathfrak{b}$ indicates that $\mathfrak{a} \leq C \mathfrak{b}$, with a positive constant that depends neither on $\mathfrak{a}$, $\mathfrak{b}$ nor the discretization parameter. The value of $C$ might change at each occurrence.

\subsection{Assumptions}\label{sec:assumption}

We assume that the nonlinear function $a$ involved in the monotone, semilinear, and elliptic PDE \eqref{def:state_eq} is such that:

\begin{enumerate}[label=(A.\arabic*)]
\item \label{A1} $a:\Omega\times \mathbb{R}\rightarrow  \mathbb{R}$ is a Carath\'eodory function of class $C^2$ with respect to the second variable and $a(\cdot,0)\in L^{2}(\Omega)$.

\item \label{A2} $\frac{\partial a}{\partial y}(x,y)\geq 0$ for a.e. $x\in\Omega$ and for all $y \in \mathbb{R}$.

\item \label{A3} For all $M>0$, there exists a positive constant $C_{M}$ such that 
\begin{equation*}
\sum_{i=1}^{2}\left|\frac{\partial^{i} a}{\partial y^{i} }(x,y)\right|\leq C_{M},
\end{equation*}
for a.e. $x\in \Omega$ and $|y|\leq M$.
\end{enumerate}

The following properties follow immediately from the previous assumptions. First, $a$ is monotone increasing in $y$ for a.e. $x \in \Omega$. In particular, for $v,w \in L^2(\Omega)$, we have
\begin{equation}\label{eq:monotone_operator}
(a(\cdot,v)-a(\cdot,w),v-w)_{L^2(\Omega)}\geq 0.
\end{equation}
Second, $a$ and $\frac{\partial a}{\partial y }$ are locally Lipschitz with respect to $y$, i.e., there exist positive constants $C_M$ and $L_M$ such that 
\begin{equation}\label{eq:Lipschitz_cont}
|a(x,v)-a(x,w)| \leq C_{M} |v-w|,
\qquad
\left|\frac{\partial a}{\partial y }(x,v)-\frac{\partial a}{\partial y }(x,w)\right|\leq L_{M} |v-w|,
\end{equation}
for a.e $x \in \Omega$ and $v,w \in \mathbb{R}$ such that $|v|,|w|\leq M$.


\section{Semilinear problem}\label{sec:semi_prob}
In this section, we review some of the main results related to the existence and uniqueness of solutions for problem \eqref{def:state_eq}. We also review a posteriori error estimates for a particular finite element setting.
\subsection{Weak formulation}\label{sec:weak_semilinear}
Given $f\in L^q(\Omega)$ with $q>d/2$, we consider the following weak problem: Find $y\in H_0^1(\Omega)$ such that
\begin{equation}\label{eq:weak_pde}
(\nabla y, \nabla v)_{L^2(\Omega)} + (a(\cdot,y),v)_{L^2(\Omega)} = (f,v)_{L^2(\Omega)}\quad \forall \: v \in H_0^1(\Omega).
\end{equation}

Invoking the main theorem on monotone operators \cite[Theorem 26.A]{MR1033498}, \cite[Theorem 2.18]{MR3014456} and an argument due to Stampacchia \cite{MR192177}, \cite[Theorem B.2]{MR1786735}, the following result can be derived; see \cite[Section 2]{MR3586845} and \cite[Theorem 4.8]{Troltzsch}.

\begin{theorem}[well--posedness]\label{thm:well_posedness_semilinear}
Let $f \in L^{q}(\Omega)$ with $q>d/2$. Let $a = a(x,y): \Omega \times \mathbb{R} \rightarrow \mathbb{R}$ be a Carath\'eodory function that is monotone increasing in $y$. If $a(\cdot,0) \in L^{q}(\Omega)$, with $q>d/2$, then, problem \eqref{eq:weak_pde} has a unique solution $y\in H_0^1(\Omega)\cap L^\infty(\Omega)$. In addition, we have the estimate
\begin{equation*}
\| \nabla y\|_{L^{2}(\Omega)} +\| y \|_{L^{\infty}(\Omega)}\lesssim \|f-a(\cdot,0)\|_{L^{q}(\Omega)},
\end{equation*}
with a hidden constant that is independent of $a$ and $f$.
\end{theorem}


\subsection{Finite element discretization}\label{sec:discretization}
\label{subsec:fem}
 We denote by $\mathscr{T} = \{T\}$ a conforming partition of $\bar{\Omega}$ into simplices $T$ with size $h_T := \textrm{diam}(T)$. We denote by $\mathbb{T}$ the collection of conforming and shape regular meshes that are refinements of an initial mesh $\mathscr{T}_0$. We denote by $\Sides$ the set of internal $(d-1)$-dimensional interelement boundaries $S $ of $\T$. If $T \in \T$, we define $\Sides_T$ as the subset of $\Sides$ that contains the sides of $T$. For $S \in \Sides$, we set $\mathcal{N}_S = \{ T^+, T^-\}$, where $T^+, T^- \in \T$ are such that $S=T^+ \cap T^-$. In addition, we define the \emph{star} or \emph{patch} associated to the element $T \in \T$ as
\begin{equation}
\label{eq:patch}
\mathcal{N}_T= \{T' \in \T: \Sides_T \cap \Sides_{T'} \neq \emptyset \}.
\end{equation}

Given a mesh $\mathscr{T} \in \mathbb{T}$, we define the finite element space of continuous piecewise polynomials of degree one as
\begin{equation}\label{def:piecewise_linear_set}
\mathbb{V}(\T):=\{v_\T\in C(\bar{\Omega}): v_\T|_T\in \mathbb{P}_{1}(T) \ \forall \ T\in \T\}\cap H_0^1(\Omega).
\end{equation}

Given a discrete function $v_{\T} \in \mathbb{V}(\T)$, we define, for any internal side $S \in \Sides$, the jump or interelement residual  $\llbracket \nabla v_\mathscr{T}\cdot \boldsymbol{\nu} \rrbracket$ by
\[
\llbracket \nabla v_\mathscr{T}\cdot \boldsymbol{\nu} \rrbracket:= \boldsymbol{\nu}^{+} \cdot \nabla v_{\mathscr{T}}|^{}_{T^{+}} + \boldsymbol{\nu}^{-} \cdot \nabla v_{\mathscr{T}}|^{}_{T^{-}},
\]
where $\boldsymbol{\nu}^{+}, \boldsymbol{\nu}^{-}$ denote the unit normals to $S$ pointing towards $T^{+}$, $T^{-} \in \T$, respectively, which are such that $T^{+} \neq T^{-}$ and $\partial T^{+} \cap \partial T^{-} = S$.

We define the Galerkin approximation to problem \eqref{eq:weak_pde} by
\begin{equation}\label{eq:discrete_semi}
y_\T\in\mathbb{V}(\T):
\quad 
(\nabla y_\T,\nabla v_\T)_{L^2(\Omega)}+(a(\cdot,y_\T),v_\T)_{L^2(\Omega)}=(f,v_\T)_{L^2(\Omega)} 
\end{equation}
for all $v_\T\in \mathbb{V}(\T)$. Standard results yield the existence and uniqueness of a discrete solution $y_{\T}$.

\subsection{A posteriori error analysis for the semilinear equation}

Let $f \in L^2(\Omega)$ and let $a = a(x,y): \Omega \times \mathbb{R} \rightarrow \mathbb{R}$ be as in the statement of Theorem \ref{thm:well_posedness_semilinear} with $a(\cdot,0) \in L^2(\Omega)$. Let us assume, in addition, that $a$ is locally Lipschitz with respect to $y$. With the notation introduced in section \ref{subsec:fem} at hand, we define the following a posteriori local error indicators and error estimator
\begin{equation*}
\mathcal{E}_{T}^2:=h_T^2\|f-a(\cdot,y_\T)\|_{L^2(T)}^2+h_T\|\llbracket\nabla y_\T\cdot\boldsymbol{\nu}\rrbracket\|_{L^2(\partial T\setminus\partial\Omega)}^2,
\quad 
\mathcal{E}_{\T}^2:=\sum_{T\in\T}\mathcal{E}_{T}^2,
\end{equation*}
respectively. Notice that since $a$ is locally Lipschitz with respect to $y$ and $a(\cdot,0) \in L^2(\Omega)$, the residual term $h_T^2\|f-a(\cdot,y_\T)\|_{L^2(T)}^2$ is well--defined.

We present the following reliability result and, for the sake of readability, a proof.

\begin{theorem}[global reliability of $\mathcal{E}_{\T}$]\label{thm:global_reli_weak}
Let $f \in L^2(\Omega)$ and let $a = a(x,y): \Omega \times \mathbb{R} \rightarrow \mathbb{R}$ be as in the statement of Theorem \ref{thm:well_posedness_semilinear} with $a(\cdot,0) \in L^2(\Omega)$. Let us assume, in addition, that $a$ is locally Lipschitz with respect to $y$. Let $y\in H_0^1(\Omega)\cap L^\infty(\Omega)$ be the unique solution to problem \eqref{eq:weak_pde} and $y_\T\in\mathbb{V}(\T)$ its finite element approximation obtained as the solution to \eqref{eq:discrete_semi}. Then
\[
\| \nabla(y-y_\T) \|_{L^2(\Omega)} \lesssim \mathcal{E}_{\T}.
\]
The hidden constant is independent of $y$, $y_{\T}$, the size of the elements in the mesh $\T$, and $\#\T$.
\end{theorem}
\begin{proof}
Let $v \in H_0^1(\Omega)$. Since $y$ solves \eqref{eq:weak_pde}, we invoke Galerkin orthogonality and an elementwise integration by parts formula to arrive at
\begin{multline*}
(\nabla (y-y_\T),\nabla v)_{L^2(\Omega)}+
(a(\cdot,y)-a(\cdot,y_\T),v)_{L^2(\Omega)}
\\
= \sum_{T\in\T}
\int_{T}(f-a(x,y_\T) )(v-I_\T v ) \mathrm{d} x
+\sum_{S\in\Sides}\int_{S} \llbracket\nabla y_\T\cdot\boldsymbol{\nu}\rrbracket ( v-I_\T v) \mathrm{d}x,
\end{multline*}
where $I_\T:L^1(\Omega)\rightarrow \mathbb{V}(\T)$ denotes the Cl\'ement interpolation operator \cite{MR2373954,MR0520174}. Standard approximation properties for $I_\T$ and the finite overlapping property of stars allow us to conclude that
\begin{multline*}
(\nabla (y-y_\T),\nabla v)_{L^2(\Omega)}
+
(a(\cdot,y)-a(\cdot,y_\T),v)_{L^2(\Omega)}
\lesssim
\\
\left(\sum_{T\in\T}h_T^2\|f-a(\cdot,y_\T)\|_{L^2(T)}^2+h_T\|\llbracket\nabla y_\T\cdot\boldsymbol{\nu}\rrbracket\|_{L^2(\partial T\setminus \partial\Omega)}^2\right)^{\tfrac{1}{2}}\|\nabla v\|_{L^2(\Omega)}.
\end{multline*}
Set $v=y-y_\T\in H_0^1(\Omega)$ and invoke property \eqref{eq:monotone_operator} to conclude.
\end{proof}


\section{A semilinear optimal control problem}\label{sec:pt_ocp}
In this section, we precisely describe a weak version of the optimal control problem \eqref{def:cost_func}--\eqref{def:box_constraints}, which reads as follows:
\begin{equation}
\min \{ J(y,u): (y,u) \in H_0^1(\Omega) \times \mathbb{U}_{ad}\}
\label{def:weak_ocp}
\end{equation}
subject to the \emph{monotone, semilinear, and elliptic state equation}
\begin{equation}\label{eq:weak_st_eq}
(\nabla y,\nabla v)_{L^2(\Omega)}+(a(\cdot,y),v)_{L^2(\Omega)}=(u,v)_{L^2(\Omega)} \quad \forall\ v\in H_0^1(\Omega).
\end{equation}

The existence of an optimal state-control pair is as follows; see \cite[Theorem 6.16]{MR1756264}, \cite[Theorem 4.15]{Troltzsch}, and \cite[Theorem 6]{MR3586845}.

\begin{theorem}[existence of the solution]\label{thm:existence_control}
Suppose that assumptions \textnormal{\ref{A1}}--\textnormal{\ref{A3}} hold. Then, the optimal control problem \eqref{def:weak_ocp}--\eqref{eq:weak_st_eq} admits at least one solution $(\bar{y},\bar{u}) \in H_{0}^{1}(\Omega) \cap L^{\infty}(\Omega) \times \mathbb{U}_{ad}$.
\end{theorem}


\subsection{First order necessary optimality conditions}\label{sec:1st_order}

To formulate first order optimality conditions for problem \eqref{def:weak_ocp}--\eqref{eq:weak_st_eq}, we introduce the so-called control-to-state map $\mathcal{S}: L^{q}(\Omega) \rightarrow H_{0}^{1}(\Omega)\cap L^\infty(\Omega)$ ($q>d/2$), which, given a control $u \in L^{q}(\Omega) \subset\mathbb{U}_{ad}$, associates to it the unique state $y$ that solves \eqref{eq:weak_st_eq}. With this operator at hand, we introduce the reduced cost functional
\begin{equation*}
j(u) := J(\mathcal{S}u,u) = \dfrac{1}{2}\|\mathcal{S}u - y_{\Omega}\|_{L^{2}(\Omega)}^{2} + \dfrac{\nu}{2}\|u\|_{L^{2}(\Omega)}^{2}.
\end{equation*}

Suppose that  assumptions \textnormal{\ref{A1}}--\textnormal{\ref{A3}} hold, then the control-to-state map $\mathcal{S}$ is Fr\'echet differentiable from $L^{q}(\Omega)$ into $H_0^1(\Omega) \cap L^{\infty}(\Omega)$ ($q>d/2$) \cite[Theorem 4.17]{Troltzsch}. As a consequence, if $\bar{u}$ denotes a local optimal control for problem \eqref{def:weak_ocp}--\eqref{eq:weak_st_eq}, we thus have the variational inequality \cite[Lemma 4.18]{Troltzsch}
\begin{equation}
\label{eq:variational_inequality}
j'(\bar u) (u - \bar u) \geq 0 \quad \forall \: u \in \mathbb{U}_{ad}. 
\end{equation}
Here, $j'(\bar{u})$ denotes the Gate\^aux derivative of the functional $j$ in $\bar{u}$. To explore \eqref{eq:variational_inequality} we introduce the adjoint variable $p \in H_{0}^{1}(\Omega) \cap L^{\infty}(\Omega)$ as the unique solution to the \emph{adjoint equation}
\begin{equation}\label{eq:adj_eq}
(\nabla w,\nabla p)_{L^2(\Omega)} + \left(\tfrac{\partial a}{\partial y}(\cdot,y)p,w\right)_{L^2(\Omega)} = (y - y_{\Omega},w)_{L^2(\Omega)}\quad \forall\ w \in H_0^1(\Omega),
\end{equation}
where $y = \mathcal{S}u$ solves \eqref{eq:weak_st_eq}. Problem \eqref{eq:adj_eq} is well--posed.

With these ingredients at hand, we present the desired necessary optimality condition for our PDE--constrained optimization problem; see \cite[Theorem 4.20]{Troltzsch} and \cite[Theorem 3.2]{MR1937089}.

\begin{theorem}[first order necessary optimality conditions]
\label{optimality_cond}
Suppose that assumptions \textnormal{\ref{A1}--\ref{A3}} hold. Then, every local optimal control $\bar{u} \in \mathbb{U}_{ad}$ for problem \eqref{def:weak_ocp}--\eqref{eq:weak_st_eq} satisfies, together with the adjoint state $\bar p \in H_0^1(\Omega) \cap L^{\infty}(\Omega)$, the variational inequality
\begin{equation}\label{eq:var_ineq}
(\bar{p}+\nu \bar{u},u-\bar{u})_{L^2(\Omega)}\geq 0 \quad \forall\ u\in \mathbb{U}_{ad}.
\end{equation}
Here, $\bar p$ denotes the solution to \eqref{eq:adj_eq} with $y$ replaced by $\bar{y} = \mathcal{S} \bar u$.
\end{theorem}

We now introduce the projection operator $\Pi_{[\texttt{a},\texttt{b}]} : L^1(\Omega) \rightarrow  \mathbb{U}_{ad}$ as 
\begin{equation}\label{def:projector_pi}
\Pi_{[\texttt{a},\texttt{b}]}(v) := \min\{ \texttt{b}, \max\{ v, \texttt{a}\} \} \textrm{ a.e in } \Omega.
\end{equation} 
With this projector at hand, we present the following result: The local optimal control $\bar{u}$ satisfies \eqref{eq:var_ineq} if and only if 
\begin{equation}\label{eq:projection_control}
\bar{u}(x):=\Pi_{[\texttt{a},\texttt{b}]}(-\nu^{-1}\bar{p}(x)) \textrm{ a.e. } x \in \Omega.
\end{equation}
In particular, this formula implies that $\bar u \in H^1(\Omega) \cap L^{\infty}(\Omega)$; see \cite[Theorem A.1]{MR1786735}.

\subsection{Second order sufficient optimality condition}\label{sec:2nd_order}
We follow \cite{MR3586845,MR2902693} and present necessary and sufficient second order optimality conditions.

Let $\bar{u} \in \mathbb{U}_{ad}$ satisfy the first order optimality conditions \eqref{eq:weak_st_eq}, \eqref{eq:adj_eq}, and \eqref{eq:var_ineq}. Define $\bar{\mathfrak{p}} :=  \bar p + \nu \bar u$. In view of \eqref{eq:var_ineq}, it follows that
\begin{equation*}
\bar{\mathfrak{p}}(x) 
\begin{cases}
= 0  & \text{ a.e. } x \in \Omega \text{ if } \texttt{a}< \bar{u} < \texttt{b}, \\
\geq  0  & \text{ a.e. } x \in \Omega \text{ if }\bar{u}=\texttt{a}, \\
\leq  0  & \text{ a.e. } x \in \Omega \text{ if } \bar{u}=\texttt{b}.
\end{cases}
\end{equation*}

Define the \emph{cone of critical directions}
\begin{equation*}
C_{\bar{u}}:=\{v\in L^2(\Omega) \text{ satisfying } \eqref{eq:cone_def} \text{ and } v(x) = 0 \text{ if } \bar{\mathfrak{p}}(x) \neq 0\},
\end{equation*}
with
\begin{equation}
\label{eq:cone_def}
v(x)
\begin{cases}
\geq 0 & \text{ a.e. } x\in\Omega \text{ if } \bar{u}(x)=\texttt{a},\\
\leq 0 & \text{ a.e. } x\in\Omega \text{ if } \bar{u}(x)=\texttt{b}.
\end{cases}
\end{equation}

We are now in conditions to present second order necessary and sufficient optimality conditions; see \cite[Theorem 23]{MR3586845}.

\begin{theorem}[second order necessary and sufficient optimality conditions]\label{thm:second_order_cond}
Suppose that assumptions \textnormal{\ref{A1}--\ref{A3}} hold. If $\bar{u} \in \mathbb{U}_{ad}$ is local minimum
for problem \eqref{def:weak_ocp}--\eqref{eq:weak_st_eq}, then
\begin{equation*}
j''(\bar{u})v^2 \geq 0 \quad \forall\ v \in C_{\bar{u}}.
\end{equation*}
Conversely, if $(\bar{y},\bar{p},\bar{u}) \in H_0^1(\Omega)\times H_0^1(\Omega)\times \mathbb{U}_{ad}$ satisfies the first order optimality conditions \eqref{eq:weak_st_eq}, \eqref{eq:adj_eq}, and \eqref{eq:var_ineq}, and  
\begin{equation*}
j''(\bar{u})v^2 > 0 \quad \forall\ v \in C_{\bar{u}}\setminus \{0\},
\end{equation*}
then, there exist $\mu >0$ and $\varepsilon>0$ such that
\begin{equation*}
j(u)\geq j(\bar{u})+ \frac{\mu}{2}\|u-\bar{u}\|_{L^2(\Omega)}^2 \quad \forall\ u\in \mathbb{U}_{ad}\cap \bar{B}_{\varepsilon}(\bar{u}),
\end{equation*}
where $\bar{B}_{\varepsilon}(\bar{u})$ denotes the closed ball in $L^2(\Omega)$ with center at $\bar{u}$ and radius $\varepsilon$.
\end{theorem}

Define
\begin{equation}\label{def:critical_cone_tau}
C_{\bar{u}}^\tau:=\{v\in L^2(\Omega) \textnormal{ satisfying \eqref{eq:cone_def} and } v(x)=0 \textnormal{ if } |\bar{\mathfrak{p}}(x)|>\tau \}.
\end{equation}

The next result will be of importance for deriving a posteriori error estimates for the numerical discretizations of \eqref{def:weak_ocp}--\eqref{eq:weak_st_eq} that we will propose; see \cite[Theorem 25]{MR3586845}.

\begin{theorem}[equivalent optimality condition]\label{thm:equivalent_second_order}
Suppose that assumptions \textnormal{\ref{A1}--\ref{A3}} hold. If $\bar{u}\in \mathbb{U}_{ad}$ satisfies \eqref{eq:var_ineq} then, the following statements are equivalent:
\begin{equation}\label{eq:second_order_2_2}
j''(\bar{u})v^2 > 0 \quad \forall\ v \in C_{\bar{u}}\setminus \{0\},
\end{equation}
and
\begin{equation}\label{eq:second_order_equivalent}
\exists \mu, \tau >0: \quad j''(\bar{u})v^2 \geq \mu \|v\|_{L^2(\Omega)}^2 \quad \forall\ v \in C_{\bar{u}}^\tau.
\end{equation}
\end{theorem}

We close this section with the following estimate: Let $u, h, v \in L^\infty(\Omega)$ and $\mathsf{M} >0$ be such that $\max\{\| u\|_{L^\infty(\Omega)},\|h\|_{L^\infty(\Omega)}\} \leq \mathsf{M}$. Then, there exists $C_{\mathsf{M}} > 0$ such that \cite[Lemma 4.26]{Troltzsch}
\begin{equation}\label{eq:estimate_second_der}
|j''(u + h)v^2-j''(u)v^2|\leq C_{\mathsf{M}} \|h\|_{L^\infty(\Omega)}\|v\|_{L^2(\Omega)}^2.
\end{equation}


\subsection{Finite element discretization}\label{sec:fem_ocp}
We present a finite element discretization of our optimal control problem. The approximation of the optimal control $\bar u$ is done by piecewise constant functions: $\bar u_{\T} \in \mathbb{U}_{ad}(\mathscr{T})$, where
\begin{equation*}
\mathbb{U}_{ad}(\mathscr{T}):=\mathbb{U}(\mathscr{T})\cap \mathbb{U}_{ad},
\quad
\mathbb{U}(\mathscr{T}):=\{ u_\T\in L^\infty(\Omega): u_\T|^{}_T\in \mathbb{P}_0(T) \ \forall \ T\in \T\}.
\end{equation*}
The optimal state and adjoint state are discretized using the finite element space $\mathbb{V}(\T)$ defined in \eqref{def:piecewise_linear_set}. In this setting, the discrete counterpart of \eqref{def:weak_ocp}--\eqref{eq:weak_st_eq}
reads as follows: Find $\min J(y_{\T},u_{\T})$ subject to the discrete state equation
\begin{equation}
\label{eq:discrete_state_equation}
y^{}_\mathscr{T} \in \mathbb{V}(\mathscr{T}): \quad
(\nabla y^{}_\mathscr{T},\nabla v^{}_\mathscr{T})_{L^2(\Omega)}+(a(\cdot,y_\T),v_\T)_{L^2(\Omega)}  =  (u^{}_\mathscr{T},v^{}_\mathscr{T})^{}_{L^2(\Omega)}
\end{equation}
for all $v^{}_\mathscr{T} \in \mathbb{V}(\mathscr{T})$ and the discrete constraints $u^{}_{\mathscr{T}} \in \mathbb{U}_{ad}(\T)$. This problem admits at least a solution \cite[section 7]{MR3586845}. In addition, if $\bar u^{}_\mathscr{T}$ denotes a local solution, then
\begin{equation*}
\label{eq:discrete_opt_system}
 (\bar{p}^{}_\mathscr{T}+\nu\bar{u}^{}_\mathscr{T},u^{}_\mathscr{T}-\bar{u}^{}_\mathscr{T})^{}_{L^2(\Omega)}  \geq  0 \quad \forall\ u^{}_\mathscr{T} \in \mathbb{U}_{ad}(\T),
\end{equation*}
where $\bar{p}^{}_\mathscr{T} \in \mathbb{V}(\T)$ is such that
\begin{equation}
\label{eq:discrete_adjoint_equation}
(\nabla w^{}_\mathscr{T},\nabla \bar{p}^{}_\mathscr{T})_{L^2(\Omega)}+\left(\tfrac{\partial a}{\partial y}(\cdot,\bar{y}_\T)\bar{p}_\T,w_\T\right)_{L^2(\Omega)}    =   (\bar{y}^{}_\mathscr{T}-y^{}_{\Omega},w^{}_\mathscr{T})^{}_{L^2(\Omega)} 
\end{equation}
for all $w^{}_\mathscr{T} \in \mathbb{V}(\T)$.

Define, on the basis of the projection operator \eqref{def:projector_pi}, the auxiliary variable
\begin{equation}\label{def:control_tilde}
\tilde{u}:=\Pi_{[\texttt{a},\texttt{b}]}(-\nu^{-1}\bar{p}_\T).
\end{equation}
Notice that $\tilde{u}\in\mathbb{U}_{ad}$ satisfies the following variational inequality \cite[Lemma 2.26]{Troltzsch}
\begin{equation}\label{eq:var_ineq_u_tilde}
(\bar{p}_\T+\nu\tilde{u},u-\tilde{u})_{L^2(\Omega)}\geq 0 \quad \forall\ u\in \mathbb{U}_{ad}.
\end{equation}

The following result is instrumental for our a posteriori error analysis.

\begin{theorem}[auxiliary estimate]\label{thm:error_bound_control_tilde}
Suppose that assumptions \textnormal{\ref{A1}--\ref{A3}} hold. Let $\bar u \in \mathbb{U}_{ad}$ be a local solution to \eqref{def:weak_ocp}--\eqref{eq:weak_st_eq} satisfying the sufficient second order optimality condition \eqref{eq:second_order_2_2}, or equivalently \eqref{eq:second_order_equivalent}. Let $\mathsf{M}$ be a positive constant such that 
$\max \{ \| \bar{u}+\theta_{\T}(\tilde{u}-\bar{u})\|_{L^{\infty}(\Omega)}, \|\tilde{u} - \bar u \|_{L^{\infty}(\Omega)}\} \leq \mathsf{M}$ with $\theta_{\T} \in (0,1)$.  Let $\bar u_{\T}$ be a local minimum of the discrete optimal control problem and $\T$ be a mesh such that 
\begin{equation}\label{eq:assumption_mesh}
\|\bar{p}-\bar{p}_\T\|_{L^\infty(\Omega)} \leq \min\{\nu\mu(2C_{\mathsf{M}})^{-1},\tau/2\}.
\end{equation}
Then
$\tilde{u}-\bar{u}\in C_{\bar{u}}^\tau$ and
\begin{equation}\label{eq:inequality_control_tilde}
\frac{\mu}{2}\|\bar{u}-\tilde{u}\|_{L^2(\Omega)}^2
\leq
(j'(\tilde{u})-j'(\bar{u}))(\tilde{u}-\bar{u}).
\end{equation}
The constant $C_{\mathsf{M}}$ is given by \eqref{eq:estimate_second_der} while the auxiliary variable $\tilde{u}$ is defined in \eqref{def:control_tilde}.
\end{theorem}
\begin{proof}
We proceed in two steps:

\emph{Step 1}. Let us assume, for the moment, that $\tilde{u}-\bar{u}\in C_{\bar{u}}^\tau$, with $C_{\bar{u}}^\tau$ defined in \eqref{def:critical_cone_tau}. Since $\bar u$ satisfies the sufficient second order optimality condition \eqref{eq:second_order_equivalent}, we are thus allow to set $v=\tilde{u}-\bar{u}$ there. This yields
\begin{equation}\label{eq:diff_1}
\mu\|\tilde{u}-\bar{u}\|_{L^2(\Omega)}^2 \leq j''(\bar{u})(\tilde{u}-\bar{u})^2.
\end{equation}
On the other hand, in view of the mean value theorem, we obtain, for some $\theta_{\T} \in (0,1)$, 
\begin{equation*}
\label{eq:mean_value_identity}
(j'(\tilde{u})-j'(\bar{u}))(\tilde{u}-\bar{u})=j''(\zeta)(\tilde{u}-\bar{u})^2,
\end{equation*}
with $\zeta=\bar{u}+\theta_{\T}(\tilde{u}-\bar{u})$. Thus, in view of \eqref{eq:diff_1}, we arrive at
\begin{align}\label{eq:ineq_u_tilde_bar}
\mu\|\tilde{u}-\bar{u}\|_{L^2(\Omega)}^2 
& \leq (j'(\tilde{u})-j'(\bar{u}))(\tilde{u}-\bar{u}) + (j''(\bar{u})-j''(\zeta))(\tilde{u}-\bar{u})^2.
\end{align}
Since $\mathsf{M} > 0$ is such that 
$\max \{ \| \bar{u}+\theta_{\T}(\tilde{u}-\bar{u})\|_{L^{\infty}(\Omega)}, \|\tilde{u} - \bar u \|_{L^{\infty}(\Omega)}\} \leq \mathsf{M}$  and $j$ is of class $C^2$ in $L^2(\Omega)$, we can thus apply \eqref{eq:estimate_second_der} to derive
\begin{equation*}
(j''(\bar{u})-j''(\zeta))(\tilde{u}-\bar{u})^2 
\leq C_{\mathsf{M}} \|\tilde{u}-\bar{u}\|_{L^\infty(\Omega)}\|\tilde{u}-\bar{u}\|_{L^2(\Omega)}^2,
\end{equation*}
where we have also used that $\theta_{\T} \in (0,1)$. Invoke \eqref{eq:projection_control} and \eqref{def:control_tilde}, the Lipschitz property of the projection operator $\Pi_{[\texttt{a},\texttt{b}]}$, defined in \eqref{def:projector_pi}, and assumption \eqref{eq:assumption_mesh}, to arrive at
\begin{equation*}
(j''(\bar{u})-j''(\zeta))(\tilde{u}-\bar{u})^2 
\leq C_{\mathsf{M}}{\nu}^{-1} \|\bar{p} - \bar{p}_\T \|_{L^\infty(\Omega)}\|\tilde{u}-\bar{u}\|_{L^2(\Omega)}^2
\leq \frac{\mu}{2}\|\tilde{u}-\bar{u}\|_{L^2(\Omega)}^2.
\end{equation*}
Replacing this inequality into \eqref{eq:ineq_u_tilde_bar} allows us to conclude the desired inequality \eqref{eq:inequality_control_tilde}.

\emph{Step 2}. We now prove that $\tilde{u}-\bar{u}\in C_{\bar{u}}^\tau$. Since $\tilde{u}\in \mathbb{U}_{ad}$, we can immediately conclude that $\tilde{u}-\bar{u}\geq 0$ if $\bar{u}=\texttt{a}$ and that $\tilde{u}-\bar{u}\leq 0$ if $\bar{u}=\texttt{b}$. These arguments reveal that $v = \tilde{u}-\bar{u}$ satisfies \eqref{eq:cone_def}. It thus suffices to verify the remaining condition in \eqref{def:critical_cone_tau}. To accomplish this task, we first use the triangle inequality and invoke the Lipschitz property of $\Pi_{[\texttt{a},\texttt{b}]}$, in conjunction with \eqref{eq:assumption_mesh}, to obtain
\begin{equation}\label{eq:difference_d}
\|\bar{p}+\nu \bar{u}-(\bar{p}_\T+\nu\tilde{u})\|_{L^\infty(\Omega)}\leq 2\|\bar{p}-\bar{p}_\T\|_{L^\infty(\Omega)}<\tau.
\end{equation}

Now, let $\xi\in \Omega$ be such that $\bar{\mathfrak{p}}(\xi) = (\bar{p}+\nu \bar{u})(\xi)>\tau$. Since $\tau > 0$, this implies that
$
\bar{u}(\xi)> - \nu^{-1}\bar{p}(\xi).
$
Therefore, from the projection formula \eqref{eq:projection_control}, we conclude that $\bar{u}(\xi)=\texttt{a}$. On the other hand, since $\xi\in \Omega$ is such that $(\bar{p}+\nu \bar{u})(\xi)>\tau$, from \eqref{eq:difference_d} we can conclude that
\[
(\bar{p}_\T+\nu\tilde{u})(\xi)=\bar{p}_\T(\xi)+\nu\tilde{u}(\xi)>0,\]
and thus that
$
\tilde{u}(\xi)> -\nu^{-1}\bar{p}_\T(\xi).
$
This, on the basis of the definition of the auxiliary variable $\tilde{u}$, given in \eqref{def:control_tilde}, yields that $\tilde{u}(\xi)=\texttt{a}$. Consequently, $\bar{u}(\xi)=\tilde{u}(\xi)=\texttt{a}$, and thus $(\tilde{u}-\bar{u})(\xi)=0$. Similar arguments allow us to conclude that, if $\bar{\mathfrak{p}}(\xi) =(\bar{p}+\nu \bar{u})(\xi)<-\tau$, then $(\tilde{u}-\bar{u})(\xi)=0$. This concludes the proof.
\end{proof}

\section{A posteriori error analysis: Reliability estimates}
\label{sec:reliability}

In this section, we devise and analyze an a posteriori error estimator for the discretization \eqref{eq:discrete_state_equation}--\eqref{eq:discrete_adjoint_equation} of the optimal control problem \eqref{def:weak_ocp}--\eqref{eq:weak_st_eq}.

To simplify the exposition of the material, we define, for $(v,w,z) \in H_0^1(\Omega)\times H_0^1(\Omega)\times L^2(\Omega)$, the norm 
\begin{equation}\label{def:total_error}
\VERT{(v,w,z)\VERT}_{\Omega} := \| \nabla v \|_{L^{2}(\Omega)} + \| \nabla w \|_{L^{2}(\Omega)} + \| z \|_{L^{2}(\Omega)}. 
\end{equation}

The goal of this section is to obtain an upper bound for the error in the norm $\VERT{ \cdot \VERT}_{\Omega}$. This will be obtained on the basis of estimates on the error between the solution to the discretization \eqref{eq:discrete_state_equation}--\eqref{eq:discrete_adjoint_equation} and auxiliary variables that we define in what follows. Let $\hat y \in H_{0}^{1}(\Omega)$ be the solution to
\begin{equation}\label{eq:hat_functions}
(\nabla\hat{y},\nabla v)_{L^2(\Omega)}+(a(\cdot,\hat{y}),v)_{L^2(\Omega)} =  (\bar{u}^{}_\mathscr{T},v)^{}_{L^2(\Omega)} \quad \forall\ v \in H_0^1(\Omega).
\end{equation}
Define
\begin{equation}\label{eq:error_estimator_state_eq}
\mathcal{E}_{st,T}^2:=h_T^2\|\bar{u}_\T-a(\cdot,\bar{y}_\T)\|_{L^2(T)}^2+h_T\|\llbracket\nabla \bar{y}_\T\cdot\boldsymbol{\nu}\rrbracket\|_{L^2(\partial T\setminus\partial\Omega)}^2,
\ \
\mathcal{E}_{st}^2:=\sum_{T\in\T}\mathcal{E}_{st,T}^2.
\end{equation}
An application of Theorem \ref{thm:global_reli_weak} immediately yields the a posteriori error bound
\begin{equation}\label{eq:estimate_state_hat_discrete}
\|\nabla (\hat{y}-\bar{y}_\T)\|_{L^2(\Omega)}
\lesssim \mathcal{E}_{st}.
\end{equation}

Let $\hat p \in H_{0}^{1}(\Omega)$ be the solution to
\begin{equation}\label{eq:hat_functions_p}
(\nabla w,\nabla \hat{p})_{L^2(\Omega)}+\left(\tfrac{\partial a}{\partial y}(\cdot,\bar{y}_\T)\hat{p},w\right)_{L^2(\Omega)} =  (\bar{y}_\T-y^{}_{\Omega},w)^{}_{L^2(\Omega)} \quad \forall\ w\in H_0^1(\Omega).
\end{equation}
Define, for $T\in\T$, the local error indicators
\begin{equation}\label{eq:indicators_adjoint_eq}
\mathcal{E}_{ad,T}^2:=h_T^2\|\bar{y}_\T-y_\Omega-\tfrac{\partial a}{\partial y}(\cdot,\bar{y}_\T)\bar{p}_\T\|_{L^2(T)}^2+h_T\|\llbracket\nabla \bar{p}_\T\cdot\boldsymbol{\nu}\rrbracket\|_{L^2(\partial T\setminus\partial\Omega)}^2,
\end{equation}
and the a posteriori error estimator
\begin{equation}\label{eq:error_estimator_adjoint_eq}
\mathcal{E}_{ad}:=\left(\sum_{T\in\T}\mathcal{E}_{ad,T}^2\right)^{\frac{1}{2}}.
\end{equation}

The following result yields an upper bound for the error $\|\nabla(\hat{p}-\bar{p}_\T)\|_{L^2(\Omega)}$ in terms of the computable quantity $\mathcal{E}_{ad}$.

\begin{lemma}[estimate for $\hat{p}-\bar{p}_\T$]\label{lemma:estimate_aux_variables}
Suppose that assumptions \textnormal{\ref{A1}}--\textnormal{\ref{A3}} hold. Let $\bar u \in \mathbb{U}_{ad}$ be a local solution to \eqref{def:weak_ocp}--\eqref{eq:weak_st_eq}. Let $\bar u_{\T}$ be a local minimum of the discretization \eqref{eq:discrete_state_equation}--\eqref{eq:discrete_adjoint_equation} with $\bar{y}_{\T}$ and $\bar{p}_\T$ being the associated state and adjoint state, respectively. Then, the auxiliary variable $\hat{p}$, defined in \eqref{eq:hat_functions_p}, satisfies 
\begin{equation}\label{eq:adjoint_hat_estimate}
\|\nabla(\hat{p}-\bar{p}_\T)\|_{L^2(\Omega)}\lesssim \mathcal{E}_{ad}.
\end{equation}
The hidden constant is independent of the solution to \eqref{def:weak_ocp}--\eqref{eq:weak_st_eq}, its finite element approximation, the size of the elements in the mesh $\T$, and $\#\T$.
\end{lemma}
\begin{proof}
We proceed as in the proof of Theorem \ref{thm:global_reli_weak}. Let $w \in H_0^1(\Omega)$. Since $\hat p$ solves \eqref{eq:hat_functions_p}, we invoke Galerkin orthogonality and an elementwise integration by parts formula to conclude that
\begin{multline*}
(\nabla w,\nabla (\hat{p}-\bar{p}_\T))_{L^2(\Omega)}+\left(\tfrac{\partial a}{\partial y}(\cdot,\bar{y}_\T)(\hat{p}-\bar{p}_\T),w\right)_{L^2(\Omega)}
 \\
 =
 \sum_{T\in\T}\left(\bar{y}_\T-y^{}_{\Omega}-\tfrac{\partial a}{\partial y}(\cdot,\bar{y}_\T)\bar{p}_\T,w-I_\T w\right)^{}_{L^2(T)}+\sum_{S\in\Sides}(\llbracket \nabla \bar{p}_\T\cdot\boldsymbol{\nu}\rrbracket,w-I_\T w)_{L^2(S)}.
\end{multline*}
Standard approximation properties for $I_\T$ and the finite overlapping property of stars allow us to conclude that
\begin{multline*}
(\nabla w,\nabla (\hat{p}-\bar{p}_\T))_{L^2(\Omega)}+\left(\tfrac{\partial a}{\partial y}(\cdot,\bar{y}_\T)(\hat{p}-\bar{p}_\T),w\right)_{L^2(\Omega)}\lesssim\\
\left(\sum_{T\in\T}h_T^2\|\bar{y}_\T-y_\Omega-\tfrac{\partial a}{\partial y}(\cdot,\bar{y}_\T)\bar{p}_\T\|_{L^2(T)}^2+h_T\|\llbracket\nabla \bar{p}_\T\cdot\boldsymbol{\nu}\rrbracket\|_{L^2(\partial T\setminus\partial\Omega)}^2\right)^{\frac{1}{2}}\!\|\nabla w\|_{L^2(\Omega)}.
\end{multline*}
Set $w=\hat{p}-\bar{p}_\T$ and invoke assumption \ref{A2} to conclude. 
\end{proof}

We define a global error estimator associated to the discretization of the optimal control variable as follows:
\begin{equation}\label{eq:error_estimator_control_var}
\mathcal{E}_{ct,T}^2:=\|\tilde{u}-\bar{u}_\T\|_{L^2(T)}^2,\quad 
\mathcal{E}_{ct}:=\left(\sum_{T\in\T}\mathcal{E}_{ct,T}^2\right)^\frac{1}{2}.
\end{equation}
We recall that the auxiliary variable $\tilde{u}$ is defined as in \eqref{def:control_tilde}.

The following two auxiliary variables, related to $\tilde{u}\in \mathbb{U}_{ad}\subset L^2(\Omega)$, will be of particular importance for our analysis. The variable $\tilde y \in H_0^1(\Omega)$, which solves
\begin{equation*}
(\nabla\tilde{y},\nabla v)_{L^2(\Omega)}+(a(\cdot,\tilde{y}),v)_{L^2(\Omega)} = (\tilde{u},v)^{}_{L^2(\Omega)} \quad \forall\ v \in H_0^1(\Omega),
\end{equation*}
and $\tilde p \in H_0^1(\Omega)$, which is defined as the solution to
\begin{equation*}
(\nabla w,\nabla \tilde{p})_{L^2(\Omega)}+\left(\tfrac{\partial a}{\partial y}(\cdot,\tilde{y})\tilde{p},w\right)_{L^2(\Omega)} =  (\tilde{y}-y^{}_{\Omega},w)^{}_{L^2(\Omega)} \quad 
\forall \ w\in H_0^1(\Omega).
\end{equation*}

After all these definitions and preparations, we define an a posteriori error estimator for the optimal control problem \eqref{def:weak_ocp}--\eqref{eq:weak_st_eq}, which can be decomposed as the sum of three contributions:
\begin{equation}\label{def:error_estimator_ocp}
\mathcal{E}_{ocp}^2:=\mathcal{E}_{st}^2 + \mathcal{E}_{ad}^2 + \mathcal{E}_{ct}^2.
\end{equation}
The estimators $\mathcal{E}_{st}$, $\mathcal{E}_{ad}$, and $\mathcal{E}_{ct}$, are defined as in \eqref{eq:error_estimator_state_eq}, \eqref{eq:error_estimator_adjoint_eq}, and \eqref{eq:error_estimator_control_var}, respectively.

We are now ready to state and prove the main result of this section.

\begin{theorem}[global reliability]\label{thm:control_bound}
Suppose that assumptions \textnormal{\ref{A1}}--\textnormal{\ref{A3}} hold. Let $\bar u \in \mathbb{U}_{ad}$ be a local solution to \eqref{def:weak_ocp}--\eqref{eq:weak_st_eq} satisfying the sufficient second order condition \eqref{eq:second_order_2_2}, or equivalently \eqref{eq:second_order_equivalent}. Let $\bar u_{\T}$ be a local minimum of the associated discrete optimal control problem with $\bar{y}_{\T}$ and $\bar{p}_\T$ being the corresponding state and adjoint state, respectively. Let $\T$ be a mesh such that \eqref{eq:assumption_mesh} holds, then
\begin{equation}\label{eq:global_rel}
\VERT{(\bar{y} - \bar{y}_{\T},\bar{p} - \bar{p}_{\T},\bar{u} - \bar{u}_{\T})\VERT}_{\Omega}
\lesssim 
\mathcal{E}_{ocp}.
\end{equation}
The hidden constant is independent of the continuous and discrete optimal variables, the size of the elements in the mesh $\T$, and $\#\T$.
\end{theorem}
\begin{proof}
We proceed in four steps.

\emph{Step 1}. The goal of this step is to control the term $\|\bar{u}-\bar{u}_\T\|_{L^2(\Omega)}$. We begin with a simple application of the triangle inequality and write
\begin{equation}\label{eq:triangle_ineq_control}
\|\bar{u}-\bar{u}_\T\|_{L^2(\Omega)}
\leq 
\|\bar{u}-\tilde{u}\|_{L^2(\Omega)} +\mathcal{E}_{ct},
\end{equation}
where $\tilde{u}:=\Pi_{[\texttt{a},\texttt{b}]}\left(-\nu^{-1}\bar{p}_\T\right)$ and $\mathcal{E}_{ct}$ is defined as in \eqref{eq:error_estimator_control_var}. Let us now bound the first term on the right hand side of \eqref{eq:triangle_ineq_control}. To accomplish this task, we set $u = \tilde u$ in \eqref{eq:var_ineq} and $u = \bar u$ in \eqref{eq:var_ineq_u_tilde} to obtain
\[
- j'(\bar u)(\tilde u - \bar u) = - (\bar p + \nu \bar u, \tilde u - \bar u)_{L^2(\Omega)} \leq 0, \qquad - (\bar p_{\T} + \nu \tilde u, \tilde u - \bar u)_{L^2(\Omega)} \geq 0.
\]
In light of these estimates, we invoke \eqref{eq:inequality_control_tilde} to obtain
\begin{align*}
\tfrac{\mu}{2} \|\bar{u}-\tilde{u}\|_{L^2(\Omega)}^2 &  \leq   j'(\tilde{u})(\tilde{u}-\bar{u})-j'(\bar{u})(\tilde{u}-\bar{u}) \leq  j'(\tilde{u})(\tilde{u}-\bar{u})
\\
& = (\tilde p + \nu \tilde u, \tilde{u}-\bar{u} )_{L^2(\Omega)} \leq (\tilde{p}-\bar{p}_\T,\tilde{u}-\bar{u})_{L^2(\Omega)}.
\end{align*}
Adding and subtracting the auxiliary variable $\hat{p}$, defined as the solution to \eqref{eq:hat_functions_p}, and utilizing basic inequalities we arrive at
\begin{equation}
\label{eq:auxiliar_estimate}
\|\bar{u}-\tilde{u}\|_{L^2(\Omega)}^2
\lesssim
(\|\tilde{p}-\hat{p}\|_{L^2(\Omega)}+\|\hat{p}-\bar{p}_\T\|_{L^2(\Omega)})\|\tilde{u}-\bar{u}\|_{L^2(\Omega)}.
\end{equation}
We now invoke a Poincar\'e inequality and the error estimate $\|\nabla(\hat{p}-\bar{p}_\T)\|_{L^2(\Omega)}\lesssim \mathcal{E}_{ad}$, which follows from \eqref{eq:adjoint_hat_estimate}, to obtain
\begin{equation}\label{eq:estimate_u_tilde_u}
\|\bar{u}-\tilde{u}\|_{L^2(\Omega)}
\lesssim
\|\nabla(\tilde{p}-\hat{p})\|_{L^2(\Omega)}+\mathcal{E}_{ad}.
\end{equation}

The rest of this step is dedicated to estimate the term $\|\nabla(\tilde{p}-\hat{p})\|_{L^2(\Omega)}$. To accomplish this task, we first notice that, for every $w \in H_0^1(\Omega)$, $\tilde{p}-\hat{p}\in H_0^1(\Omega)$ solves
\begin{equation*}
\label{eq:tildep-hatp}
(\nabla w,\nabla (\tilde{p}-\hat{p}))_{L^2(\Omega)}+\left(\tfrac{\partial a}{\partial y}(\cdot,\tilde{y})\tilde{p}-\tfrac{\partial a}{\partial y}(\cdot,\bar{y}_\T)\hat{p},w\right)_{L^2(\Omega)}=(\tilde{y}-\bar{y}_\T,w)_{L^2(\Omega)}.
\end{equation*}
Set $w=\tilde{p}-\hat{p}$ and invoke a generalized H\"older's inequality to obtain 
\begin{align*}
&\|\nabla(\tilde{p}-\hat{p})\|_{L^2(\Omega)}^2+\left(\tfrac{\partial a}{\partial y}(\cdot,\tilde{y})(\tilde{p}-\hat{p}),\tilde{p}-\hat{p}\right)_{L^2(\Omega)}
\\ 
\nonumber
&= (\tilde{y}-\bar{y}_\T,\tilde{p}-\hat{p})_{L^2(\Omega)}+\left(
\left[
\tfrac{\partial a}{\partial y}(\cdot,\bar{y}_\T)-\tfrac{\partial a}{\partial y}(\cdot,\tilde{y})
\right]
\hat{p},\tilde{p}-\hat{p}\right)_{L^2(\Omega)}
\\ 
\nonumber
&\leq \|\tilde{y}-\bar{y}_\T\|_{L^2(\Omega)} \|\tilde{p}-\hat{p}\|_{L^2(\Omega)}
+
\left\|\tfrac{\partial a}{\partial y}(\cdot,\bar{y}_\T)-\tfrac{\partial a}{\partial y}(\cdot,\tilde{y}) \right \|_{L^2(\Omega)}\|\hat{p}\|_{L^4(\Omega)}\|\tilde{p}-\hat{p}\|_{L^4(\Omega)}.
\end{align*}
Since $\bar y_{\T}, \tilde y \in L^{\infty}(\Omega)$ and $\frac{\partial a}{\partial y}$ is locally Lipschitz with respect to $y$, we obtain
\begin{equation*}
\|\nabla(\tilde{p}-\hat{p})\|_{L^2(\Omega)}^2 
\lesssim 
\\
\|\tilde{y}-\bar{y}_\T\|_{L^2(\Omega)} 
\left(
\|\tilde{p}-\hat{p}\|_{L^2(\Omega)}+\|\hat{p}\|_{L^4(\Omega)}\|\tilde{p}-\hat{p}\|_{L^4(\Omega)}
\right).
\end{equation*}
We thus use a Poincar\'e inequality and the embedding $H^1(\Omega)\hookrightarrow L^4(\Omega)$ to arrive at
\begin{equation}\label{eq:error_eq_tilde_hat_ad}
\|\nabla(\tilde{p}-\hat{p})\|_{L^2(\Omega)} \lesssim \\
\|\tilde{y}-\bar{y}_\T\|_{L^2(\Omega)} (1+\|\nabla \hat{p}\|_{L^2(\Omega)}).
\end{equation}
Stability estimates for the problems that $\hat{p}$ and $\bar{y}_\T$ solve yield the estimate
\[
\|\nabla\hat{p}\|_{L^2(\Omega)}\lesssim \|y^{}_\Omega\|_{L^2(\Omega)} + \|y^{}_{\T} \|_{L^2(\Omega)}
\lesssim 
 \|y^{}_\Omega\|_{L^2(\Omega)} 
 +
\rho|\Omega|^{\frac{1}{2}},
\]
where $\rho=\max\{|\texttt{a}|,|\texttt{b}|\}$. Replacing this estimate into \eqref{eq:error_eq_tilde_hat_ad}, and invoking, again, a Poincar\'e inequality, we obtain
\begin{equation}\label{eq:estimate_adjoint_tilde_hat}
\|\nabla(\tilde{p}-\hat{p})\|_{L^2(\Omega)}
\lesssim
\|\tilde{y}-\bar{y}_\T\|_{L^2(\Omega)}
\lesssim
\|\nabla(\tilde{y}-\bar{y}_\T)\|_{L^2(\Omega)},
\end{equation}
with a hidden constant that is independent of the continuous and discrete optimal variables, the size of the elements in the mesh $\T$, and $\#\T$ but depends on the continuous problem data.

We now turn our attention to bounding the term $\|\nabla(\tilde{y}-\bar{y}_\T)\|_{L^2(\Omega)}$ in \eqref{eq:estimate_adjoint_tilde_hat}. To accomplish this task, we invoke the auxiliary variable $\hat{y}$, defined as the solution to \eqref{eq:hat_functions}, and use the triangle inequality to obtain
\begin{equation}\label{eq:estimate_state_tilde_hat_2}
\|\nabla(\tilde{y}-\bar{y}_\T)\|_{L^2(\Omega)}
\lesssim
\|\nabla(\tilde{y}-\hat{y})\|_{L^2(\Omega)}+\mathcal{E}_{st},
\end{equation}
where we have also used the a posteriori error estimate \eqref{eq:estimate_state_hat_discrete}. It thus suffices to bound $\|\nabla(\tilde{y}-\hat{y})\|_{L^2(\Omega)}$. To do this, we first notice that $\tilde{y}-\hat{y}\in H_0^1(\Omega)$ solves the problem:
\begin{equation}
\label{eq:tildey-haty}
(\nabla (\tilde{y}-\hat{y}),\nabla v)_{L^2(\Omega)}+(a(\cdot,\tilde{y})-a(\cdot,\hat{y}),v)_{L^2(\Omega)}=(\tilde{u}-\bar{u}_\T,v)_{L^2(\Omega)} \quad \forall\ v\in H_0^1(\Omega).
\end{equation}
Set $v=\tilde{y}-\hat{y}$ and invoke the fact that $a$ is monotone increasing in $y$ \eqref{eq:monotone_operator} to arrive at
$
\|\nabla(\tilde{y}-\hat{y})\|_{L^2(\Omega)}
\lesssim
\|\tilde{u}-\bar{u}_\T\|_{L^2(\Omega)}=\mathcal{E}_{ct}.
$
Replacing this estimate into \eqref{eq:estimate_state_tilde_hat_2} and the obtained one into \eqref{eq:estimate_adjoint_tilde_hat} yield
\begin{equation}\label{eq:estimate_term_p_tilde_hat}
\|\nabla(\tilde{p}-\hat{p})\|_{L^2(\Omega)}\lesssim \mathcal{E}_{st}+\mathcal{E}_{ct}.
\end{equation}
On the basis of \eqref{eq:triangle_ineq_control}, \eqref{eq:estimate_u_tilde_u} and \eqref{eq:estimate_term_p_tilde_hat}, we conclude the a posteriori error estimate
\begin{equation}\label{eq:estimate_error_control}
\|\bar{u}-\bar{u}_\T\|_{L^2(\Omega)}\lesssim \mathcal{E}_{ad}+\mathcal{E}_{st}+\mathcal{E}_{ct}.
\end{equation}

\emph{Step 2}. The goal of this step is to bound $\|\nabla(\bar{y}-\bar{y}_\T)\|_{L^2(\Omega)}$. To accomplish this task, we invoke the auxiliary state $\hat{y}$, defined as the solution to \eqref{eq:hat_functions} and apply the triangle inequality. In fact, we have
\begin{equation}
\label{eq:state_ineq_triangular}
\|\nabla(\bar{y}-\bar{y}_\T)\|_{L^2(\Omega)}
\lesssim  \|\nabla(\bar{y}-\hat{y})\|_{L^2(\Omega)}+\mathcal{E}_{st},
\end{equation}
where we have also used the a posteriori error estimate \eqref{eq:estimate_state_hat_discrete}. It thus suffices to estimate $\|\nabla(\bar{y}-\hat{y})\|_{L^2(\Omega)}$. To achieve this goal, we invoke the state equation \eqref{eq:weak_st_eq}, with $u$ replaced by $\bar{u}$, problem \eqref{eq:hat_functions}, and the monotony of the nonlinear term $a$ \eqref{eq:monotone_operator}. These arguments reveal that
\begin{align*}
\|\nabla(\bar{y}-\hat{y})\|_{L^2(\Omega)}^2 &\leq  (\nabla (\bar{y}-\hat{y}),\nabla(\bar{y}-\hat{y}))_{L^2(\Omega)}+(a(\cdot,\bar{y})-a(\cdot,\hat{y}),\bar{y}-\hat{y})_{L^2(\Omega)}\\
&=(\bar{u}-\bar{u}_\T,\bar{y}-\hat{y})_{L^2(\Omega)} \lesssim \| \bar u - \bar{u}_\T \|_{L^2(\Omega)} \|\nabla(\bar{y}-\hat{y})\|_{L^2(\Omega)}.
\end{align*} 
Consequently,
$
\|\nabla(\bar{y}-\hat{y})\|_{L^2(\Omega)} \lesssim\| \bar{u} - \bar{u}_{\T} \|_{L^{2}(\Omega)}.
$
Replacing this estimate into \eqref{eq:state_ineq_triangular} and utilizing \eqref{eq:estimate_error_control} allow us to conclude that
\begin{equation}\label{eq:estimate_error_state}
\|\nabla(\bar{y}-\bar{y}_\T)\|_{L^2(\Omega)}\lesssim \mathcal{E}_{ad}+\mathcal{E}_{st}+\mathcal{E}_{ct}.
\end{equation}

\emph{Step 3}. We now bound the term $\|\nabla(\bar{p}-\bar{p}_\T)\|_{L^2(\Omega)}$. To accomplish this task, we add and subtract $\hat p$, defined as the solution to \eqref{eq:hat_functions_p}, and use, again, the triangle inequality to obtain that
\begin{equation}
\label{eq:adjoint_ineq_triangular}
\|\nabla(\bar{p}-\bar{p}_\T)\|_{L^2(\Omega)}
\lesssim \|\nabla(\bar{p}-\hat{p})\|_{L^2(\Omega)}+\mathcal{E}_{ad},
\end{equation}
where we have also used the a posteriori error estimate \eqref{eq:adjoint_hat_estimate}. It thus suffices to bound $\|\nabla(\bar{p}-\hat{p})\|_{L^2(\Omega)}$. Set $w=\bar{p}-\hat{p}$ in the weak problem that $\bar{p}-\hat{p}$ solves. This yields
\begin{align*}
\|\nabla(\bar{p}-\hat{p})\|_{L^2(\Omega)}^2
&+\left(\tfrac{\partial a}{\partial y}(\cdot,\bar{y})(\bar{p}-\hat{p}),\bar{p}-\hat{p}\right)_{L^2(\Omega)}\\ 
&=(\bar{y}-\bar{y}_\T,\bar{p}-\hat{p})_{L^2(\Omega)}+\left(\left[\tfrac{\partial a}{\partial y}(\cdot,\bar{y}_\T)-\tfrac{\partial a}{\partial y}(\cdot,\bar{y})\right]\hat{p},\bar{p}-\hat{p}\right)_{L^2(\Omega)}.
\end{align*}
This identity, in view of a generalized H\"older's inequality, the local Lipschitz property of $\frac{\partial a}{\partial y}$, with respect to the $y$ variable, and assumption \ref{A2}, allows us to arrive at
\begin{equation*}
\|\nabla(\bar{p}-\hat{p})\|_{L^2(\Omega)}^2 \lesssim \\
\|\bar{y}-\bar{y}_\T\|_{L^2(\Omega)} (\|\bar{p}-\hat{p}\|_{L^2(\Omega)}+\|\hat{p}\|_{L^4(\Omega)}\|\bar{p}-\hat{p}\|_{L^4(\Omega)}).
\end{equation*}
Using similar ideas to the ones that lead to \eqref{eq:error_eq_tilde_hat_ad} and \eqref{eq:estimate_adjoint_tilde_hat}, we can conclude that
\begin{equation}\label{eq:error_adjoint_eq_2}
\|\nabla(\bar{p}-\hat{p})\|_{L^2(\Omega)}
\lesssim
\|\nabla(\bar{y}-\bar{y}_\T)\|_{L^2(\Omega)}.
\end{equation}
Replacing \eqref{eq:estimate_error_state} into \eqref{eq:error_adjoint_eq_2}, and the obtained one into \eqref{eq:adjoint_ineq_triangular}, we obtain
\begin{equation}\label{eq:estimate_error_adjoint_st}
\|\nabla(\bar{p}-\bar{p}_\T)\|_{L^2(\Omega)}
\lesssim 
\mathcal{E}_{ad}+\mathcal{E}_{st}+\mathcal{E}_{ct}.
\end{equation}

\emph{Step 4}. Combining \eqref{eq:estimate_error_control}, \eqref{eq:estimate_error_state}, and \eqref{eq:estimate_error_adjoint_st} allows us to arrive at \eqref{eq:global_rel}. This concludes the proof.
\end{proof}


\section{A posteriori error analysis: Efficiency estimates}
\label{sec:eff_estimator}
In this section, we prove the local efficiency of the a posteriori error indicators $\mathcal{E}_{st,T}$ and $\mathcal{E}_{ad,T}$ and the global efficiency of  the a posteriori error estimator $\mathcal{E}_{ocp}$. To accomplish this task, we will proceed on the basis of standard residual estimation techniques \cite{MR1885308,MR3059294}.

Let us begin by introducing the following notation: for an edge/face or triangle/tetrahedron $G$, let $\mathcal{V}(G)$ be the set of vertices of $G$. With this notation at hand, we recall, for $T\in\mathscr{T}$ and $S\in\mathscr{S}$, the definition of the standard element and edge bubble functions \cite{MR1885308,MR3059294}
\begin{equation*}
\varphi^{}_{T}=
(d+1)^{(d+1)}\prod_{\textsc{v} \in \mathcal{V}(T)} \lambda^{}_{\textsc{v}},
\qquad
\varphi^{}_{S}=
d^{d} \prod_{\textsc{v} \in \mathcal{V}(S)}\lambda^{}_{\textsc{v}}|^{}_{T'},
\end{equation*}
respectively, where $T' \subset \mathcal{N}_{S}$ and $\lambda_\textsc{v}^{}$ are the barycentric coordinates of $T$. Recall that $\mathcal{N}_{S}$ denotes the patch composed of the two elements of $\mathscr{T}$ that share $S$.

The following identities are essential to perform an efficiency analysis. First, since $\bar y \in H_0^1(\Omega)$ solves \eqref{eq:weak_st_eq}, an elementwise integration by parts formula implies that
\begin{multline}\label{eq:error_eq_state}
(\nabla (\bar{y}-\bar{y}_\T),\nabla v)_{L^2(\Omega)}+(a(\cdot,\bar{y})-a(\cdot,\bar{y}_\T),v)_{L^2(\Omega)}=(\bar{u}-\bar{u}_\T,v)_{L^2(\Omega)} \\
+\sum_{T\in\T}(\bar{u}_\T-a(\cdot,\bar{y}_\T),v)_{L^2(T)}+\sum_{S\in\Sides}(\llbracket\nabla \bar{y}_\T\cdot \nu\rrbracket,v)_{L^2(S)}
\end{multline}
for all $v\in H_0^1(\Omega)$. Second, since $\bar p $ solves \eqref{eq:adj_eq}, similar arguments yield
\begin{multline}\label{eq:error_eq_adjoint}
(\nabla w,\nabla (\bar{p}-\bar{p}_\T))_{L^2(\Omega)}+\left(\tfrac{\partial a}{\partial y}(\cdot,\bar{y})(\bar{p}-\bar{p}_\T),w\right)_{L^2(\Omega)}=(\bar{y}-\bar{y}_\T,w)_{L^2(\Omega)}
\\
+\left(\left[\tfrac{\partial a}{\partial y}(\cdot,\bar{y}_\T)-\tfrac{\partial a}{\partial y}(\cdot,\bar{y})\right]\bar{p}_\T,w\right)_{L^2(\Omega)}+\sum_{S\in\Sides}(\llbracket\nabla \bar{p}_\T\cdot \nu\rrbracket,w)_{L^2(S)}
\\
+\sum_{T\in\T}\left(\left(\bar{y}_\T-\mathscr{P}_{\T} y_\Omega -\tfrac{\partial a}{\partial y}(\cdot,\bar{y}_\T)\bar{p}_\T,w\right)_{L^2(T)}+(\mathscr{P}_{\T} y_\Omega -y_\Omega,w)_{L^2(T)}\right)
\end{multline}
for all $w\in H_0^1(\Omega)$. Here, $\mathscr{P}_{\T}$ denotes the $L^2$-projection onto piecewise linear, over $\T$, functions.

We are ready to prove the local efficiency of the indicator $\mathcal{E}_{st}$ defined in \eqref{eq:error_estimator_state_eq}.

\begin{theorem}[local efficiency of $\mathcal{E}_{st}$]\label{thm:local_eff_state}
Suppose that assumptions \textnormal{\ref{A1}}--\textnormal{\ref{A3}} hold. Let $\bar u \in \mathbb{U}_{ad}$ be a local solution to \eqref{def:weak_ocp}--\eqref{eq:weak_st_eq}. Let $\bar u_{\T}$ be a local minimum of the discretization \eqref{eq:discrete_state_equation}--\eqref{eq:discrete_adjoint_equation} with $\bar{y}_{\T}$ and $\bar{p}_\T$ being the associated state and adjoint state, respectively. Then, for $T\in\T$, the local error indicator $\mathcal{E}_{st,T}$ satisfies 
\begin{equation}\label{eq:local_eff_st}
\mathcal{E}_{st,T}
\lesssim 
\|\nabla(\bar{y}-\bar{y}_\T)\|_{L^2(\mathcal{N}_{T})}+h_T\|\bar{y}-\bar{y}_\T\|_{L^2(\mathcal{N}_{T})}+h_T\|\bar{u}-\bar{u}_\T\|_{L^2(\mathcal{N}_{T})},
\end{equation}
where $\mathcal{N}_{T}$ is defined as in \eqref{eq:patch}. The hidden constant is independent of the continuous and discrete optimal variables, the size of the elements in the mesh $\T$, and $\#\T$.
\end{theorem}
\begin{proof}
We estimate each term in the definition of the local error indicator $\mathcal{E}_{st,T}$, given in \eqref{eq:error_estimator_state_eq}, separately.

\emph{Step 1}. 
Let $T \in \T$. We first bound the element term $h_T^2\|\bar{u}_\T-a(\cdot,\bar{y}_\T)\|_{L^2(T)}^2$. To accomplish this task, we invoke standard residual estimation techniques \cite{MR1885308,MR3059294}. Set  $v= \varphi_T (\bar{u}_\T-a(\cdot,\bar{y}_\T))$ in \eqref{eq:error_eq_state}. Then, standard properties of the bubble function $\varphi_T$ combined with basic inequalities yield
\begin{multline*}
\|\bar{u}_\T-a(\cdot,\bar{y}_\T)\|_{L^2(T)}^2\lesssim \left(h_T^{-1}\|\nabla(\bar{y}-\bar{y}_\T)\|_{L^2(T)}+\|\bar{u}-\bar{u}_\T\|_{L^2(T)}\right.\\
\left.+\|a(\cdot,\bar{y})-a(\cdot,\bar{y}_\T)\|_{L^2(T)}\right)\|\bar{u}_\T-a(\cdot,\bar{y}_\T)\|_{L^2(T)}.
\end{multline*}
This, in view of the local Lipschitz property of $a$ with respect to $y$ \eqref{eq:Lipschitz_cont}, implies that
\begin{multline*}
h_T^2\|\bar{u}_\T-a(\cdot,\bar{y}_\T)\|^2_{L^2(T)}
\lesssim \|\nabla(\bar{y}-\bar{y}_\T)\|^2_{L^2(T)}\\
+h_T^2\|\bar{u}-\bar{u}_\T\|^2_{L^2(T)}+\, h_T^2\|\bar{y}-\bar{y}_\T\|^2_{L^2(T)}.
\end{multline*}

\emph{Step 2}. Let $T\in\T$ and $S\in\Sides_{T}$. We bound $h_T\|\llbracket\nabla \bar{y}_\T\cdot\boldsymbol{\nu}\rrbracket\|_{L^2(S)}^2$ in \eqref{eq:error_estimator_state_eq}, i.e., the jump or interelement residual term. To accomplish this task, we set $v=\varphi_S\llbracket\nabla \bar{y}_\T\cdot\boldsymbol{\nu}\rrbracket$ in \eqref{eq:error_eq_state} and utilize standard bubble functions arguments to obtain
\begin{multline*}
\|\llbracket\nabla \bar{y}_\T\cdot\boldsymbol{\nu}\rrbracket\|_{L^2(S)}^2
\lesssim
\sum_{T'\in\mathcal{N}_S}\left(h_T^{-1}\|\nabla (\bar{y}-\bar{y}_\T)\|_{L^2(T')}+\|(a(\cdot,\bar{y})-a(\cdot,\bar{y}_\T)\|_{L^2(T')}\right.\\
\left.+\|\bar{u}-\bar{u}_\T\|_{L^2(T')}+\|\bar{u}_\T-a(\cdot,\bar{y}_\T)\|_{L^2(T')}\right)h_{T}^{\frac{1}{2}}\|\llbracket\nabla \bar{y}_\T\cdot\boldsymbol{\nu}\rrbracket\|_{L^2(S)}.
\end{multline*}
Using, again, the local Lipschitz property of $a$ with respect to $y$ we arrive at
\begin{multline*}
h_T\|\llbracket\nabla \bar{y}_\T\cdot\boldsymbol{\nu}\rrbracket\|^2_{L^2(S)}
\lesssim
\sum_{T'\in\mathcal{N}_S}\left(\|\nabla (\bar{y}-\bar{y}_\T)\|^2_{L^2(T')}\right.\\
\left.+\, h_T^2\|\bar{y}-\bar{y}_\T\|^2_{L^2(T')} 
+h_T^2\|\bar{u}-\bar{u}_\T\|^2_{L^2(T')}\right).
\end{multline*}

The collection of the estimates derived in Steps 1 and 2 concludes the proof.
\end{proof}

We now continue with the study of the local efficiency properties of the estimator $\mathcal{E}_{ad}$ defined in \eqref{eq:error_estimator_adjoint_eq}.

\begin{theorem}[local efficiency of $\mathcal{E}_{ad}$]\label{thm:local_eff_adjoint}
Suppose that assumptions \textnormal{\ref{A1}}--\textnormal{\ref{A3}} hold. Let $\bar u \in \mathbb{U}_{ad}$ be a local solution to \eqref{def:weak_ocp}--\eqref{eq:weak_st_eq}. Let $\bar u_{\T}$ be a local minimum of the discretization \eqref{eq:discrete_state_equation}--\eqref{eq:discrete_adjoint_equation} with $\bar{y}_{\T}$ and $\bar{p}_\T$ being the associated state and adjoint state, respectively. Then, for $T\in\T$, the local error indicator $\mathcal{E}_{ad,T}$ satisfies
\begin{multline}\label{eq:local_eff_ad}
\mathcal{E}_{ad,T}
\lesssim \|\nabla(\bar{p}-\bar{p}_\T)\|_{L^2(\mathcal{N}_{T})}+(1+h_T)\|\bar{y}-\bar{y}_\T\|_{L^2(\mathcal{N}_{T})} \\
+h_T\left(\|\bar{p}-\bar{p}_\T\|_{L^2(\mathcal{N}_{T})}+\|y_\Omega-\mathscr{P}_{\T}y_\Omega\|_{L^2(\mathcal{N}_{T})}\right),
\end{multline}
where $\mathcal{N}_{T}$ is defined as in \eqref{eq:patch}. The hidden constant is independent of the continuous and discrete optimal variables, the size of the elements in the mesh $\T$, and $\#\T$.
\end{theorem}
\begin{proof}
We estimate each term in the definition of the local error indicator $\mathcal{E}_{ad,T}$, given in \eqref{eq:indicators_adjoint_eq}, separately.

\emph{Step 1}. Let $T\in\T$. A simple application of the triangle inequality yields
\begin{multline*}
h_T\|\bar{y}_\T-y_\Omega-\tfrac{\partial a}{\partial y}(\cdot,\bar{y}_\T)\bar{p}_\T\|_{L^2(T)}
\\
\leq  h_T\|\bar{y}_\T-\mathscr{P}_{\T}y_\Omega-\tfrac{\partial a}{\partial y}(\cdot,\bar{y}_\T)\bar{p}_\T\|_{L^2(T)}+h_T\|\mathscr{P}_{\T}y_\Omega-y_\Omega\|_{L^2(T)}.
\end{multline*}
To estimate the first term on the right hand side of the previous estimate and also to simplify the presentation of the material, we define 
\[
\mathfrak{R}_{T}^{ad}:=\bar{y}_\T-\mathscr{P}_{\T} y_\Omega -\tfrac{\partial a}{\partial y}(\cdot,\bar{y}_\T)\bar{p}_\T.
\]
Now, set $w=\varphi_T \mathfrak{R}_T^{ad}$ in \eqref{eq:error_eq_adjoint} and invoke basic inequalities to arrive at
\begin{multline}\label{eq:residual_estimate_ad_0}
\| \varphi_T^{1/2} \mathfrak{R}_{T}^{ad}\|_{L^2(T)}^2 
\lesssim
\|\nabla(\bar{p}-\bar{p}_\T)\|_{L^2(T)}\|\nabla (\varphi_T \mathfrak{R}_T^{ad})\|_{L^2(T)}
\\
+ \|\varphi_T \mathfrak{R}_T^{ad} \|_{L^2(T)}
\left(\|\bar{y}-\bar{y}_\T\|_{L^2(T)}+\|\tfrac{\partial a}{\partial y}(\cdot,\bar{y})(\bar{p}-\bar{p}_\T)\|_{L^2(T)}+\|\mathscr{P}_{\T}y_\Omega-y_\Omega\|_{L^2(T)}\right)
\\
+\|\tfrac{\partial a}{\partial y}(\cdot,\bar{y})-\tfrac{\partial a}{\partial y}(\cdot,\bar{y}_\T)\|_{L^2(T)}\| \bar{p}_\T\|_{H^1(T)}\|\varphi_T \mathfrak{R}_T^{ad}\|_{H^1(T)}.
\end{multline}
Since $\mathfrak{R}_T^{ad}\varphi_T\in H_0^1(T)$, we have $\|\mathfrak{R}_T^{ad}\varphi_T\|_{H^1(T)} \lesssim \|\nabla(\mathfrak{R}_T^{ad}\varphi_T)\|_{L^2(T)}$. On the basis of \eqref{eq:residual_estimate_ad_0}, standard inverse inequalities and bubble functions arguments yield
\begin{multline}\label{eq:residual_estimate_ad_1}
\|\mathfrak{R}_{T}^{ad}\|_{L^2(T)}
\lesssim
h_T^{-1}\|\nabla(\bar{p}-\bar{p}_\T)\|_{L^2(T)}+\|\tfrac{\partial a}{\partial y}(\cdot,\bar{y})(\bar{p}-\bar{p}_\T)\|_{L^2(T)}\\
+ h_T^{-1}\|\tfrac{\partial a}{\partial y}(\cdot,\bar{y})-\tfrac{\partial a}{\partial y}(\cdot,\bar{y}_\T)\|_{L^2(T)}\| \bar{p}_\T\|_{H^1(T)}+\|\bar{y}-\bar{y}_\T\|_{L^2(T)}+\|\mathscr{P}_{\T}y_\Omega-y_\Omega\|_{L^2(T)}.
\end{multline}
Stability estimates for the problems that $\bar{p}_\T$ and $\bar{y}_\T$ solve yield the estimate
\begin{equation}\label{eq:stability_discrete_ad}
\|\bar{p}_\T\|_{H^1(T)}
\leq
\|\bar{p}_\T\|_{H^1(\Omega)}
\lesssim
\|y^{}_\Omega\|_{L^2(\Omega)} + \|y^{}_{\T} \|_{L^2(\Omega)}
\lesssim 
\|y^{}_\Omega\|_{L^2(\Omega)} + \rho|\Omega|^{\frac{1}{2}},
\end{equation}
where $\rho=\max\{|\texttt{a}|,|\texttt{b}|\}$. Replacing this estimate into \eqref{eq:residual_estimate_ad_1}, invoking the local Lipschitz property of $a$ with respect to the variable $y$ \eqref{eq:Lipschitz_cont} and assumption \ref{A3}, we conclude
\begin{multline}\label{eq:residual_estimate_ad_2}
h_T\|\mathfrak{R}_{T}^{ad}\|_{L^2(\Omega)}
\lesssim
\|\nabla(\bar{p}-\bar{p}_\T)\|_{L^2(T)}+h_T\|\bar{p}-\bar{p}_\T\|_{L^2(T)}\\
+(1+h_T)\|\bar{y}-\bar{y}_\T\|_{L^2(T)}+h_T\|\mathscr{P}_{\T}y_\Omega-y_\Omega\|_{L^2(T)}.
\end{multline}
Notice that the hidden constant is independent of the continuous and discrete optimal variables, the size of the elements in the mesh $\T$, and $\#\T$ but depends on the continuous problem data.

\emph{Step 2}. Let $T\in\T$ and $S\in\Sides_{T}$. Now we bound the jump term $\|\llbracket \nabla\bar{p}_\T\cdot\boldsymbol{\nu}\rrbracket\|_{L^2(S)}$ in \eqref{eq:indicators_adjoint_eq}.  To accomplish this task, we set $w=\llbracket\nabla \bar{p}_\T\cdot\boldsymbol{\nu}\rrbracket\varphi_S$ in \eqref{eq:error_eq_adjoint} and proceed with similar arguments as the ones used in \eqref{eq:residual_estimate_ad_0}--\eqref{eq:residual_estimate_ad_1}. We thus obtain
\begin{multline*}
\|\llbracket\nabla \bar{p}_\T\cdot\boldsymbol{\nu}\rrbracket\|_{L^2(S)}^2\lesssim\! \sum_{T'\in\mathcal{N}_{S}}\bigg(\!h_T^{-1}\|\nabla(\bar{p}-\bar{p}_\T)\|_{L^2(T')}+\|\bar{p}-\bar{p}_\T\|_{L^2(T')}
\\
+ \|\bar{y}-\bar{y}_\T\|_{L^2(T')}+\|\mathfrak{R}_T^{ad}\|_{L^2(T')}+\|\mathscr{P}_{\T}y_\Omega-y_\Omega\|_{L^2(T')}
\\
+h_T^{-1}\|\bar{p}_\T\|_{H^1(T)}\|\tfrac{\partial a}{\partial y}(\cdot,\bar{y})-\tfrac{\partial a}{\partial y}(\cdot,\bar{y}_\T)\|_{L^2(T')}\bigg)h_{T}^{\frac{1}{2}}\|\llbracket\nabla \bar{p}_\T\cdot\boldsymbol{\nu}\rrbracket\|_{L^2(S)}.
\end{multline*}
Finally, utilize the stability estimate \eqref{eq:stability_discrete_ad}, the local Lipschitz continuity of $\tfrac{\partial a}{\partial y}(\cdot,y)$ with respect to $y$ \eqref{eq:Lipschitz_cont}, and estimate \eqref{eq:residual_estimate_ad_2}, to conclude 
\begin{multline*}
h_{T}^{\frac{1}{2}}\|\llbracket\nabla \bar{p}_\T\cdot\boldsymbol{\nu}\rrbracket\|_{L^2(S)}\lesssim \sum_{T'\in\mathcal{N}_{S}}\left(\|\nabla(\bar{p}-\bar{p}_\T)\|_{L^2(T')}+h_{T}\|\bar{p}-\bar{p}_\T\|_{L^2(T')}\right.\\
\left.+(1+h_{T})\|\bar{y}-\bar{y}_\T\|_{L^2(T')}+h_{T}\|\mathscr{P}_{\T}y_\Omega-y_\Omega\|_{L^2(T')}\right).
\end{multline*}
Combine the estimates derived in Steps 1 and 2 to arrive at the desired estimate \eqref{eq:local_eff_ad}.
\end{proof}

The results of Theorems \ref{thm:local_eff_state} and \ref{thm:local_eff_adjoint} immediately yield the global efficiency of $\mathcal{E}_{ocp}$. To derive such a result, we define, for $w \in L^2(\Omega)$,
\[
\mathrm{osc}(w,\T):= \left(\sum_{T\in\T}h_T^2\|w-\mathscr{P}_{\T}w\|_{L^2(T)}^2\right)^\frac{1}{2}.
\]

\begin{theorem}[global efficiency of $\mathcal{E}_{ocp}$]\label{thm:global_eff}
Suppose that assumptions \textnormal{\ref{A1}}--\textnormal{\ref{A3}} hold. Let $\bar u \in \mathbb{U}_{ad}$ be a local solution to \eqref{def:weak_ocp}--\eqref{eq:weak_st_eq}. Let $\bar u_{\T}$ be a local minimum of the discretization \eqref{eq:discrete_state_equation}--\eqref{eq:discrete_adjoint_equation} with $\bar{y}_{\T}$ and $\bar{p}_\T$ being the associated state and adjoint state, respectively.  Then, the error estimator $\mathcal{E}_{ocp}$, defined in \eqref{def:error_estimator_ocp}, satisfies
\begin{equation*}
\mathcal{E}_{ocp}
\lesssim \|\bar{p}-\bar{p}_\T\|_{H^1(\Omega)}+\|\bar{y}-\bar{y}_\T\|_{H^1(\Omega)} 
+\|\bar{u}-\bar{u}_\T\|_{L^2(\Omega)}
+\mathrm{osc}(y_{\Omega},\T).
\end{equation*}
The hidden constant is independent of the continuous and discrete optimal variables, the size of the elements in the mesh $\T$, and $\#\T$.
\end{theorem}
\begin{proof}
We begin by invoking the definition of the global indicator $\mathcal{E}_{st}$, given by \eqref{eq:error_estimator_state_eq}, and the local efficiency estimate \eqref{eq:local_eff_st} to arrive at
\begin{equation}\label{eq:global_estimate_st}
\mathcal{E}_{st}
\lesssim 
\|\nabla(\bar{y}-\bar{y}_\T)\|_{L^2(\Omega)}+\text{diam}(\Omega)\|\bar{y}-\bar{y}_\T\|_{L^2(\Omega)}+\text{diam}(\Omega)\|\bar{u}-\bar{u}_\T\|_{L^2(\Omega)}.
\end{equation}
On the other hand, in view of \eqref{eq:error_estimator_adjoint_eq}, the efficiency estimate \eqref{eq:local_eff_ad} provides the bound
\begin{multline}\label{eq:global_estimate_ad}
\mathcal{E}_{ad}
\lesssim \|\nabla(\bar{p}-\bar{p}_\T)\|_{L^2(\Omega)}+(1+\text{diam}(\Omega))\|\bar{y}-\bar{y}_\T\|_{L^2(\Omega)} \\
+\text{diam}(\Omega)\|\bar{p}-\bar{p}_\T\|_{L^2(\Omega)}+\mathrm{osc}(y_{\Omega},\T).
\end{multline}
It thus suffices to control $\mathcal{E}_{ct} $. In view of \eqref{eq:error_estimator_control_var}, a trivial application of the triangle inequality yields
\begin{align*}
\mathcal{E}_{ct} 
&\leq \|\tilde{u}-\bar{u}\|_{L^2(\Omega)}+\|\bar{u}-\bar{u}_\T\|_{L^2(\Omega)}\\
&= \|\Pi_{[\texttt{a},\texttt{b}]}(-\nu^{-1}\bar{p}_\T)-\Pi_{[\texttt{a},\texttt{b}]}(-\nu^{-1}\bar{p})\|_{L^2(\Omega)}+\|\bar{u}-\bar{u}_\T\|_{L^2(\Omega)},
\end{align*}
where $\Pi_{[\texttt{a},\texttt{b}]}$ is defined as in \eqref{def:projector_pi}. This estimate, in conjunction with the Lipschitz property of $\Pi_{[\texttt{a},\texttt{b}]}$ and a Poincar\'e inequality, implies
\begin{equation}\label{eq:global_estimate_ct}
\mathcal{E}_{ct}
\lesssim
\nu^{-1}\|\nabla (\bar{p}_\T-\bar{p})\|_{L^2(\Omega)}+\|\bar{u}-\bar{u}_\T\|_{L^2(\Omega)}.
\end{equation}
The proof concludes by gathering the estimates \eqref{eq:global_estimate_st}, \eqref{eq:global_estimate_ad}, and \eqref{eq:global_estimate_ct}.
\end{proof}


\section{Extensions}\label{sec:extensions}
We present a few extensions of the theory developed in the previous sections.

\subsection{Piecewise linear approximation}
In this section, we consider a similar finite element discretization as the one introduced in section \ref{sec:fem_ocp} with the difference that to approximate  the optimal control variable $\bar u$ we employ  piecewise linear functions i.e., $\bar u_{\T} \in \mathbb{U}_{ad,1}(\mathscr{T})$, where
\begin{equation*}
\mathbb{U}_{ad,1}(\mathscr{T}):=\mathbb{U}_1(\mathscr{T})\cap \mathbb{U}_{ad},
\quad
\mathbb{U}_1(\mathscr{T}):=\{ u_\T\in C(\bar \Omega): u_\T|^{}_T\in \mathbb{P}_1(T) \ \forall \ T\in \T\}.
\end{equation*}

The following discrete optimal control problem can thus be proposed: 
Find $\min J(y_{\T},u_{\T})$ subject to the discrete state equation
\begin{equation}
\label{eq:discrete_state_equation_linear}
y^{}_\mathscr{T} \in \mathbb{V}(\mathscr{T}): \quad
(\nabla y^{}_\mathscr{T},\nabla v^{}_\mathscr{T})_{L^2(\Omega)}+(a(\cdot,y_\T),v_\T)_{L^2(\Omega)}  =  (u^{}_\mathscr{T},v^{}_\mathscr{T})^{}_{L^2(\Omega)}
\end{equation}
for all $v^{}_\mathscr{T} \in \mathbb{V}(\mathscr{T})$ and the discrete control constraints $u^{}_{\mathscr{T}} \in \mathbb{U}_{ad,1}(\T)$. The well--posedness of this solution technique as well as first order optimality conditions follow from \cite[Theorem 3.3]{MR2350349}. For a priori error estimates, we refer the reader to \cite[Theorem 4.1]{MR2350349} and \cite[section 10]{MR3586845}.

We propose an a posteriori error estimator that accounts for the discretization of the state, adjoint state, and control variables when the error, in each one of these variables, is measured in the $L^2(\Omega)$-norm.  As it is customary when performing an a posteriori error analysis based on duality, we assume that $\Omega$ is convex.

Assume that we have at hand, a posteriori error estimators ${E}_{st}$ and ${E}_{ad}$ such that
\begin{equation}\label{eq:estimate_state_hat_discrete_duality}
\|\hat{y}-\bar{y}_\T\|_{L^2(\Omega)}
\lesssim {E}_{st},
\quad
\|\hat{p}-\bar{p}_\T\|_{L^2(\Omega)}
\lesssim {E}_{ad}.
\end{equation}

Define, for $(v,w,z) \in L^2(\Omega)\times L^2(\Omega)\times L^2(\Omega)$, the norm 
\[
\| (v,w,z)\|_{\Omega} := \| v \|_{L^{2}(\Omega)} + \| w \|_{L^{2}(\Omega)} + \| z \|_{L^{2}(\Omega)}. 
\]

We present the following global reliability result.

\begin{theorem}[global reliability]\label{thm:control_bound_duality}
Suppose that assumptions \textnormal{\ref{A1}}--\textnormal{\ref{A3}} hold. Let $\bar u \in \mathbb{U}_{ad}$ be a local solution to \eqref{def:weak_ocp}--\eqref{eq:weak_st_eq} satisfying the sufficient second order condition \eqref{eq:second_order_2_2}, or equivalently \eqref{eq:second_order_equivalent}. Let $\bar u_{\T}$ be a local minimum of the associated discrete optimal control problem with $\bar{y}_{\T}$ and $\bar{p}_\T$ being the corresponding state and adjoint state, respectively. Let $\T$ be a mesh such that \eqref{eq:assumption_mesh} holds, then
\begin{equation}\label{eq:global_rel_duality}
\| (\bar{y} - \bar{y}_{\T},\bar{p} - \bar{p}_{\T},\bar{u} - \bar{u}_{\T})\|_{\Omega}
\lesssim 
{E}_{st} + {E}_{ad} + \mathcal{E}_{ct}.
\end{equation}
The hidden constant is independent of the continuous and discrete optimal variables, the size of the elements in the mesh $\T$, and $\#\T$.
\end{theorem}
\begin{proof}
The proof of the estimate \eqref{eq:global_rel_duality} follows closely the arguments developed in the proof of Theorem \ref{thm:control_bound}. In fact, with the estimate \eqref{eq:auxiliar_estimate} at hand, we arrive at
\begin{equation}
\label{eq:auxiliar_estimate_duality}
\|\bar{u}-\tilde{u}\|_{L^2(\Omega)}
\lesssim
\|\tilde{p}-\hat{p}\|_{L^2(\Omega)}+\|\hat{p}-\bar{p}_\T\|_{L^2(\Omega)} \lesssim \|\tilde{p}-\hat{p}\|_{L^2(\Omega)} + {E}_{ad},
\end{equation}
where we have used \eqref{eq:estimate_state_hat_discrete_duality}. We now use of a Poincar\'e inequality in conjunction with the first estimate in \eqref{eq:estimate_adjoint_tilde_hat} to obtain
\begin{equation}\label{eq:aux_duality_1}
\|\tilde{p}-\hat{p}\|_{L^2(\Omega)}
\lesssim
\|\nabla(\tilde{p}-\hat{p})\|_{L^2(\Omega)}
\lesssim
\|\tilde{y}-\bar{y}_\T\|_{L^2(\Omega)}.
\end{equation}
The hidden constant is independent of the continuous and discrete optimal variables, the size of the elements in the mesh $\T$, and $\#\T$ but depends on the continuous problem data.

To control $\|\tilde{y}-\bar{y}_\T\|_{L^2(\Omega)}$ we invoke the auxiliary state $\hat y$ defined as the solution to \eqref{eq:hat_functions} and apply the triangle inequality. With these arguments we obtain
\begin{equation}
\label{eq:aux_duality_2}
\|\tilde{y}-\bar{y}_\T\|_{L^2(\Omega)} \leq  \|\tilde{y}- \hat y \|_{L^2(\Omega)} +  \| \hat y - \bar{y}_{\T}\|_{L^2(\Omega)} \lesssim  \|\tilde{y}- \hat y \|_{L^2(\Omega)} + {E}_{st},
\end{equation}
where we have also used \eqref{eq:estimate_state_hat_discrete_duality}. To bound $\|\tilde{y}- \hat y \|_{L^2(\Omega)}$ we set $v = \tilde y - \hat y$ in problem \eqref{eq:tildey-haty}. This, in view of the fact that $a$ is monotone increasing with respect to $y$, yields
\[
\|\tilde{y}- \hat y \|_{L^2(\Omega)} \lesssim \| \nabla( \tilde{y}- \hat y) \|_{L^2(\Omega)} \lesssim \| \tilde{u} -\bar{u}_{\T}  \|_{L^2(\Omega)} = \mathcal{E}_{ct}.
\]
Replacing this estimate into \eqref{eq:aux_duality_2}, and the obtained one into \eqref{eq:aux_duality_1}, we obtain the estimate $\|\tilde{p}-\hat{p}\|_{L^2(\Omega)} \lesssim E_{st} + \mathcal{E}_{ct}$. This, in view of \eqref{eq:auxiliar_estimate_duality}, reveals the a posteriori error estimate
\[
\|\bar{u}-\bar{u}_{\T}\|_{L^2(\Omega)} \lesssim {E}_{st}+{E}_{ad}+ \mathcal{E}_{ct}.
\]

The control of $\|\bar{y}-\bar{y}_{\T}\|_{L^2(\Omega)}$ and $\|\bar{p}-\tilde{p}_{\T}\|_{L^2(\Omega)}$ follow similar arguments as the ones elaborated in the proof of Theorem \ref{thm:control_bound}. For brevity, we skip details.
\end{proof}


\subsection{Sparse PDE--constrained optimization}
Define $\psi: L^1(\Omega) \rightarrow \mathbb{R}$ by $\psi(u) :=  \|u\|_{L^1(\Omega)}$. In this section, we present a posteriori error estimates for a semilinear optimal control problem that involves the nondifferentiable cost functional
\begin{equation*}
\mathfrak{J}(y, u):= J(y,u) + \vartheta \psi(u) = \frac{1}{2} \| y - y_\Omega \|^{2}_{L^{2}(\Omega)} + \frac{\nu}{2}\|u\|_{L^2(\Omega)}^2 + \vartheta \|u\|_{L^1(\Omega)}.
\end{equation*} 
Here, $\vartheta > 0$ denotes a sparsity parameter and $\nu>0$ corresponds to the so-called regularization parameter.  The linear case has been investigated in \cite{10.1093/imanum/drz025}. The cost functional $\mathfrak{J}$ involves the $L^1(\Omega)$-norm of the control variable, which is a natural measure of the control cost, and leads to sparsely supported optimal controls \cite{MR3023751,MR2826983}.

We consider the following sparse PDE--constrained optimization problem: Find $\min \{\mathfrak{J}(y, u): (y,u) \in H_0^1(\Omega) \times \mathbb{U}_{ad}\}$ subject to \eqref{eq:weak_st_eq}. This problem admits at least  one optimal solution $(\bar y, \bar u) \in H_0^1(\Omega) \times \mathbb{U}_{ad}$. In addition, if $\bar u$ is a local minimum, then there exists $\bar y \in H_0^1(\Omega)$, $\bar p \in H_0^1(\Omega)$, and $\bar \lambda \in \partial \psi(\bar u)$ such that \eqref{eq:weak_st_eq} and \eqref{eq:adj_eq} hold and 
\begin{equation*}
(\bar p + \nu \bar u + \vartheta \bar \lambda, u - \bar u)_{L^2(\Omega)} \geq 0 \quad \forall\ u \in \mathbb{U}_{ad};
\end{equation*}
see \cite[Theorem 3.1]{MR3023751}. The following characterizations for the optimal control $\bar{u}$ and its associated subgradient $\bar{\lambda}$ hold \cite[Corollary 3.2]{MR3023751}:
\begin{equation*}
\bar{\lambda}(x):=\Pi_{[-1,1]} \left(-{\vartheta}^{-1}\bar{p}(x) \right),
\quad
\bar{u}(x) = \Pi_{[\texttt{a},\texttt{b}]} \left(-{\nu}^{-1} \left[ \bar{p}(x) + \vartheta \bar{\lambda}(x) \right] \right) \textrm{ a.e. } x \in \Omega.
\end{equation*}

We propose the following discrete optimal control problem: Find $\min \mathfrak{J}(y_{\T}, u_{\T})$ subject to \eqref{eq:discrete_state_equation_linear} and the discrete control constraints $u_{\T} \in \mathbb{U}_{ad}(\T)$. The existence of solutions for this scheme as well as first order optimality conditions follow from \cite[section 4]{MR3023751}.

Define the cones
\begin{align*}
\mathfrak{C}_{\bar{u}}:&=\{v \in L^2(\Omega) \text{ satisfying } \eqref{eq:cone_def} \text{ and } j'(\bar{u})v + \vartheta\psi'(\bar{u}; v) = 0\},\\
\mathfrak{C}_{\bar{u}}^{\tau}:&=\{v \in L^2(\Omega) \text{ satisfying } \eqref{eq:cone_def} \text{ and } j'(\bar{u})v + \vartheta\psi'(\bar{u}; v) \leq \tau \|v\|_{L^2(\Omega)}\}.
\end{align*}
Necessary and sufficient second order optimality conditions follow from \cite[Theorem 3.7 and 3.9]{MR3023751}: If $\bar u$ is a local minimum, then $j''(\bar u ) v^2 \geq 0$ for all $v \in \mathfrak{C}_{\bar{u}}$. Conversely, let $\bar u \in \mathbb{U}_{ad}$ and $\lambda \in \partial \psi(\bar u)$ satisfy the associated first order optimality conditions. If $j''(\bar u ) v^2 > 0$ for all $v \in \mathfrak{C}_{\bar{u}} \setminus\{0\}$, then $\bar u$ is a local minimum. In addition, we have the equivalence \cite[Theorem 3.8]{MR3023751}
\begin{equation}
\label{eq:second_order_2_sparse}
j''(\bar u ) v^2 > 0 \ \forall v \in \mathfrak{C}_{\bar{u}} \setminus\{0\}  \iff \exists \mu, \tau >0: j''(\bar u ) v^2 \geq \mu \| v \|_{L^2(\Omega)}^2 \ \forall  v \in \mathfrak{C}_{\bar{u}}^{\tau}.
\end{equation}

Define, for a.e.  $x \in \Omega$, the auxiliary variables
\begin{equation}
\tilde{\lambda}(x):=\Pi_{[-1,1]} \left(-{\vartheta}^{-1} \bar{p}_\T(x) \right),
\quad
\tilde u(x) = \Pi_{[\texttt{a},\texttt{b}]} \left(-{\nu}^{-1} \left[ \bar{p}_{\T}(x) + \vartheta \tilde \lambda(x) \right] \right).
\label{def:control_tilde_sparse}
\end{equation}

To present a posteriori error estimates, we define the error indicators 
\begin{equation*}
\mathcal{E}_{sg,T}^2:=\|\tilde{\lambda}-\bar{\lambda}_\T\|_{L^2(T)}^2,
\quad 
\mathcal{E}_{ct,T}^2:=\|\tilde{u}-\bar{u}_\T\|_{L^2(T)}^2,
\end{equation*}
and error estimators 
\begin{equation}\label{eq:error_estimators_subgradient}
\mathcal{E}_{sg}:=\left(\sum_{T\in\T}\mathcal{E}_{sg,T}^2\right)^\frac{1}{2},
\quad
\mathcal{E}_{ct}:=\left(\sum_{T\in\T}\mathcal{E}_{ct,T}^2\right)^\frac{1}{2}.
\end{equation}

\begin{theorem}[global reliability]\label{thm:control_bound_sparse}
Suppose that assumptions \textnormal{\ref{A1}}--\textnormal{\ref{A3}} hold. Let $\bar u \in \mathbb{U}_{ad}$ be a local solution to the sparse PDE--constrained optimization problem satisfying the sufficient second order condition \eqref{eq:second_order_2_sparse}. Let $\bar u_{\T}$ be a local minimum of the associated discrete optimal control problem with $\bar{y}_{\T}$, $\bar{p}_\T$, and $\bar{\lambda}_{\T}$ being the corresponding state, adjoint state, and subgradient, respectively. Let $\T$ be a mesh such that \eqref{eq:inequality_control_tilde} holds with $\tilde{u}$ as in \eqref{def:control_tilde_sparse}, then
\begin{equation*}
\VERT{(\bar{y} - \bar{y}_{\T},\bar{p} - \bar{p}_{\T},\bar{u} - \bar{u}_{\T})\VERT}_{\Omega}+\|\bar{\lambda}-\bar{\lambda}_\T\|_{L^2(\Omega)}
\lesssim 
\mathcal{E}_{st} + \mathcal{E}_{ad} + \mathcal{E}_{ct} + \mathcal{E}_{sg}.
\end{equation*}
The hidden constant is independent of the continuous and discrete optimal variables, the size of the elements in the mesh $\T$, and $\#\T$.
\end{theorem}
\begin{proof}
Since \eqref{eq:inequality_control_tilde} is assumed to hold and it does not involve the nondifferentiable term $\psi$, the estimate of the error associated to the state, adjoint state, and control variables is as presented in the proof of Theorem \ref{thm:control_bound}. It thus suffices to control the error associated to the approximation of the subgradient $\bar{\lambda}$.
To accomplish this task, we invoke 
\eqref{eq:error_estimators_subgradient} and immediately conclude that
\begin{equation}\label{eq:estimate_sparse_sg}
\|\bar{\lambda}-\bar{\lambda}_\T\|_{L^2(\Omega)}\leq \|\bar{\lambda}-\tilde{\lambda}\|_{L^2(\Omega)}+\mathcal{E}_{sg}.
\end{equation}
The Lipschitz property of $\Pi_{[-1,1]}$ and a Poincar\'e inequality yield
\begin{equation*}
\|\bar{\lambda}-\tilde{\lambda}\|_{L^2(\Omega)}\leq \vartheta^{-1}\|\bar{p}-\bar{p}_\T\|_{L^2(\Omega)}\lesssim \|\nabla(\bar{p}-\bar{p}_\T)\|_{L^2(\Omega)}.
\end{equation*}
Replace this estimate into \eqref{eq:estimate_sparse_sg} and invoke \eqref{eq:estimate_error_adjoint_st} to conclude.
\end{proof}

\begin{remark}[feasibility of estimate \eqref{eq:inequality_control_tilde}]
Notice that $\tilde u$ coincides with the discrete approximation of $\bar u$ when the so--called variational discretization scheme is employed. For such an approximation scheme and within the framework of a priori error estimates, inequality \eqref{eq:inequality_control_tilde} is proven in \cite[section 5]{MR3023751} and \cite[Lemma 4.6]{MR3023751}.
\end{remark}

\section{Numerical results}\label{sec:numer_results}

In this section, we conduct a series of numerical examples that illustrate the performance of the devised a posteriori error estimator $\mathcal{E}_{ocp}$ defined in \eqref{def:error_estimator_ocp}.

All the experiments have been carried out with the help of a code that we implemented using \texttt{C++}. All matrices have been assembled exactly and global linear systems were solved using the multifrontal massively parallel sparse direct solver (MUMPS) \cite{MUMPS1,MUMPS2}.  The right hand sides and terms involving the functions $a(\cdot,y)$ and $y_{\Omega}$, the approximation errors, and the error estimators are computed by a quadrature formula which is exact for polynomials of degree nineteen $(19)$ for two dimensional domains and degree fourteen $(14)$ for three dimensional domains.

For a given partition $\mathscr{T}$, we seek $(\bar{y}^{}_\mathscr{T},\bar{p}^{}_\mathscr{T},\bar{u}^{}_\mathscr{T})\,\in \mathbb{V}(\mathscr{T}) \times \mathbb{V}(\mathscr{T}) \times \mathbb{U}_{ad}(\T)$ that solves the discrete problem \eqref{eq:discrete_state_equation}--\eqref{eq:discrete_adjoint_equation}. This optimality system is solved by using a Newton--type primal--dual active set strategy as described in \textbf{Algorithms} \ref{Algorithm2} and \ref{Algorithm3}. To be precise, \textbf{Algorithm} \ref{Algorithm2} presents a variant of the well--known primal--dual active set strategy that can be found, for instance, in \cite[section 2.12.4]{Troltzsch}. On the other hand, \textbf{Algorithm} \ref{Algorithm3} describes the also well--known Newton method \cite[section 4.4.1]{MR1817388}. To present the latter, we define $\mathcal{X}(\T) := \mathbb{V}(\T) \times \mathbb{V}(\T) \times \mathbb{U}(\T)$ and introduce, for $\Psi = (y_{\T},p_{\T},u_{\T})$ and $\Theta = (v_{\T},w_{\T},t_{\T})$ in $\mathcal{X}(\T)$, the operator $F_{\T}: \mathcal{X}(\T) \rightarrow \mathcal{X}(\T)'$ as
\begin{align*}
\label{eq:F_Tau_Newton}
\langle F_{\T}(\Psi),\Theta \rangle := \begin{pmatrix}
(\nabla y_{\T},\nabla v_\mathscr{T})_{L^2(\Omega)} + ( a(\cdot,y_{\T}) - u_{\T},v_\mathscr{T})_{L^2(\Omega)} \\
(\nabla w_\mathscr{T},\nabla p_{\T})_{L^2(\Omega)} + \left(\frac{\partial a}{\partial y}(\cdot,y_{\T}) p_{\T} - y_{\T} + y_{\Omega} ,w_\mathscr{T}\right)_{L^2(\Omega)} \\
\left( \nu^{-1}\Pi_{\T}p_{\T}(\mathbf{1} - \boldsymbol{\chi}_{\texttt{a}} - \boldsymbol{\chi}_{\texttt{b}}) + u_{\T}\mathbf{1} - \texttt{a}\boldsymbol{\chi}_{\texttt{a}} - \texttt{b}\boldsymbol{\chi}_{\texttt{b}},t_{\T}\right)_{L^2(\Omega)}
\end{pmatrix}.
\end{align*}
Here, $\Pi_{\T}$ denotes $L^{2}$--projection operator onto piecewise constant functions over $\T$ and $\langle \cdot,\cdot \rangle$ denotes the duality pairing between $\mathcal{X}(\T)'$ and $\mathcal{X}(\T)$. In addition, 
\[
\boldsymbol{\chi}_{\texttt{a}}, \, \boldsymbol{\chi}_{\texttt{b}} \in \mathbb{R}^{\#\T},
\quad
\mathbf{1} = (1, \dots, 1)^{\intercal} \in \mathbb{R}^{\#\T}.
\]

Given an initial guess $\Psi_{0} = (y_\T^{0},p_\T^{0},u_\T^{0}) \in \mathcal{X}(\T)$ and $k \in \mathbb{N}_{0}$, we consider the following Newton iteration:
\[
 \Psi_{k+1} = \Psi_{k} + \eta,
\]
where the incremental term $ \eta \! = \! (\delta y_{\T},\delta p_{\T},\delta u_{\T})\in\mathcal{X}(\T)$ solves 
\begin{equation}\label{eq:Newton_method_iter}
\langle F'_{\T}(\Psi_{k})(\eta), \Theta \rangle = -\langle F_{\T}(\Psi_{k}),\Theta \rangle \qquad \forall \Theta = (v_{\T},w_{\T},t_{\T})\in \mathcal{X}(\T).
\end{equation}
Here, $F'_{\T}(\Psi_{k})(\eta)$ denotes the G\^ateaux derivate of $F_{\T}$ in $\Psi_{k} = (y_\T^{k},p_\T^{k},u_\T^{k})$ evaluated at the direction $\eta$. 

Once the discrete solution is obtained, we use the local error indicator $\mathcal{E}_{ocp,T}$, defined as, 
\begin{align}\label{def:indicator_ocp}
\mathcal{E}_{ocp,T}^2:= \mathcal{E}_{st,T}^{2}+ \mathcal{E}_{ad,T}^{2} + \mathcal{E}_{ct,T}^{2},
\end{align}
to drive the adaptive procedure described in \textbf{Algorithm} \ref{Algorithm1}. A sequence of adaptively refined meshes is thus generated from the initial meshes shown in Figure \ref{fig:initial_meshes}. The total number of degrees of freedom is $\mathsf{Ndof}=2\dim(\mathbb{V}(\T)) + \dim(\mathbb{U}(\T))$. 

Finally, we define $e_y:=\bar{y}-\bar{y}_\T$, $e_p:=\bar{p}-\bar{p}_\T$, $e_u:=\bar{u}-\bar{u}_\T$, and the total error $e:=(e_y,e_p,e_u)$. To measure the total error we use $\VERT e \VERT_\Omega=\VERT(e_y,e_p,e_u)\VERT_\Omega$, where $\VERT \cdot \VERT_{\Omega}$ is defined as in \eqref{def:total_error}.
\begin{figure}[!ht]
\centering
\begin{minipage}{0.310\textwidth}\centering
\includegraphics[trim={0 0 0 0},clip,width=1.5cm,height=1.5cm,scale=0.2]{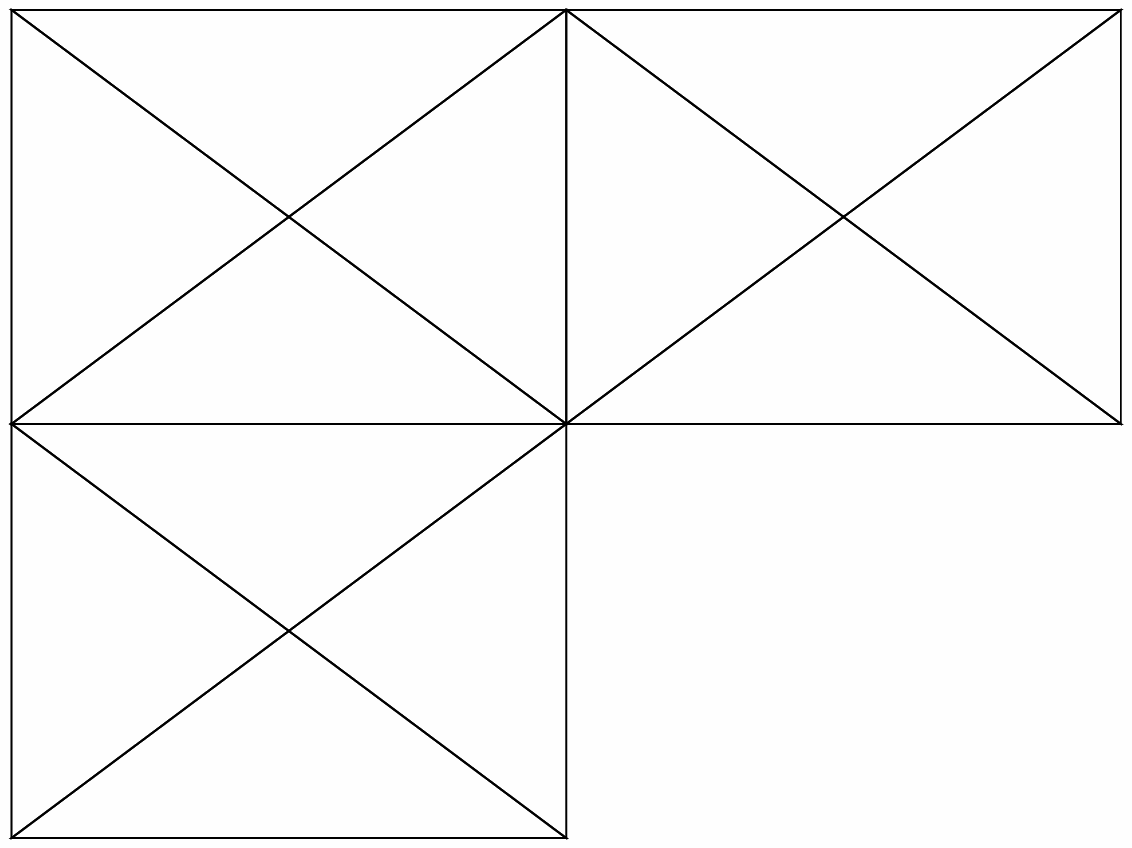}\\
\end{minipage}
\begin{minipage}{0.310\textwidth}\centering
\includegraphics[trim={0 0 0 0},clip,width=1.5cm,height=1.5cm,scale=0.2]{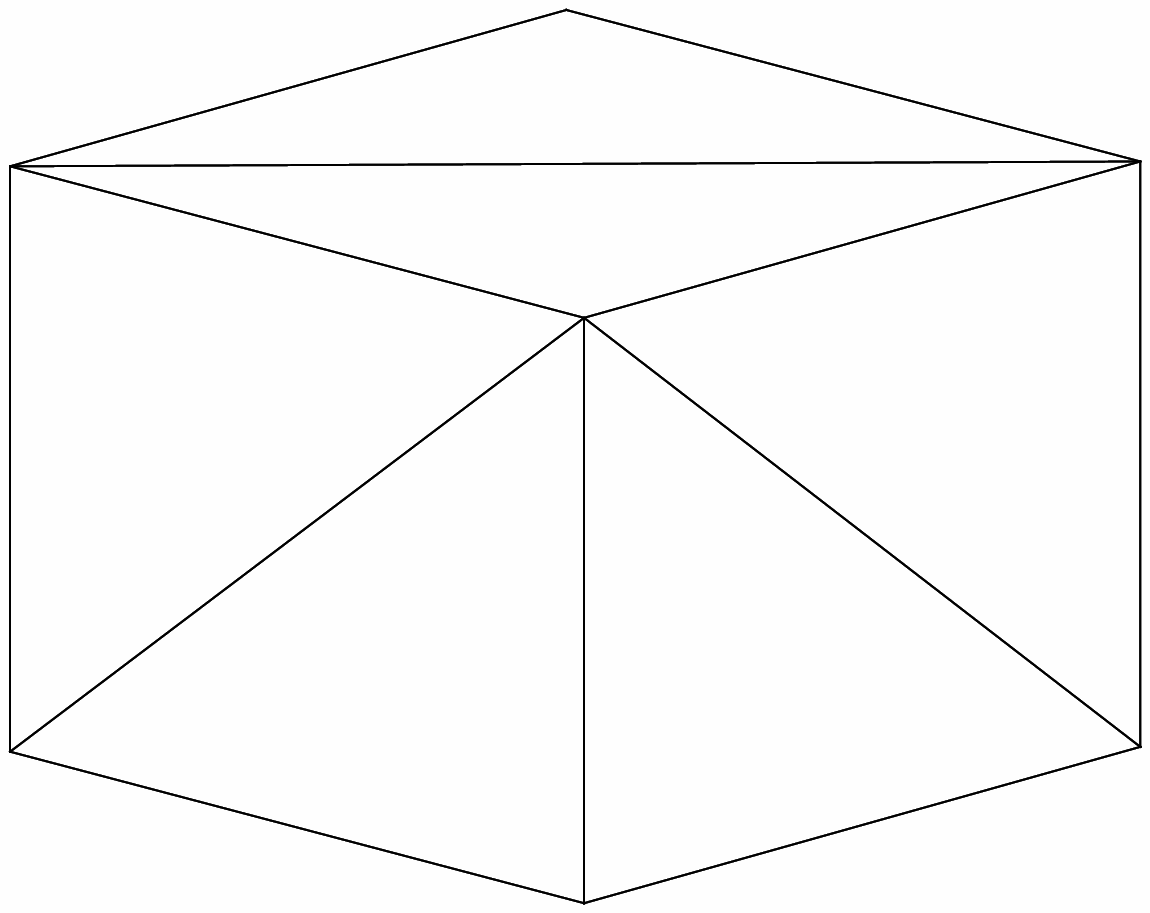}\\
\end{minipage}
\caption{The initial meshes used when the domain $\Omega$ is a $L$-shape (Example 1) and a cube (Example 2).}
\label{fig:initial_meshes}
\end{figure}
\begin{algorithm}[ht]
\caption{\textbf{Adaptive algorithm.}}
\label{Algorithm1}
\textbf{Input:} Initial mesh $\mathscr{T}_{0}$, constraints $\texttt{a}$ and $\texttt{b}$, and regularization parameter $\nu$;
\\
\textbf{Set:} $i=0$.
\\
\textbf{Active set strategy:}
\\
$\boldsymbol{1}$: Choose an initial discrete guess $(y^{0}_{\mathscr{T}_{i}},p^{0}_{\mathscr{T}_{i}},u^{0}_{\mathscr{T}_{i}}) \in \mathbb{V}(\T_{i}) \times \mathbb{V}(\T_{i}) \times \mathbb{U}(\T_{i})$; 
\\
$\boldsymbol{2}$: Compute $[\bar{y}_{\mathscr{T}_{i}},\bar{p}_{\mathscr{T}_{i}},\bar{u}_{\mathscr{T}_{i}}]=\textbf{Active-Set}[\mathscr{T}_i,\texttt{a},\texttt{b},\nu,y^{0}_{\mathscr{T}_{i}},p^{0}_{\mathscr{T}_{i}},u^{0}_{\mathscr{T}_{i}}]$ by using \textbf{Algorithm} \ref{Algorithm2};
\\
\textbf{Adaptive loop:}
\\
$\boldsymbol{3}$: For each $T \in \mathscr{T}_i$ compute the local error indicator $\mathcal{E}_{ocp,T}$ defined in \eqref{def:indicator_ocp};
\\
$\boldsymbol{4}$: Mark an element $T \in \T_i$ for refinement if $\mathcal{E}_{ocp,T}^{2}> \frac{1}{2}\max_{T'\in \mathscr{T}_i}\mathcal{E}_{ocp,T'}^{2}$;
\\
$\boldsymbol{5}$: From step $\boldsymbol{4}$, construct a new mesh, using a longest edge bisection algorithm. Set $i \leftarrow i + 1$ and go to step $\boldsymbol{1}$.
\end{algorithm}
\begin{algorithm}[!ht]
\caption{\textbf{Active set algorithm}
}
\label{Algorithm2}
\textbf{Input:} Mesh $\mathscr{T}$, constraints $\texttt{a}$ and $\texttt{b}$, regularization parameter $\nu$ and initial guess $(y^{0}_{\mathscr{T}},p^{0}_{\mathscr{T}},u^{0}_{\mathscr{T}}) \in \mathbb{V}(\T) \times \mathbb{V}(\T) \times \mathbb{U}(\T)$;
\\
$\boldsymbol{1}$: Define $\boldsymbol{\chi}_{\texttt{a}}^{old} = (\chi_{\texttt{a},T}^{old})_{T \in \T}, \boldsymbol{\chi}_{\texttt{b}}^{old} = (\chi_{\texttt{b},T}^{old})_{T \in \T} \in \mathbb{R}^{\#\T}$ with $\chi_{\texttt{a},T}^{old},\chi_{\texttt{b},T}^{old} \in \{0,1\}$.
\\
\textbf{Set:} $j=0$.
\\
$\boldsymbol{2}$: Compute $[y^{j+1}_{\T},p^{j+1}_{\T},u^{j+1}_{\T}]=\textbf{Newton}[\mathscr{T},\texttt{a},\texttt{b},\nu,\boldsymbol{\chi}_{\texttt{a}}^{old},\boldsymbol{\chi}_{\texttt{b}}^{old},y^{j}_{\T},p^{j}_{\T},u^{j}_{\T}]$ by using \textbf{Algorithm} \ref{Algorithm3}.
\\
$\boldsymbol{3}$: For each $T\in\mathscr{T}$ compute  \vspace*{-0.3cm}
\[ \chi_{\texttt{a},T}^{new} = \left\{\begin{array}{ll}
1 & \text{if } -\frac{1}{\nu}\Pi_{T}\left(p^{j+1}_{\T} \right) < \texttt{a}, \\
0 & \text{otherwise }
\end{array}
\right. 
\qquad 
\chi_{\texttt{b},T}^{new} = \left\{\begin{array}{ll}
1 & \text{if } -\frac{1}{\nu}\Pi_{T}\left(p^{j+1}_{\T} \right) > \texttt{b}, \\
0 & \text{otherwise, }
\end{array}
\right. \] 
where $\Pi_{T}$ denotes the $L^{2}$--projection onto piecewise constant functions over $T$.
\\
$\boldsymbol{4}$: If $\displaystyle{\sum_{T \in \mathscr{T}}}\left(|\chi_{\texttt{a},T}^{new} - \chi_{\texttt{a},T}^{old}| + |\chi_{\texttt{b},T}^{new} - \chi_{\texttt{b},T}^{old}| \right) = 0$, set $(\bar{y}_{\mathscr{T}},\bar{p}_{\mathscr{T}},\bar{u}_{\mathscr{T}}) = (y^{j+1}_{\T},p^{j+1}_{\T},u^{j+1}_{\T}).$
Otherwise, set $\boldsymbol{\chi}_{\texttt{a}}^{old}:=\boldsymbol{\chi}_{\texttt{a}}^{new}$, $\boldsymbol{\chi}_{\texttt{b}}^{old}:=\boldsymbol{\chi}_{\texttt{b}}^{new}$, and $j \leftarrow j+1$, and go to step $\boldsymbol{2}$.
\end{algorithm}

\begin{algorithm}[!ht]
\caption{\hspace{-0.12cm}
\textbf{Newton method
}}
\label{Algorithm3}
\textbf{Input:} Mesh $\mathscr{T}$, constraints $\texttt{a}$ and $\texttt{b}$, regularization parameter $\nu$, initial guess $(y^{0}_{\mathscr{T}},p^{0}_{\mathscr{T}},u^{0}_{\mathscr{T}}) \in \mathbb{V}(\T) \times \mathbb{V}(\T) \times \mathbb{U}(\T)$ and $\boldsymbol{\chi}_{\texttt{a}}, \boldsymbol{\chi}_{\texttt{b}} \in \mathbb{R}^{\#\T}$;
\\
\textbf{Set:} $k=0$.
\\
$\boldsymbol{1}$: Given $(y_\T^{k},p_\T^{k},u_\T^{k})$, compute the incremental $\eta=(\delta y_{\mathscr{T}},\delta p^{}_{\mathscr{T}},\delta u^{}_{\mathscr{T}}) \in \mathbb{V}(\T) \times \mathbb{V}(\T) \times \mathbb{U}(\T)$ as the solution to \eqref{eq:Newton_method_iter}.
\\
$\boldsymbol{2}$: Set $(y^{k+1}_{\T},p^{k+1}_{\T},u^{k+1}_{\T})=(y^{k}_{\T},p^{k}_{\T},u^{k}_{\T}) + (\delta y^{}_\mathscr{T},\delta p^{}_\mathscr{T},\delta u^{}_\mathscr{T})$. 
\\
$\boldsymbol{3}$: If $\max\{ \|\delta y_{\T}\|_{L^{\infty}(\Omega)}, \|\delta p_{\T}\|_{L^{\infty}(\Omega)}, \|\delta u_{\T}\|_{L^{\infty}(\Omega)}\} < 10^{-8}$, set $(y_{\mathscr{T}},p_{\mathscr{T}},u_{\mathscr{T}}) = (y^{k+1}_{\T},p^{k+1}_{\T},u^{k+1}_{\T})$. Otherwise, set $k \leftarrow k+1$ and go to step $\boldsymbol{1}$.
\end{algorithm}

In order to simplify the construction of exact solutions, we incorporate an extra source term $f \in L^{\infty}(\Omega)$ in the state equation \eqref{eq:weak_st_eq}. With such a modification, the right hand side of \eqref{eq:weak_st_eq} now reads $(f+u,v)_{L^2(\Omega)}$.
\\~\\
\textbf{Example 1.} We let $\Omega=(-1,1)^2\setminus[0,1)\times(-1,0]$, $a(\cdot,y) = \arctan(y)$, $\texttt{a} = -40$, $\texttt{b} = -0.1$, and $\nu \in \{ 10^{-3}, 10^{-4}, 10^{-5}\}$ . The exact optimal state and adjoint state are given, in polar coordinates $(r,\theta)$ with $\theta \in [0,3\pi /2]$, by
\[\bar{y}(r,\theta) = \bar{p}(r,\theta) = \sin\left( \pi/2(r\sin\theta) + 1 \right)\sin\left( \pi/2(r\cos\theta) + 1 \right)r^{2/3}\sin(2\theta/3).\]

The purpose of this numerical example is threefold. First, we compare the performance of our adaptive FEM with uniform refinement. Second, we investigate the performance of the devised a posteriori error estimator when varying the parameter $\nu$. Third, we compare the performance of our error estimator with the one presented in \cite[section 3]{MR2024491}. To present the error estimator of \cite{MR2024491}, we introduce
\begin{equation*}
\mathfrak{E}_{st}:=\mathcal{E}_{st}, \quad \mathfrak{E}_{ad}:=\mathcal{E}_{ad}, \quad \mathfrak{E}_{ct,T} := h_{T}\| \nabla \bar{p}_{\T} \|_{L^{2}(T)}, \ \ \mathfrak{E}_{ct} := \left(\displaystyle{\sum_{T \in \T}\mathfrak{E}_{ct,T}^{2}}\right)^{\frac{1}{2}},
\end{equation*}
where $\mathcal{E}_{st}$ and $\mathcal{E}_{ad}$ are defined as in \eqref{eq:error_estimator_state_eq} and \eqref{eq:error_estimator_adjoint_eq}, respectively. The total error indicator can thus be defined as follows \cite[section 3]{MR2024491}:
\begin{equation}\label{eq:chinese_total_indicator}
\mathfrak{E}_{ocp,T}^{2} = \mathfrak{E}_{st,T}^2 + \mathfrak{E}_{ad,T}^2 + \mathfrak{E}_{ct,T}^2.
\end{equation}
This error indicator can be used to perform the adaptive FEM of \textbf{Algorithm} \ref{Algorithm1} with $\mathcal{E}_{ocp,T}$ replaced by $\mathfrak{E}_{ocp,T}$. We shall denote by $\mathfrak{e}_{y}$, $\mathfrak{e}_{p}$, and $\mathfrak{e}_{u}$ the approximation errors related to the state, adjoint state, and control variables, respectively, when the error indicator $\mathfrak{E}_{ocp,T}$ is considered in \textbf{Algorithm} \ref{Algorithm1}. We measure the total error of the underlying AFEM with $\VERT \mathfrak{e}\VERT_\Omega=\VERT(\mathfrak{e}_{y},\mathfrak{e}_{p},\mathfrak{e}_{u})\VERT_{\Omega}$, where $\VERT \cdot \VERT_{\Omega}$ is defined in \eqref{def:total_error}. Finally, we introduce the effectivity indices $\Upsilon_{\mathcal{E}} := \mathcal{E}_{ocp}/\VERT e \VERT_\Omega$ and $\Upsilon_{\mathfrak{E}} := \mathfrak{E}_{ocp}/\VERT \mathfrak{e} \VERT_\Omega$.

In Figures \ref{fig:ex-1.1} and \ref{fig:ex-1.2} we present the results obtained for Example 1. In Figure \ref{fig:ex-1.1} we present, for $\nu=10^{-3}$, experimental rates of convergence for all the individual contributions of the total error $\VERT e \VERT_\Omega$ when uniform and adaptive refinement are considered. We also present the adaptively refined mesh obtained after $24$ adaptive loops.  We observe that our adaptive loop \emph{outperforms} uniform refinement. In addition, we observe optimal experimental rates of convergence for all the individual contributions of the total error $\VERT e \VERT_\Omega$. We also observe that most of the adaptive refinement occurs near to the interface of the control variable and the geometric singularity of the L--shaped domain, which attests to the efficiency of the devised estimator; see subfigure (C). In Figure \ref{fig:ex-1.2}, we present, for $\nu\in\{10^{-4},10^{-5}\}$, experimental rates of convergence for the all the contributions of the total errors $\VERT e \VERT_\Omega$ and $\VERT \mathfrak{e}\VERT_\Omega$ and all the individual contributions of the a posteriori error estimators $\mathcal{E}_{ocp}$ and $\mathfrak{E}_{ocp}$ as well as the effectivity indices $\Upsilon_{\mathcal{E}}$ and $\Upsilon_{\mathfrak{E}}$. We observe that the behavior of the individual contributions of the total errors and error estimators associated to the state and adjoint variables are quite \emph{similar} for both adaptive strategies. However, we observe an \emph{important} difference when we compare the individual contributions associated to the control variable. In fact, as it can be observed from subfigures (B.3) and (D.3), the error norm $\|\mathfrak{e}_u\|_{L^2(\Omega)}$ \emph{do not exhibit an optimal experimental rate of convergence}, while the error norm $\|e_u\|_{L^2(\Omega)}$ associated to our devised AFEM based on the error estimator $\mathcal{E}_{ocp}$ does. Finally, we observe, from subfigures (E) and (F), that the effectivity index $\Upsilon_{\mathcal{E}}$ is close to $1$ for the two different values of $\nu$ that we consider. This shows the accuracy of the proposed a posteriori error estimator $\mathcal{E}_{ocp}$ when used in the adaptive loop described in \textbf{Algorithm} \ref{Algorithm1}.


\begin{figure}[!ht]
\centering
\psfrag{error st}{{\large $\|\nabla e_{y}\|_{L^{2}(\Omega)}$}}
\psfrag{error ad}{{\large $\|\nabla e_{p}\|_{L^{2}(\Omega)}$}}
\psfrag{error ct}{{\large $\| e_{u}\|_{L^{2}(\Omega)}$}}
\psfrag{e total}{{\large $\VERT{(e_{y},e_{p},e_{u}\VERT}_{\Omega}$}}
\psfrag{orden h 11111}{{\normalsize $\mathsf{Ndof}^{-1/2}$}}
\psfrag{orden h 23333}{{\normalsize $\mathsf{Ndof}^{-1/3}$}}
\psfrag{Ndofs}{{\large $\mathsf{Ndof}$}}
\begin{minipage}[c]{0.34\textwidth}\centering
\psfrag{errors - l0001}{\hspace{-0.3cm}\large{Errors with unif. refinement}}
\includegraphics[trim={0 0 0 0},clip,width=4.4cm,height=3.2cm,scale=0.33]{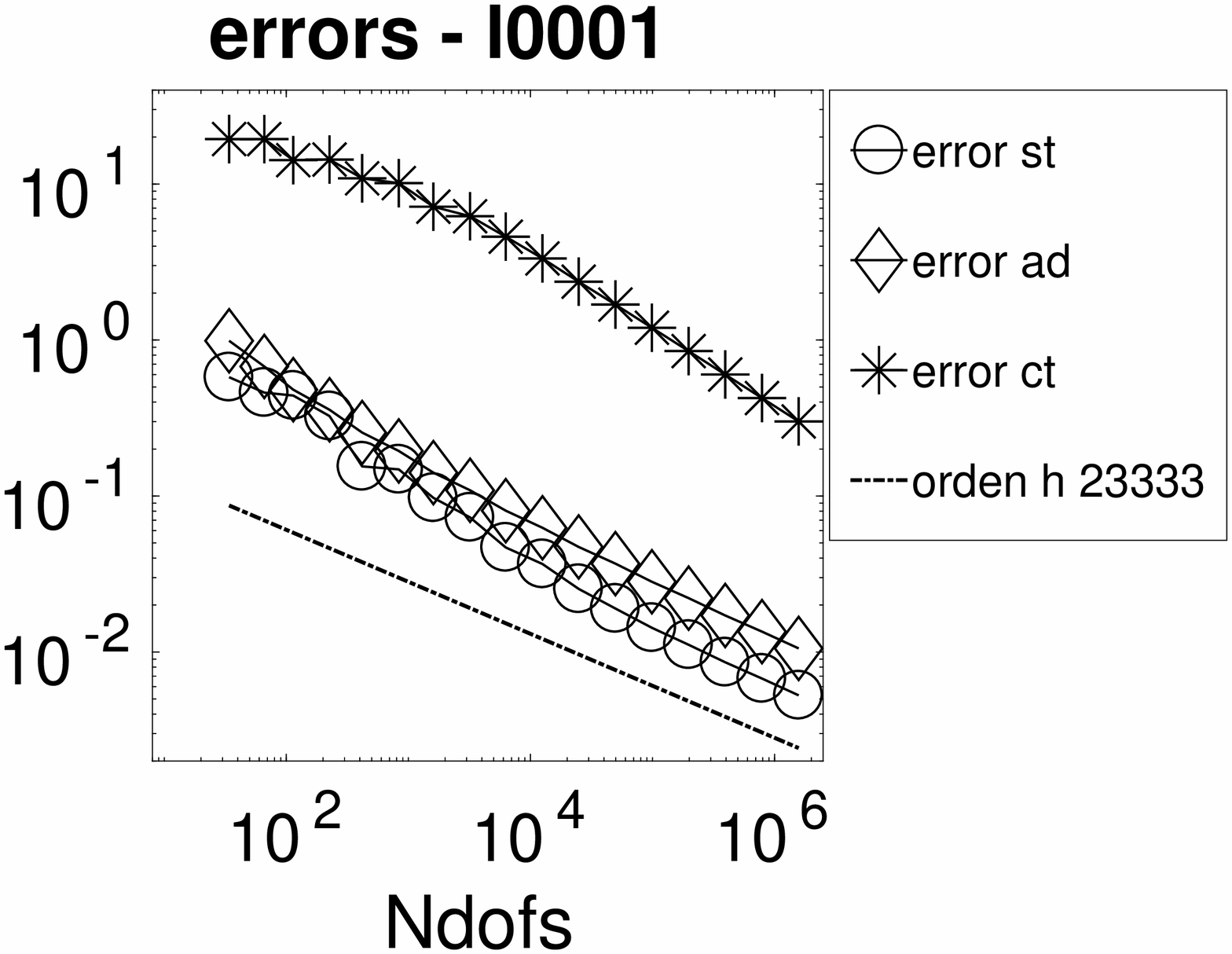}\\
\hspace{-1.0cm}\tiny{(A)}
\end{minipage}
\begin{minipage}[c]{0.34\textwidth}\centering
\psfrag{errors - l0001}{\hspace{-0.3cm}\large{Errors with adap. refinement}}
\includegraphics[trim={0 0 0 0},clip,width=4.4cm,height=3.2cm,scale=0.33]{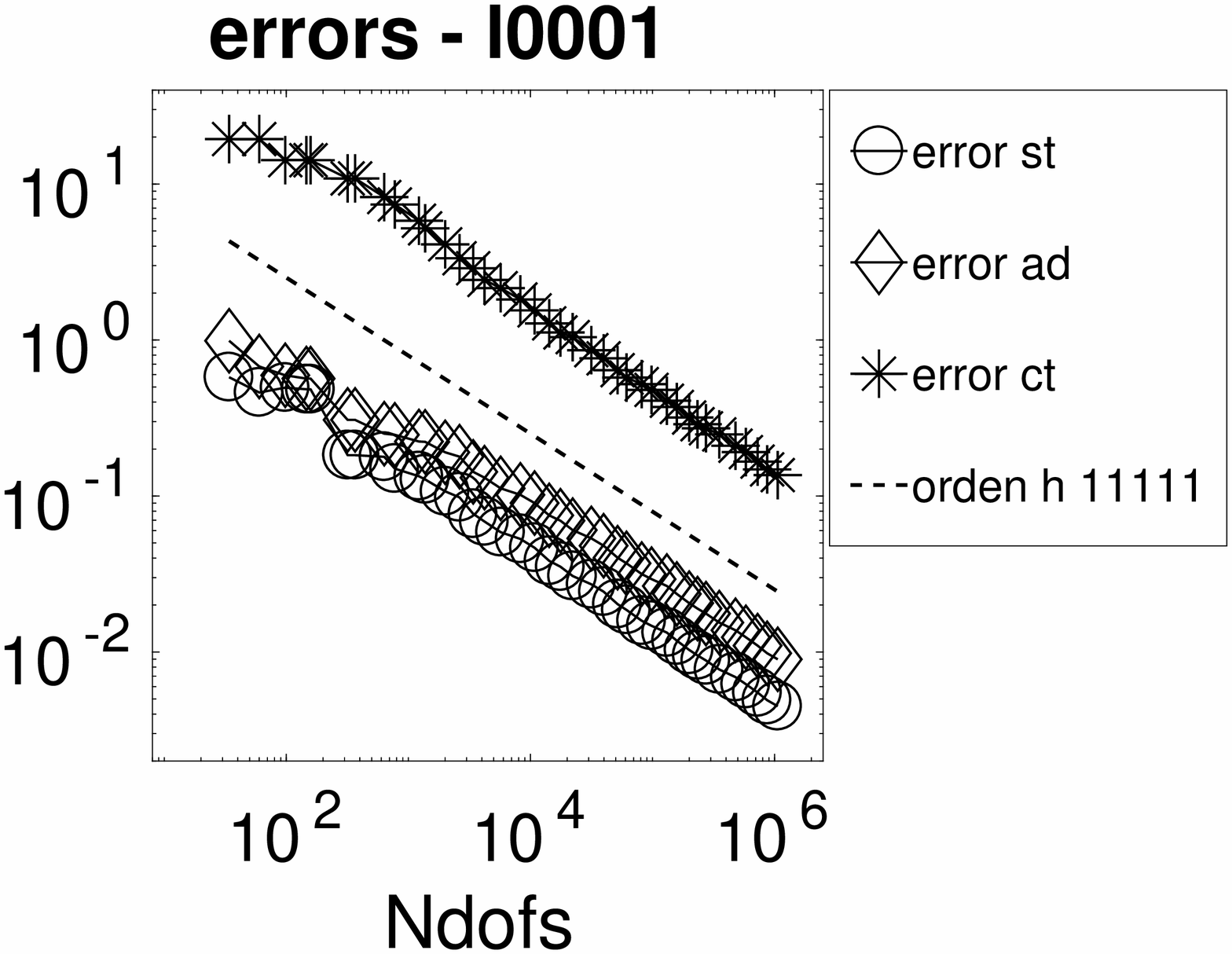}\\
\hspace{-1.0cm}\tiny{(B)}
\end{minipage}
\begin{minipage}[c]{0.28\textwidth}\centering
\psfrag{errors - l0001}{\hspace{-0.3cm}\large{Errors with adap. refinement}}
\includegraphics[trim={0 0 0 0},clip,width=3.3cm,height=2.7cm,scale=0.33]{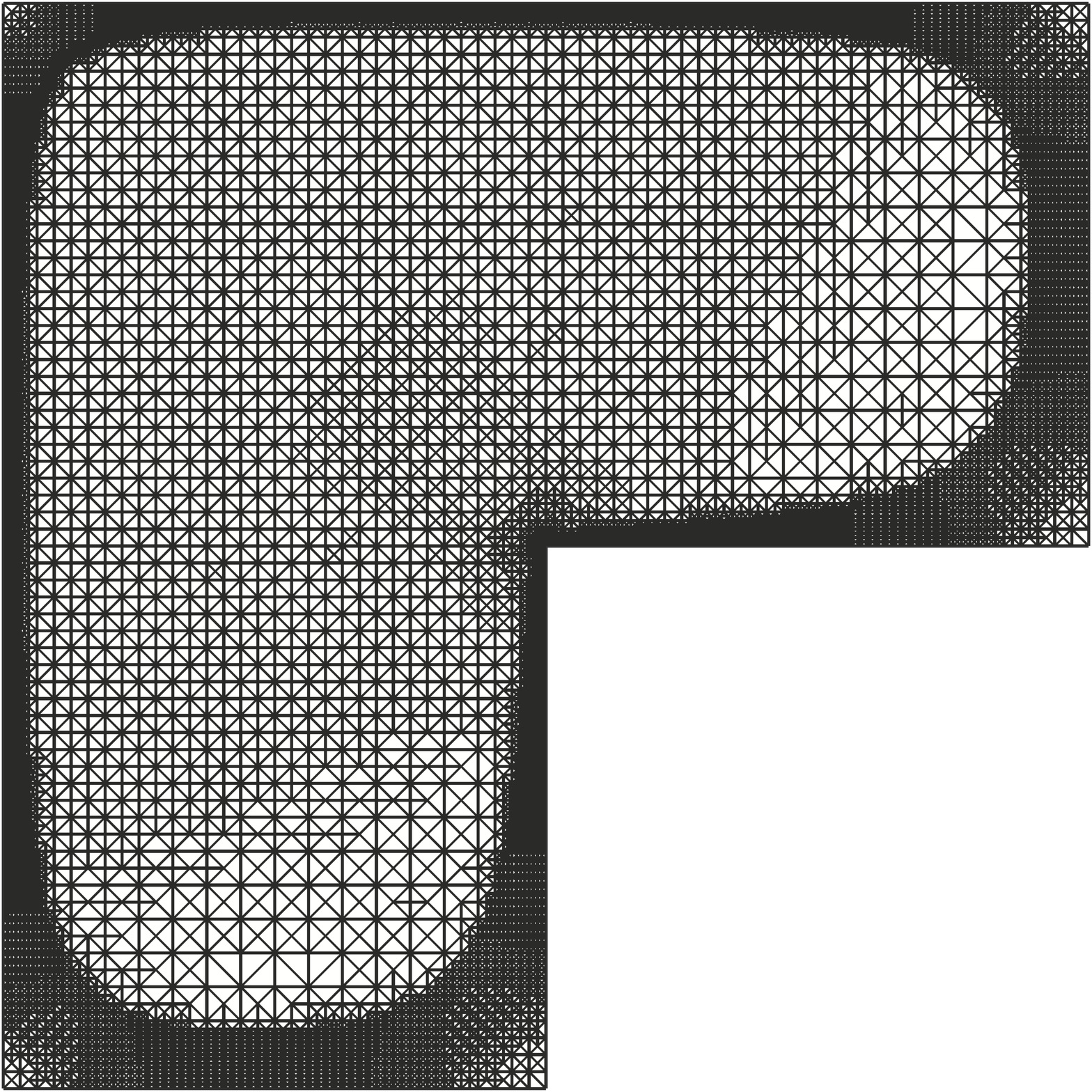}\\
\hspace{-1.0cm}\tiny{(C)}
\end{minipage}
\caption{Example 1. Experimental rates of convergence for the individual contributions $\|\nabla e_{y}\|_{L^{2}(\Omega)},\|\nabla e_{p}\|_{L^{2}(\Omega)}$, and $\| e_{u}\|_{L^{2}(\Omega)}$ for uniform (A) and adaptive refinement (B) and the $24$th adaptively refined mesh (C) for $\nu = 10^{-3}$.}
\label{fig:ex-1.1}
\end{figure}


\begin{figure}[!ht]
\centering
\psfrag{Ndof-32}{{\normalsize $\mathsf{Ndof}^{-3/2}$}}
\psfrag{orden h}{{\normalsize $\mathsf{Ndof}^{-1/2}$}}
\psfrag{orden h 11111}{{\normalsize $\mathsf{Ndof}^{-1/2}$}}
\psfrag{orden h 11}{{\normalsize $\mathsf{Ndof}^{-1/2}$}}
\psfrag{orden h 1}{{\normalsize $\mathsf{Ndof}^{-1/2}$}}
\psfrag{Ndofs}{{\large $\mathsf{Ndof}$}}
\begin{minipage}[c]{0.247\textwidth}\centering
{\footnotesize{Estimator contributions}\\
\footnotesize{for $\nu = 10^{-4}$}}~\\~\\
\psfrag{eta st n}{{\large $\mathcal{E}_{st}$}}
\psfrag{eta ad n}{{\large $\mathcal{E}_{ad}$}}
\psfrag{eta ct n}{{\large $\mathcal{E}_{ct}$}}
\psfrag{eta st c}{{\large $\mathfrak{E}_{st}$}}
\psfrag{eta ad c}{{\large $\mathfrak{E}_{ad}$}}
\psfrag{eta ct c}{{\large $\mathfrak{E}_{ct}$}}
\psfrag{orden h 1}{{\normalsize $\mathsf{Ndof}^{-1/2}$}}
\psfrag{Ndofs}{{\large $\mathsf{Ndof}$}}
\psfrag{etas st}{\hspace{-1.1cm}\large{State error estimators }}
\includegraphics[trim={0 0 0 0},clip,width=3.2cm,height=2.8cm,scale=0.32]{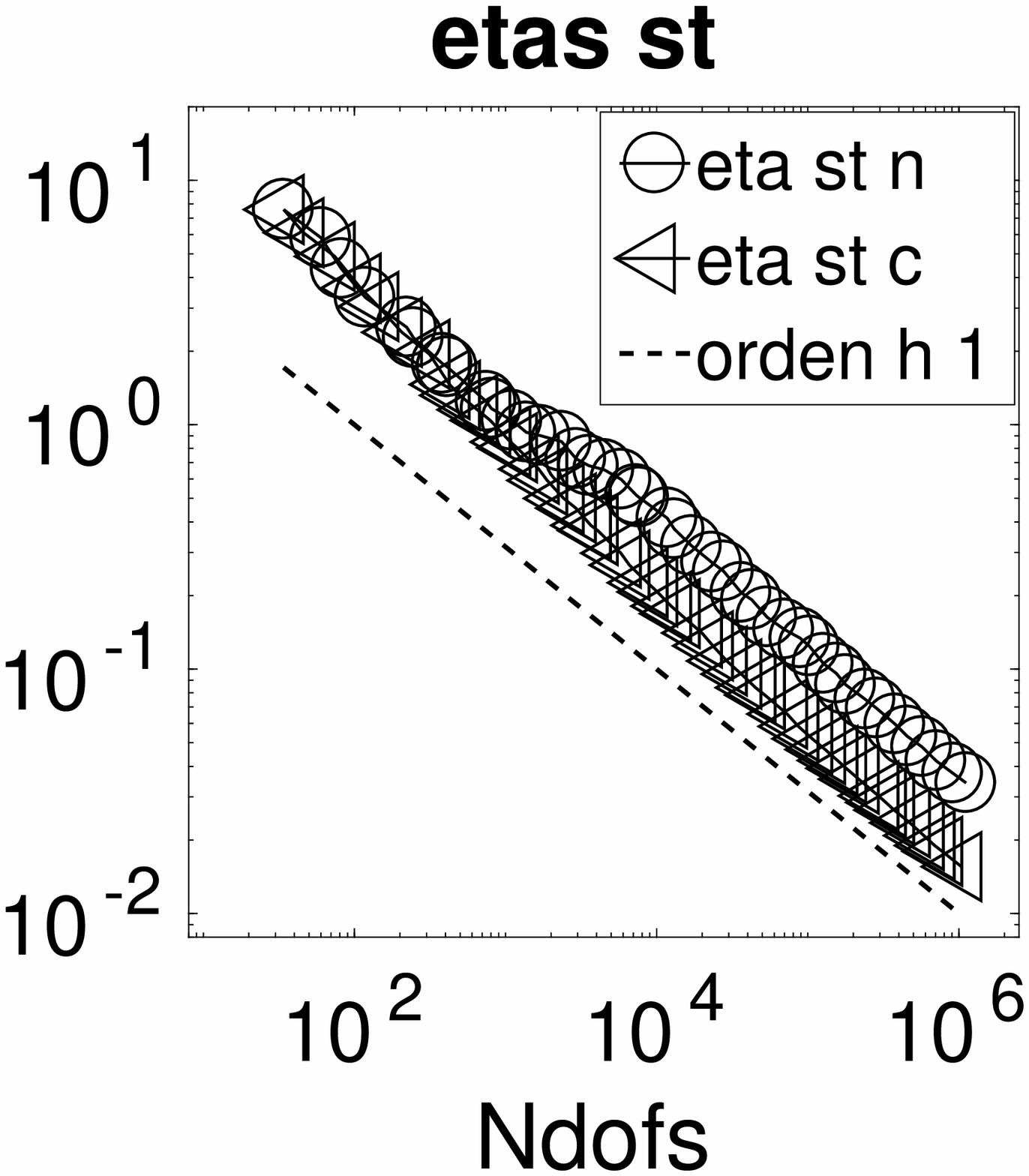}\\
\hspace{0.3cm}\tiny{(A.1)}~\\~\\
\psfrag{etas ad}{\hspace{-1.3cm}\large{Adjoint error estimators }}
\includegraphics[trim={0 0 0 0},clip,width=3.2cm,height=2.8cm,scale=0.32]{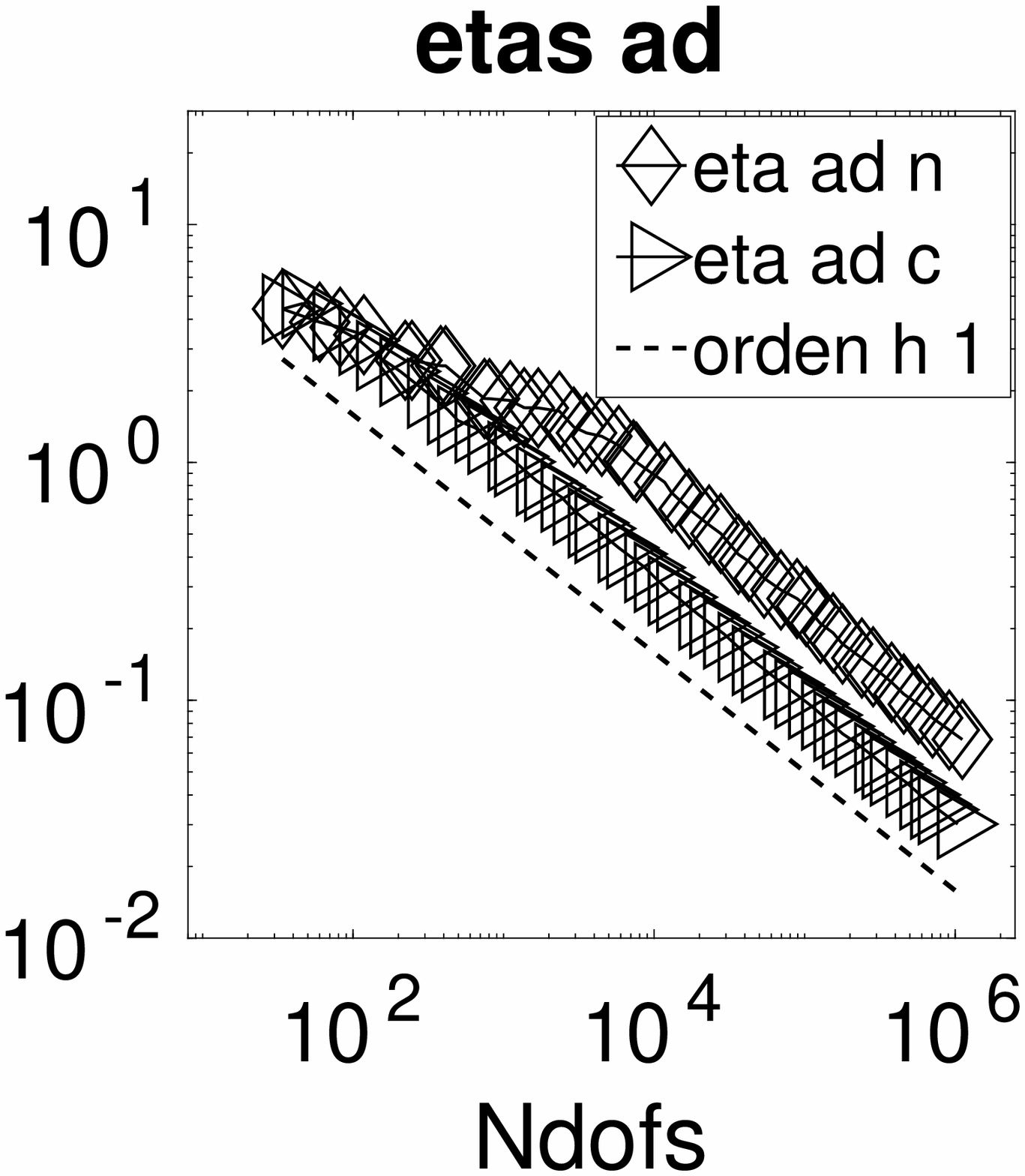}\\
\hspace{0.3cm}\tiny{(A.2)}~\\~\\
\psfrag{etas ct}{\hspace{-1.35cm}\large{Control error estimators }}
\includegraphics[trim={0 0 0 0},clip,width=3.2cm,height=2.8cm,scale=0.32]{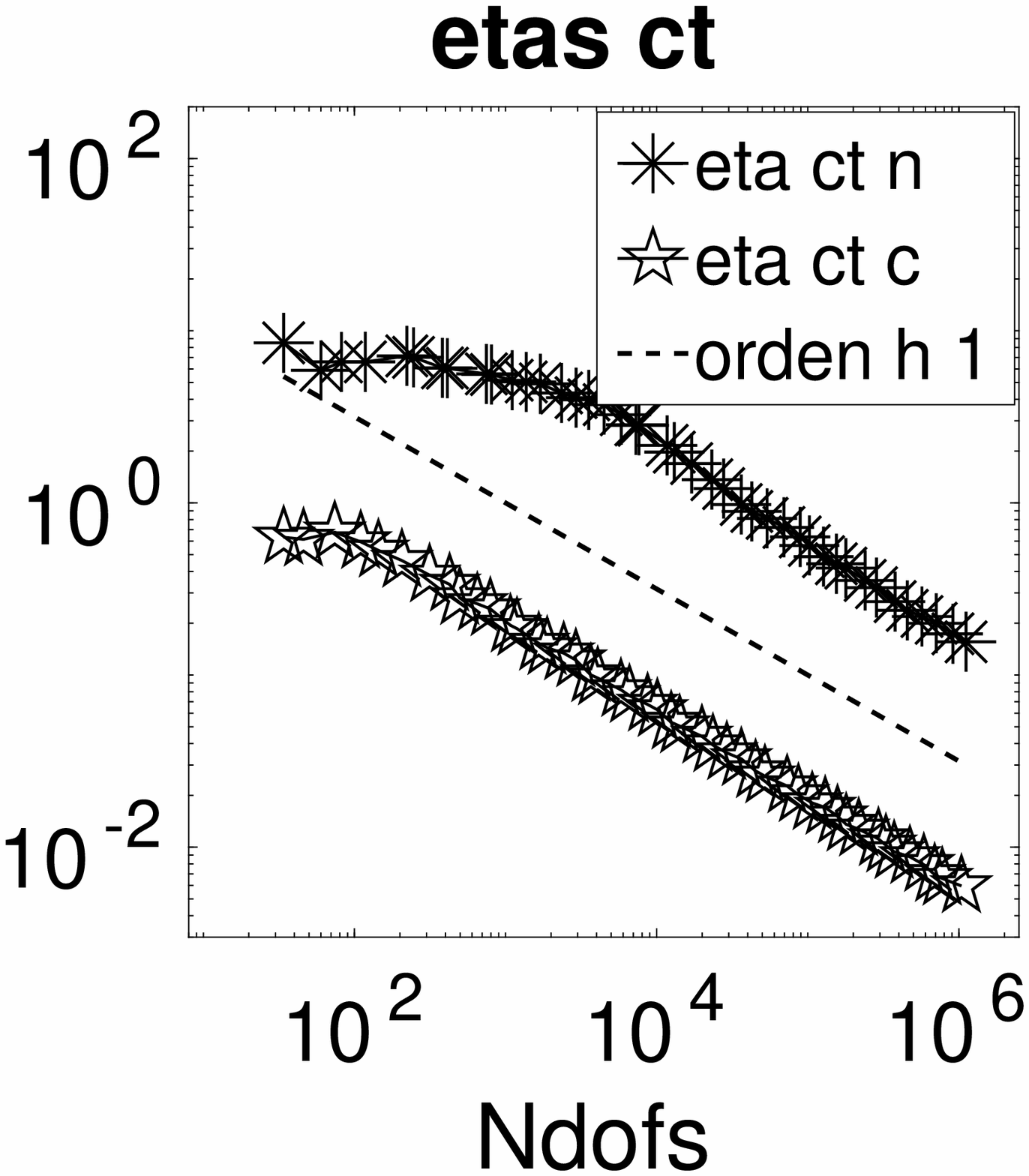}\\
\hspace{0.3cm}\tiny{(A.3)}~\\~\\
\end{minipage}
\begin{minipage}[c]{0.247\textwidth}\centering
{\footnotesize{Error contributions}\\
\footnotesize{for $\nu = 10^{-4}$}}~\\~\\
\psfrag{errores st n}{{\large $ \|\nabla e_{y}\|_{L^{2}(\Omega)}$}}
\psfrag{errores ad n}{{\large $ \|\nabla e_{p}\|_{L^{2}(\Omega)}$}}
\psfrag{errores ct n}{{\large $\|e_{u}\|_{L^{2}(\Omega)}$}}
\psfrag{errores st c}{{\large $ \|\nabla \mathfrak{e}_{y}\|_{L^{2}(\Omega)}$}}
\psfrag{errores ad c}{{\large $ \|\nabla \mathfrak{e}_{p}\|_{L^{2}(\Omega)}$}}
\psfrag{errores ct c}{{\large $\|\mathfrak{e}_{u}\|_{L^{2}(\Omega)}$}}
\psfrag{orden h 11}{{\normalsize $\mathsf{Ndof}^{-1/2}$}}
\psfrag{orden h 1111}{{\normalsize $\mathsf{Ndof}^{-1/2}$}}
\psfrag{Ndofs}{{\large $\mathsf{Ndof}$}}
\psfrag{errores st}{\hspace{-0.2cm}\large{State errors}}
\includegraphics[trim={0 0 0 0},clip,width=3.2cm,height=2.8cm,scale=0.32]{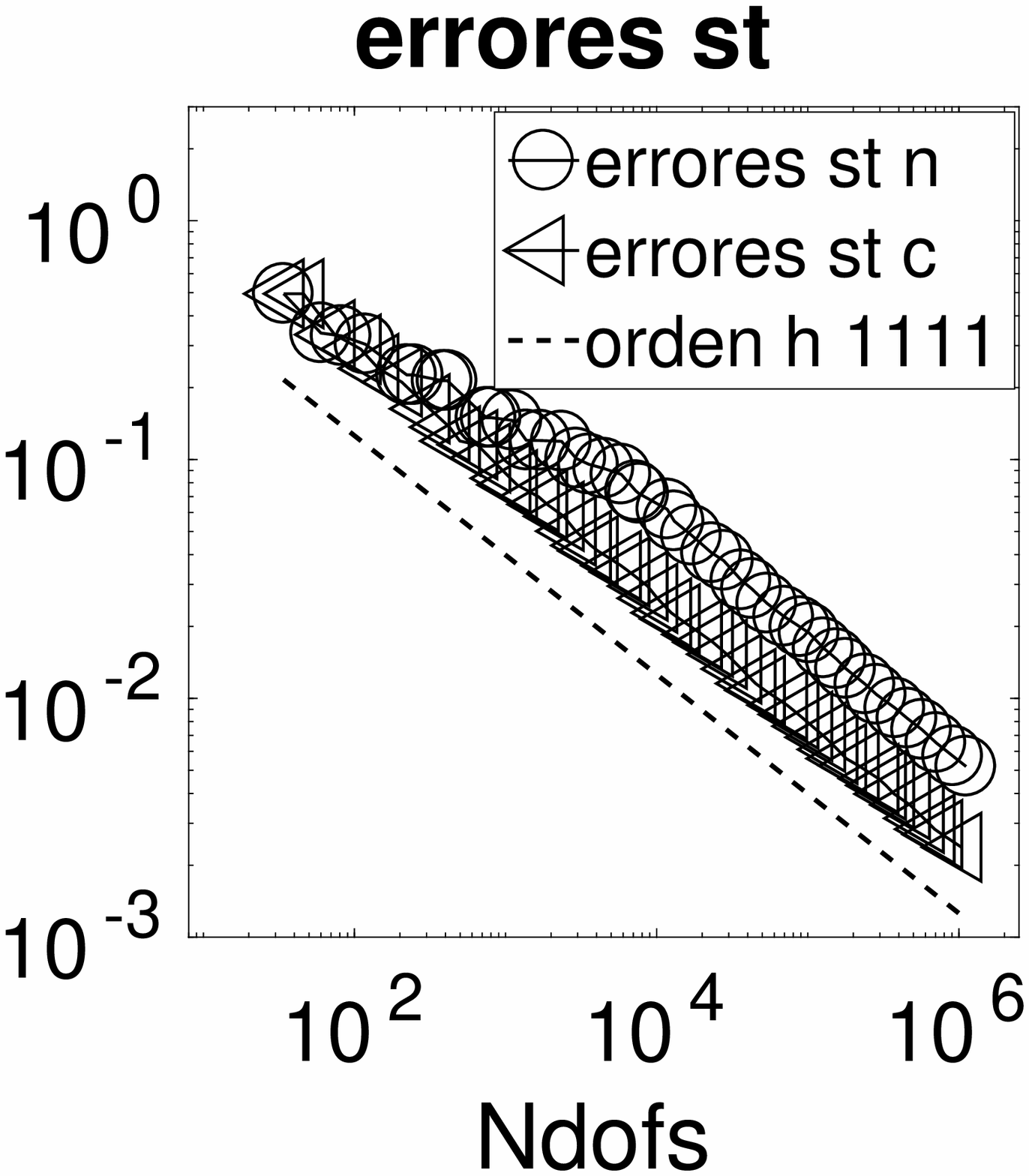}\\
\hspace{0.3cm}\tiny{(B.1)}~\\~\\
\psfrag{errores ad}{\hspace{-0.2cm}\large{Adjoint errors}}
\includegraphics[trim={0 0 0 0},clip,width=3.2cm,height=2.8cm,scale=0.32]{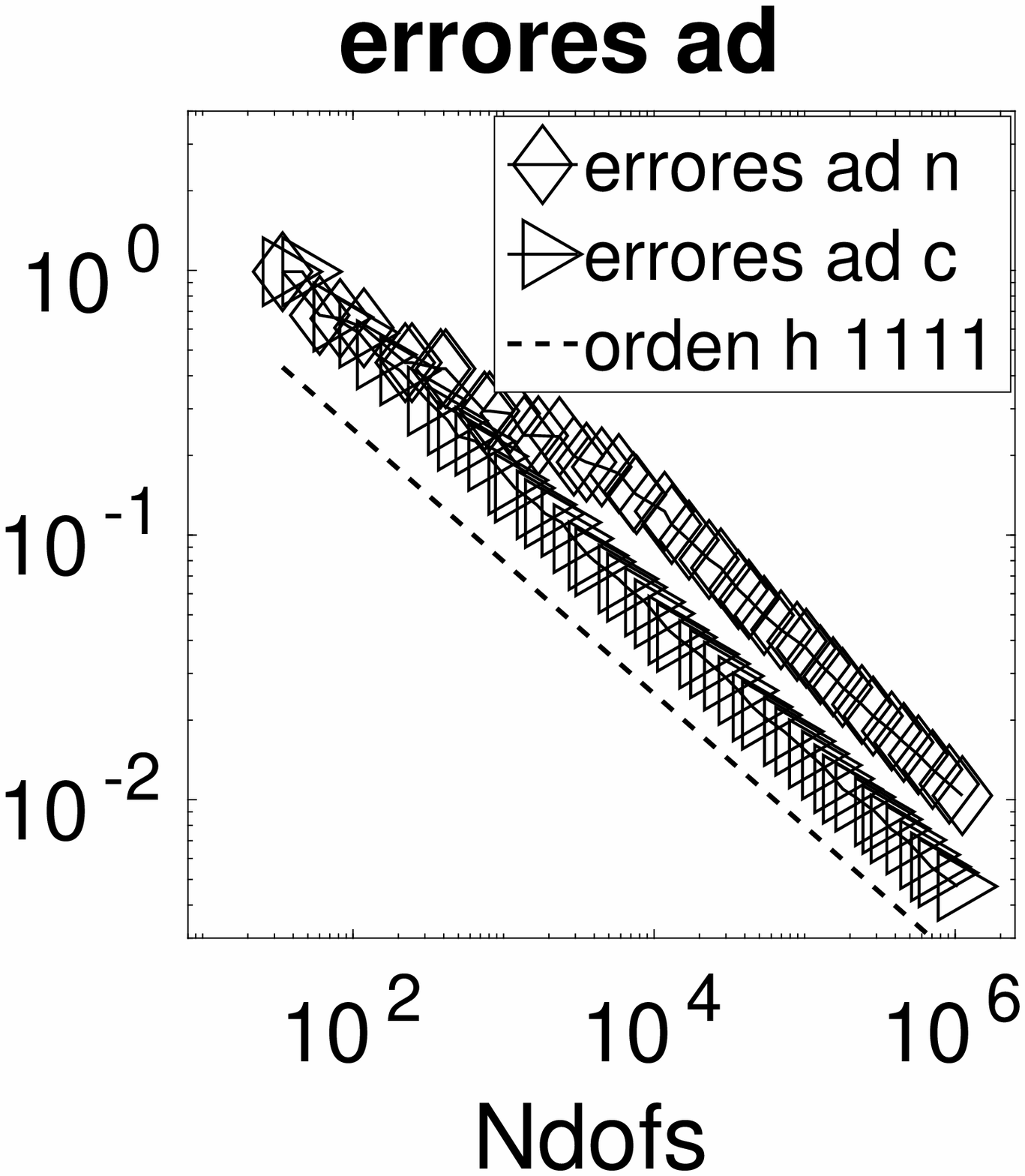}\\
\hspace{0.3cm}\tiny{(B.2)}~\\~\\
\psfrag{errores ct}{\hspace{-0.2cm}\large{Control errors}}
\includegraphics[trim={0 0 0 0},clip,width=3.2cm,height=2.8cm,scale=0.32]{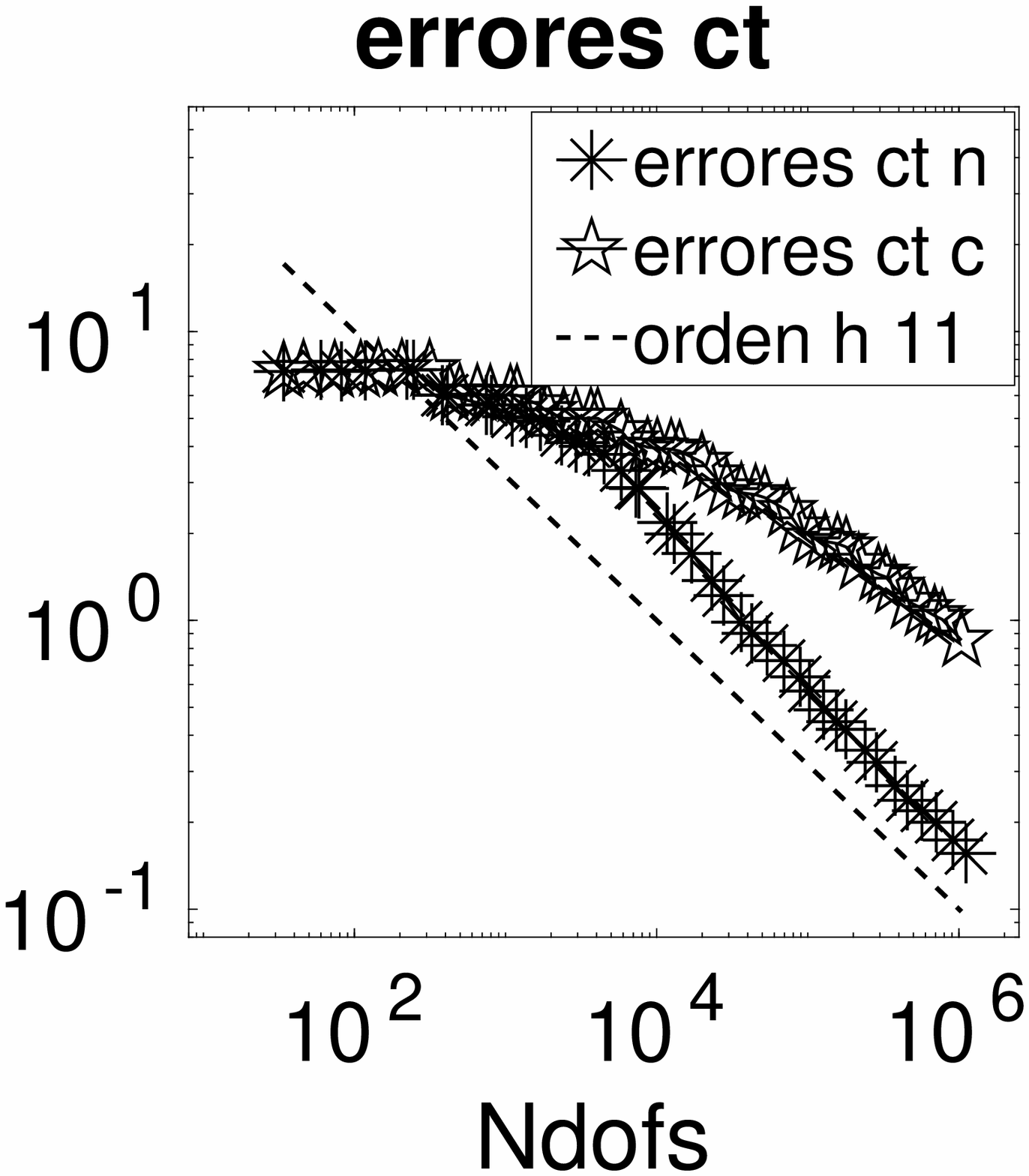}\\
\hspace{0.3cm}\tiny{(B.3)}~\\~\\
\end{minipage}
\begin{minipage}[c]{0.247\textwidth}\centering
{\footnotesize{Estimator contributions}\\
\footnotesize{for $\nu = 10^{-5}$}}~\\~\\
\psfrag{eta st n}{{\large $\mathcal{E}_{st}$}}
\psfrag{eta ad n}{{\large $\mathcal{E}_{ad}$}}
\psfrag{eta ct n}{{\large $\mathcal{E}_{ct}$}}
\psfrag{eta st c}{{\large $\mathfrak{E}_{st}$}}
\psfrag{eta ad c}{{\large $\mathfrak{E}_{ad}$}}
\psfrag{eta ct c}{{\large $\mathfrak{E}_{ct}$}}
\psfrag{orden h 111}{{\normalsize $\mathsf{Ndof}^{-1/2}$}}
\psfrag{Ndofs}{{\large $\mathsf{Ndof}$}}
\psfrag{etas st}{\hspace{-1.1cm}\large{State error estimators }}
\includegraphics[trim={0 0 0 0},clip,width=3.2cm,height=2.8cm,scale=0.32]{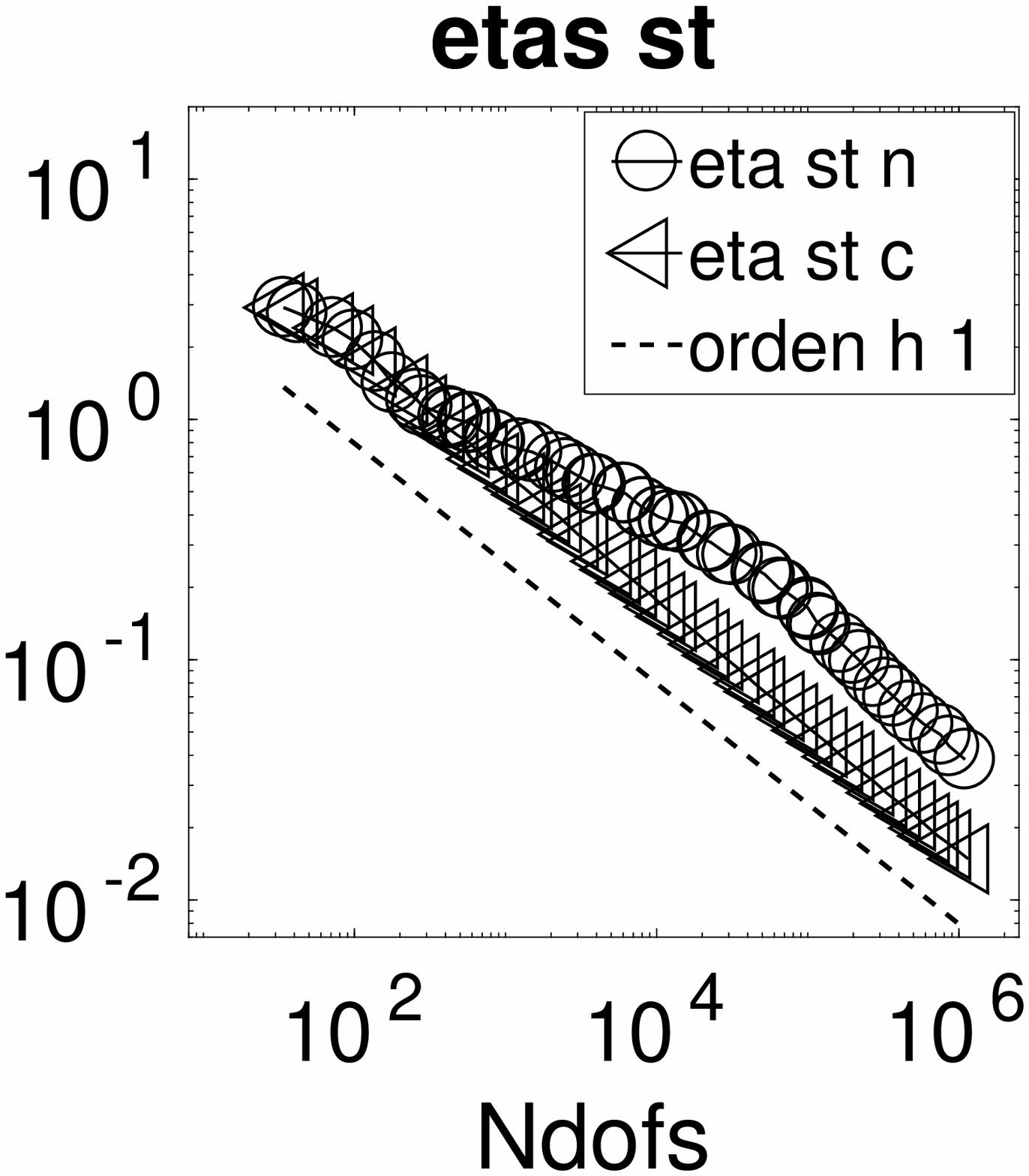}\\
\hspace{-0.0cm}\tiny{(C.1)}~\\~\\
\psfrag{etas ad}{\hspace{-1.3cm}\large{Adjoint error estimators }}
\includegraphics[trim={0 0 0 0},clip,width=3.2cm,height=2.8cm,scale=0.32]{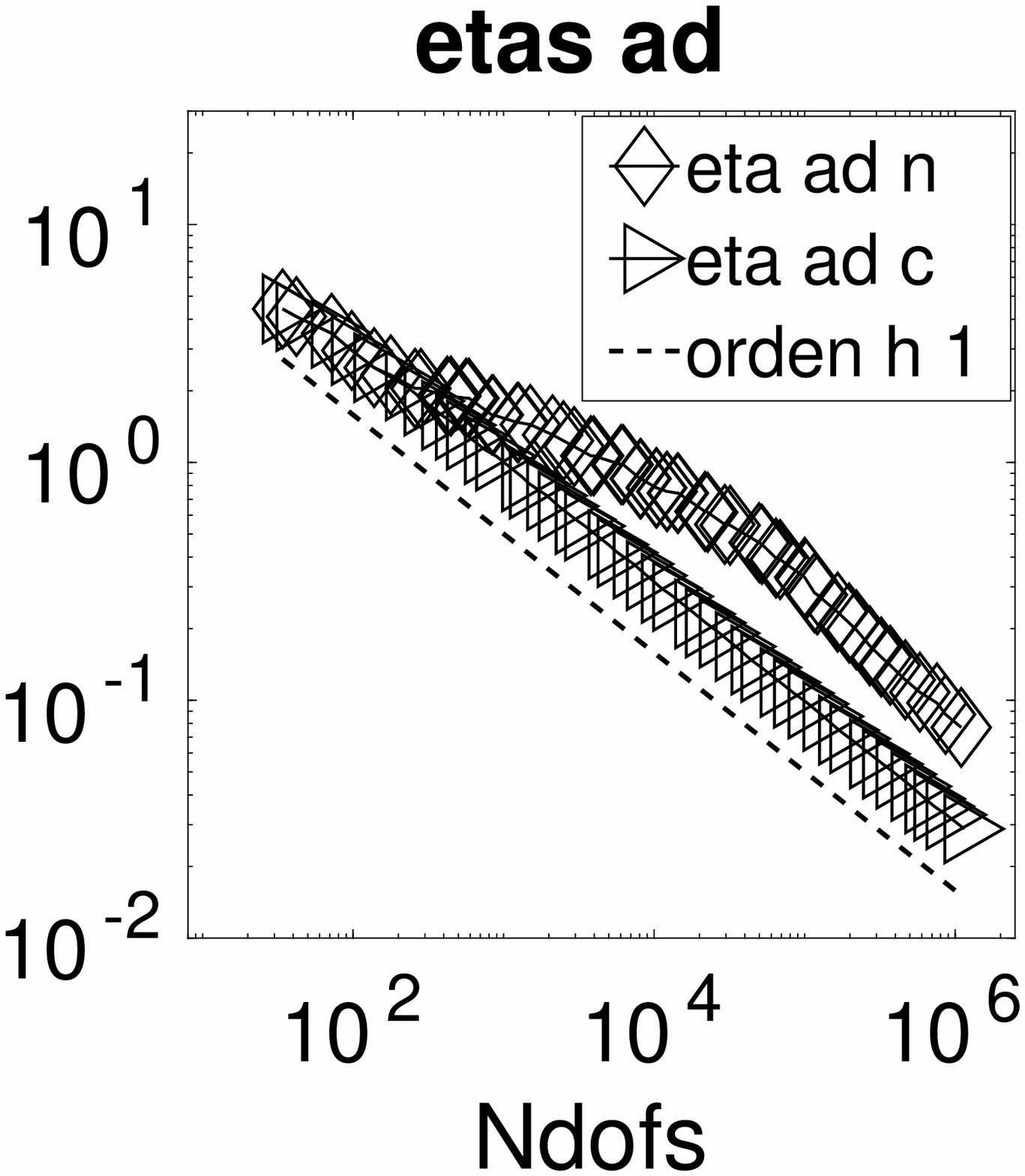}\\
\hspace{-0.0cm}\tiny{(C.2)}~\\~\\
\psfrag{etas ct}{\hspace{-1.35cm}\large{Control error estimators }}
\includegraphics[trim={0 0 0 0},clip,width=3.2cm,height=2.8cm,scale=0.32]{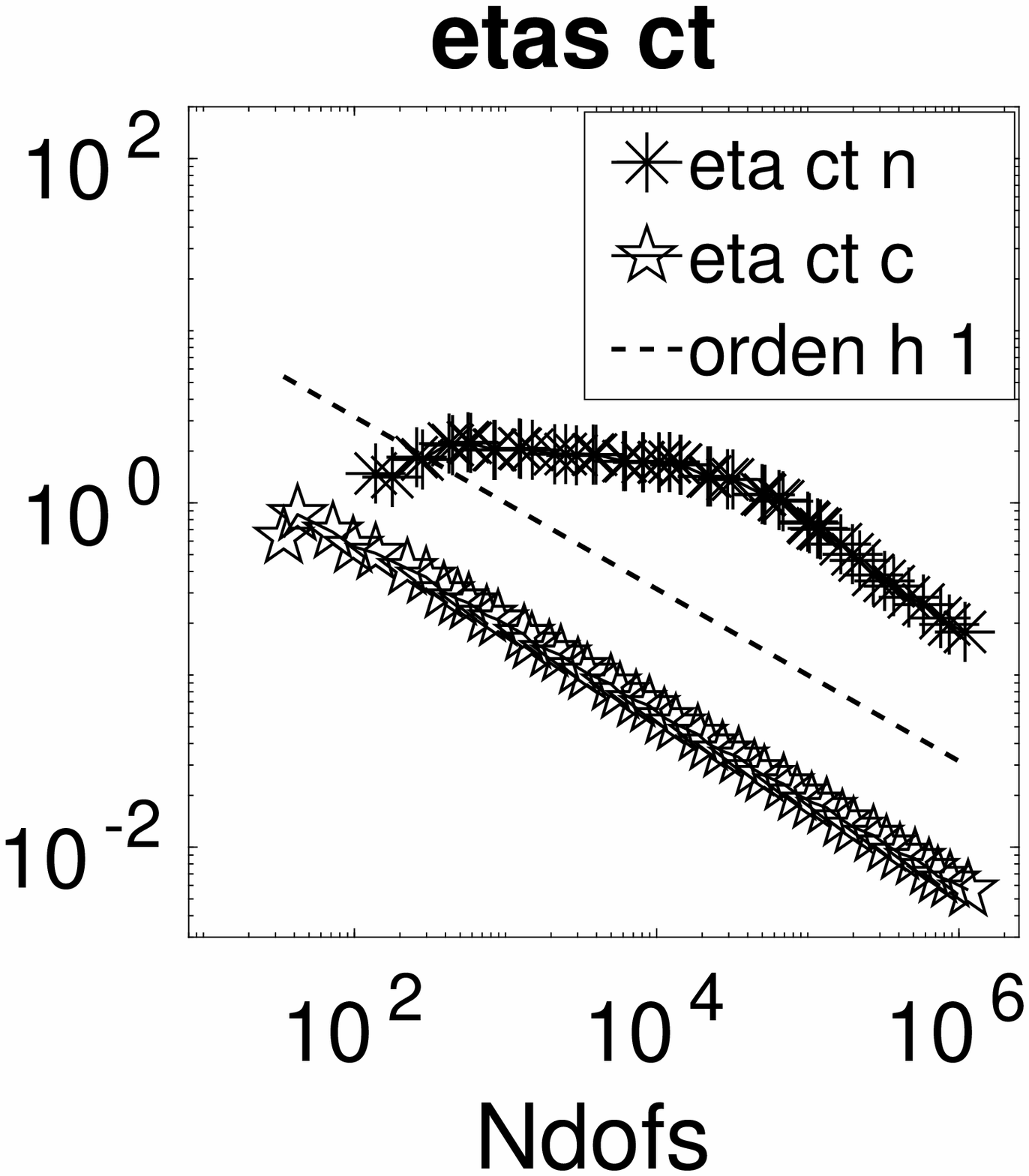}\\
\hspace{-0.0cm}\tiny{(C.3)}~\\~\\
\end{minipage}
\begin{minipage}[c]{0.247\textwidth}\centering
{\footnotesize{Error contributions}\\
\footnotesize{for $\nu = 10^{-5}$}}~\\~\\
\psfrag{errores st n}{{\large $ \|\nabla e_{y}\|_{L^{2}(\Omega)}$}}
\psfrag{errores ad n}{{\large $ \|\nabla e_{p}\|_{L^{2}(\Omega)}$}}
\psfrag{errores ct n}{{\large $\|e_{u}\|_{L^{2}(\Omega)}$}}
\psfrag{errores st c}{{\large $ \|\nabla \mathfrak{e}_{y}\|_{L^{2}(\Omega)}$}}
\psfrag{errores ad c}{{\large $ \|\nabla \mathfrak{e}_{p}\|_{L^{2}(\Omega)}$}}
\psfrag{errores ct c}{{\large $\|\mathfrak{e}_{u}\|_{L^{2}(\Omega)}$}}
\psfrag{orden h 1111}{{\normalsize $\mathsf{Ndof}^{-1/2}$}}
\psfrag{Ndofs}{{\large $\mathsf{Ndof}$}}
\psfrag{errores st}{\hspace{-0.2cm}\large{State errors}}
\includegraphics[trim={0 0 0 0},clip,width=3.2cm,height=2.8cm,scale=0.32]{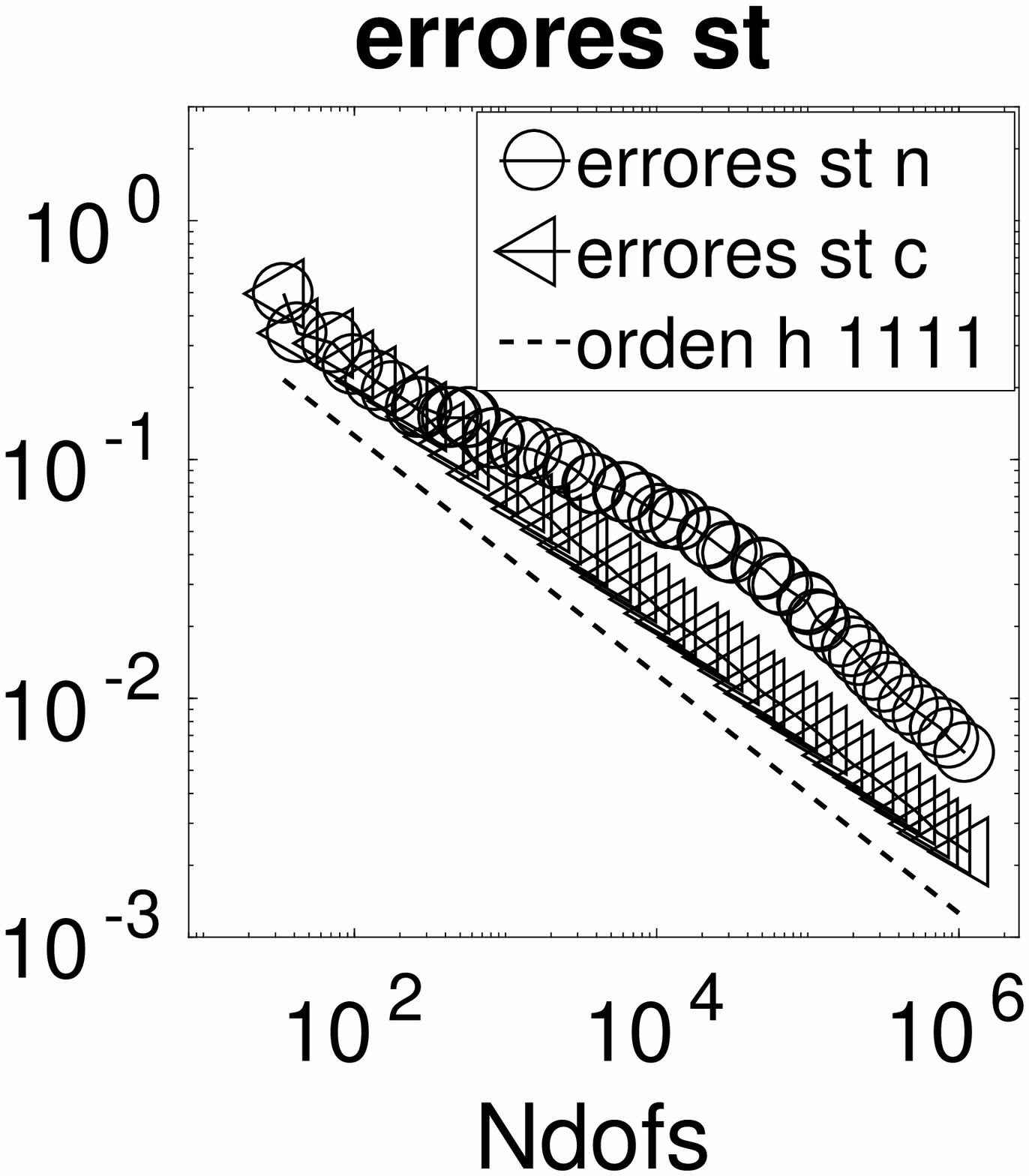}\\
\hspace{0.3cm}\tiny{(D.1)}~\\~\\
\psfrag{errores ad}{\hspace{-0.2cm}\large{Adjoint errors}}
\includegraphics[trim={0 0 0 0},clip,width=3.2cm,height=2.8cm,scale=0.32]{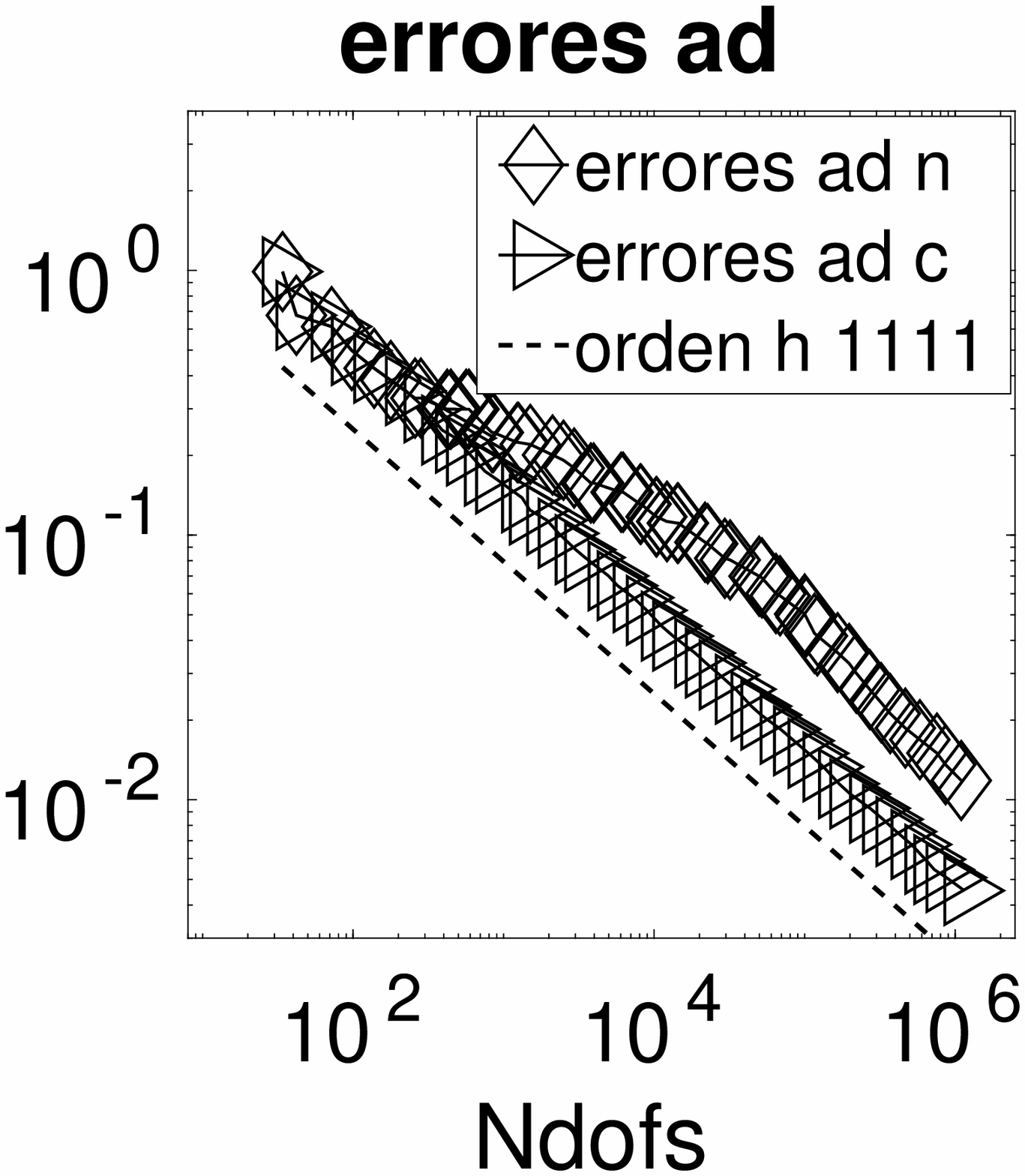}\\
\hspace{0.3cm}\tiny{(D.2)}~\\~\\
\psfrag{errores ct}{\hspace{-0.2cm}\large{Control errors}}
\includegraphics[trim={0 0 0 0},clip,width=3.2cm,height=2.8cm,scale=0.32]{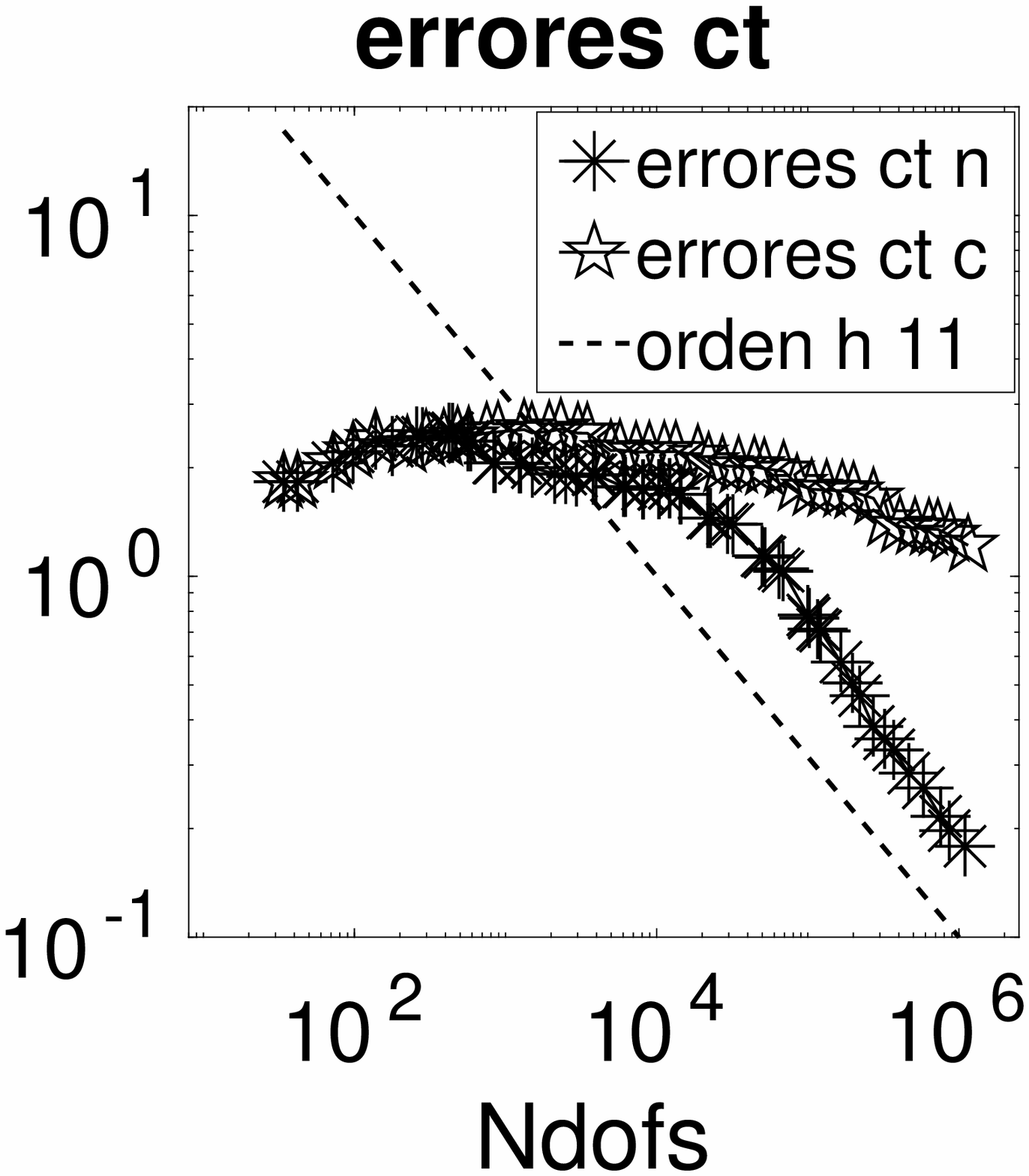}\\
\hspace{0.3cm}\tiny{(D.3)}~\\~\\
\end{minipage}
\\ ~ \\
\begin{minipage}[c]{0.4\textwidth}\centering
\psfrag{IE N}{{\large $\Upsilon_{\mathcal{E}}$}}
\psfrag{IE C}{{\large $\Upsilon_{\mathfrak{E}}$}}
\psfrag{indice ef}{\hspace{-1.4cm}\large{Effectivity indices for $\nu = 10^{-4}$}}
\includegraphics[trim={0 0 0 0},clip,width=3.8cm,height=3.0cm,scale=0.32]{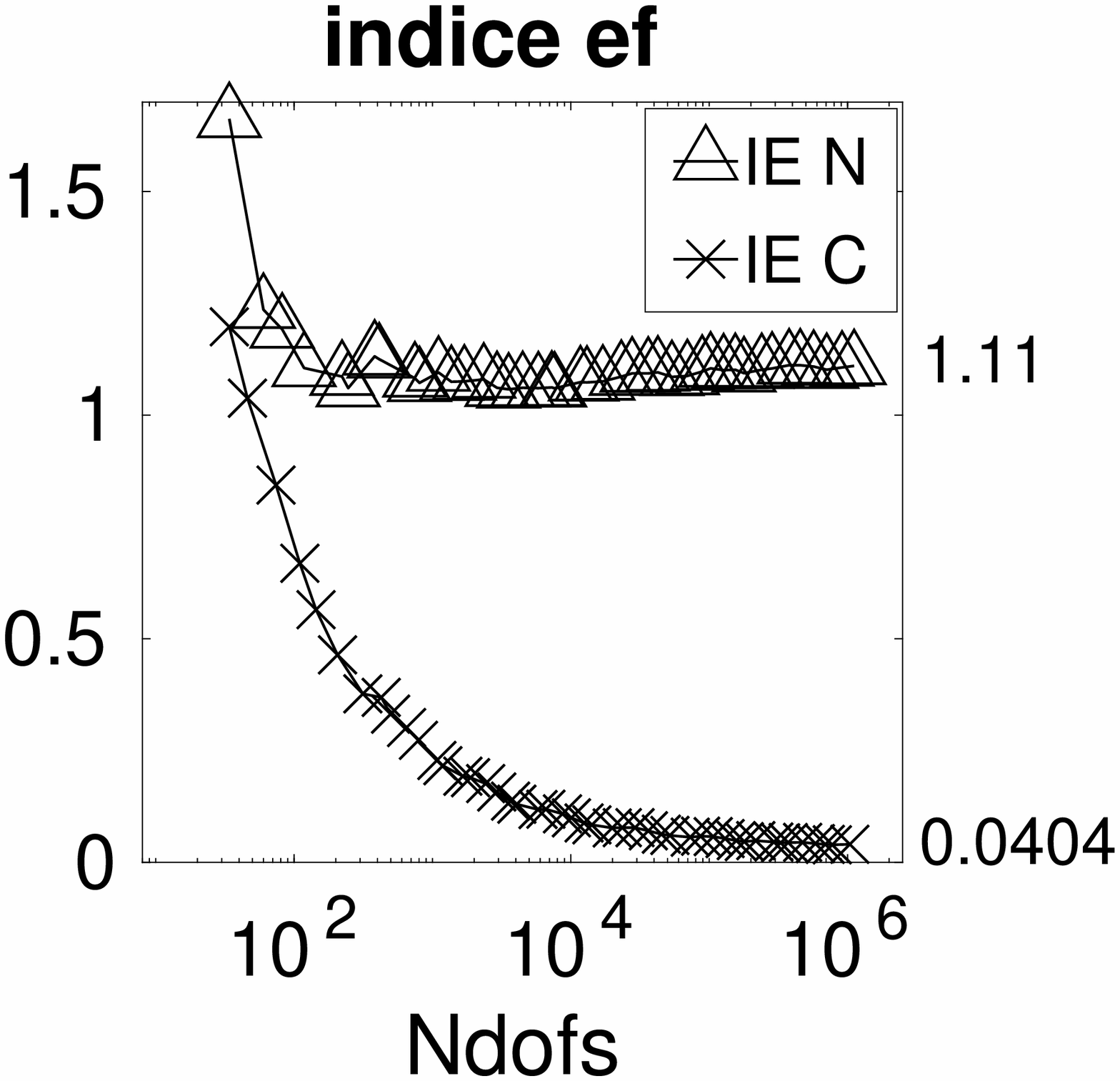}\\
\hspace{0.3cm}\tiny{(E)}
\end{minipage}
\begin{minipage}[c]{0.4\textwidth}\centering
\psfrag{IE N}{{\large $\Upsilon_{\mathcal{E}}$}}
\psfrag{IE C}{{\large $\Upsilon_{\mathfrak{E}}$}}
\psfrag{indice ef}{\hspace{-1.4cm}\large{Effectivity indices for $\nu = 10^{-5}$}}
\includegraphics[trim={0 0 0 0},clip,width=3.5cm,height=3.0cm,scale=0.32]{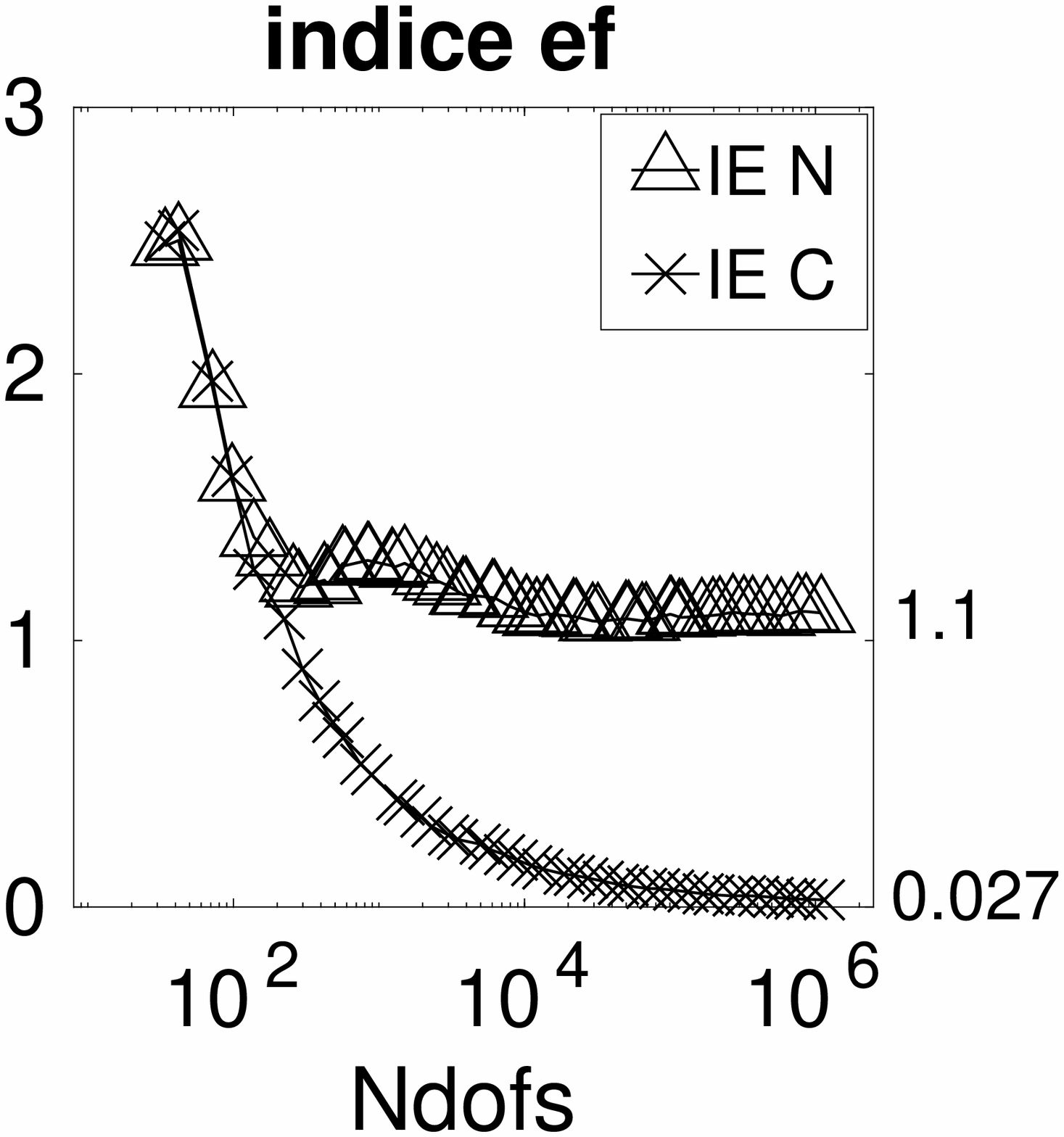}\\
\hspace{0.3cm}\tiny{(F)}~\\~\\
\end{minipage}
\caption{Example 1. Experimental rates of convergence for all the contributions of $\mathcal{E}_{ocp}$ (A.1)--(A.3) and $\mathfrak{E}_{ocp}$ (C.1)--(C.3), experimental rates of convergence for all the contributions of the total errors $\VERT e \VERT_\Omega$ (B.1)--(B.3) and $\VERT \mathfrak{e}\VERT_\Omega$ (D.1)--(D.3), and the effectivity indices $\Upsilon_{\mathcal{E}}$ and $\Upsilon_{\mathfrak{E}}$ with $\nu = 10^{-4}$ (E) and $\nu = 10^{-5}$ (F).}
\label{fig:ex-1.2}
\end{figure}

\textbf{Example 2.} We let $\Omega=(0,1)^3$, $\texttt{a} = -80$, $\texttt{b} = 100$, and $\nu = 10^{-3}$. We consider 
\[f(x_{1},x_{2},x_{3}) = 10, \quad 
y_{\Omega}(x_{1},x_{2},x_{3}) = \left\{\begin{array}{cl}
10^{2}e^{\frac{1}{\xi}}\cos(4\pi\xi), & \text{if } \xi < 0, \\
0, & \text{if } \xi \geq 0,
\end{array}
\right.
\]
where $\xi = \xi(x_{1},x_{2},x_{3}) = 4(x_{1} - 0.5)^{2} + 4(x_{2} - 0.5)^{2} + 4(x_{3} - 0.5)^{2} - 1$.

The purpose of this numerical example is to investigate the performance of the devised error estimator when different choices of the nonlinear function $a$ are considered. Let us, in particular, consider
\[ a_{1}(\cdot,y) =  10y^{3} - 2; \quad a_{2}(\cdot,y) =  10\arctan(80y) - 5; \quad a_{3}(\cdot,y) =  10\sinh(3y) - 2. \]

In Figure \ref{fig:ex_2} we present the results obtained for Example 2. We show, for the considered three different nonlinear functions $a$, experimental rates of convergence for all the individual contributions of the error estimator $\mathcal{E}_{ocp}$ and the obtained $25$th adaptively refined meshes. We observe optimal experimental rates of convergence for all the individual contributions of the error estimator $\mathcal{E}_{ocp}$.


\begin{figure}[!ht]
\centering
\psfrag{eta total}{{\large $\mathcal{E}_{ocp}$}}
\psfrag{eta ad}{{\large $\mathcal{E}_{ad}$}}
\psfrag{eta st}{{\large $\mathcal{E}_{st}$}}
\psfrag{eta ct}{{\large $\mathcal{E}_{ct}$}}
\psfrag{orden h 111}{{\normalsize $\mathsf{Ndof}^{-1/3}$}}
\psfrag{Ndofs}{{\large $\mathsf{Ndof}$}}
\begin{minipage}[c]{0.333\textwidth}\centering
$a = a_{1}$\\
\psfrag{etas}{\hspace{-1.7cm}\large{Estimator contributions}}
\includegraphics[trim={0 0 0 0},clip,width=4.35cm,height=3.8cm,scale=0.33]{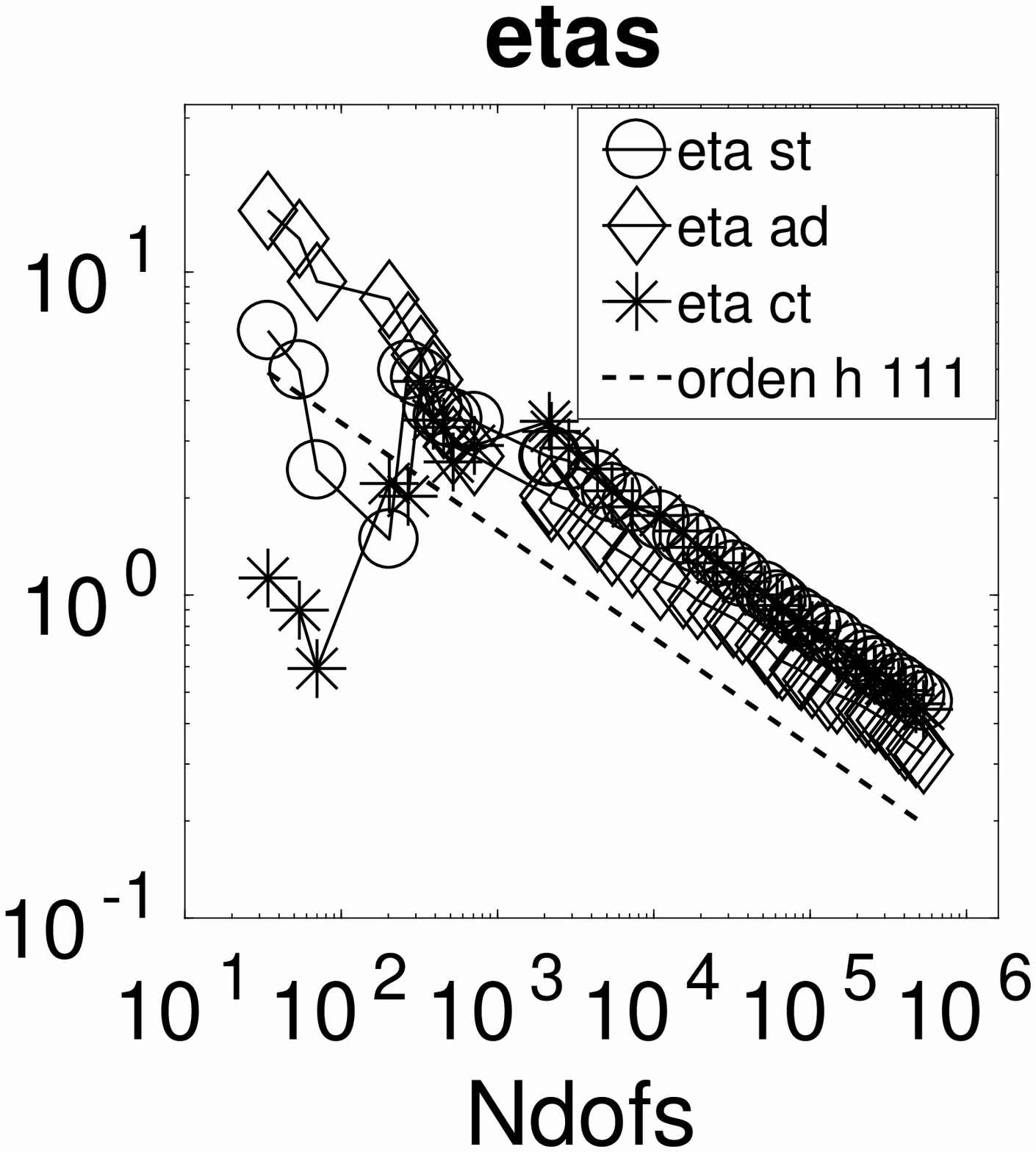}\\
\hspace{0.0cm}\tiny{(A.1)}\\
\includegraphics[trim={0 0 0 0},clip,width=3.4cm,height=3.4cm,scale=0.3]{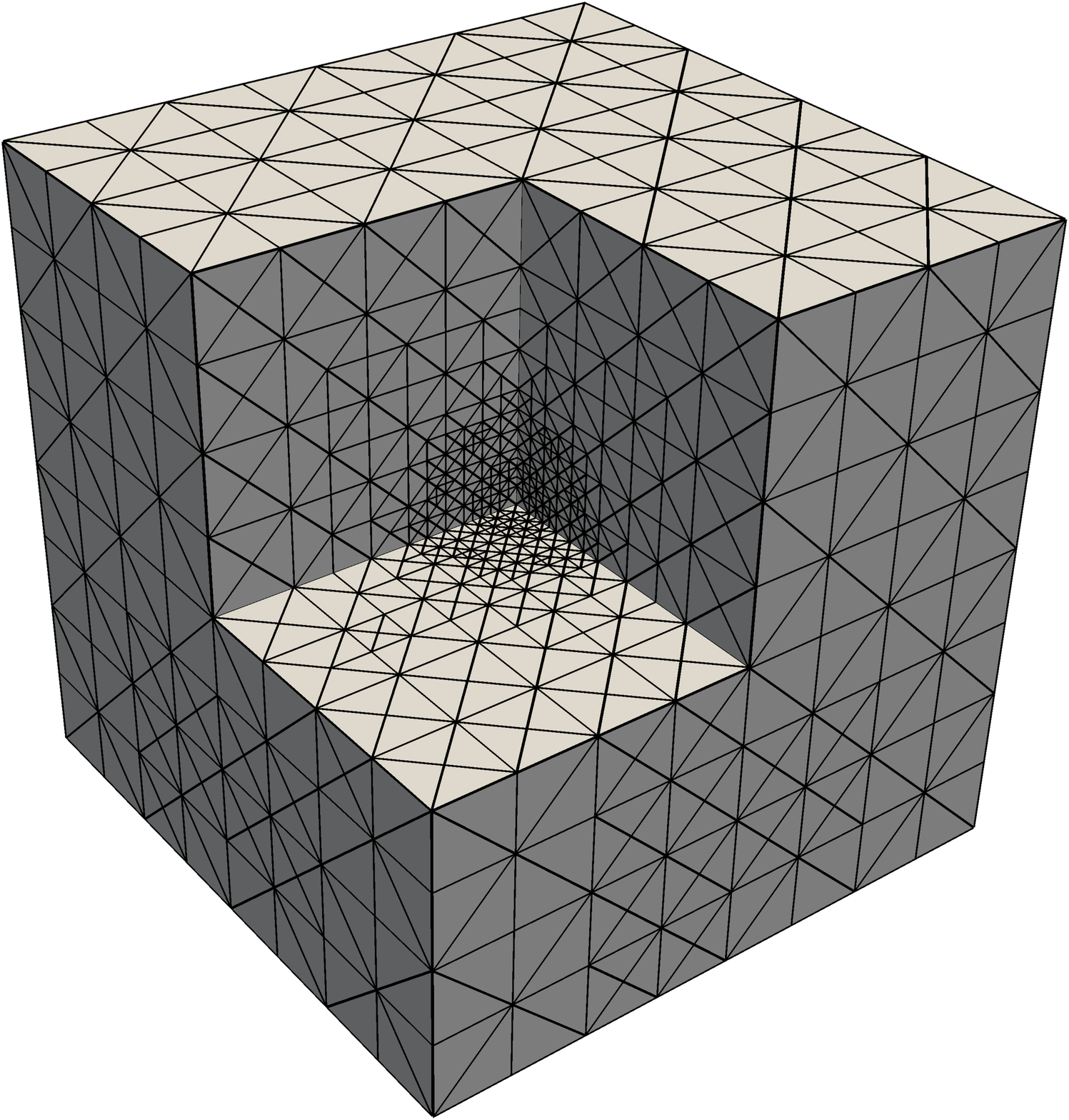}\\
\hspace{0.0cm}\tiny{(B.1)}\\
\end{minipage}
\begin{minipage}[c]{0.333\textwidth}\centering
$a = a_{2}$\\
\psfrag{etas}{\hspace{-1.7cm}\large{Estimator contributions}}
\includegraphics[trim={0 0 0 0},clip,width=4.35cm,height=3.8cm,scale=0.33]{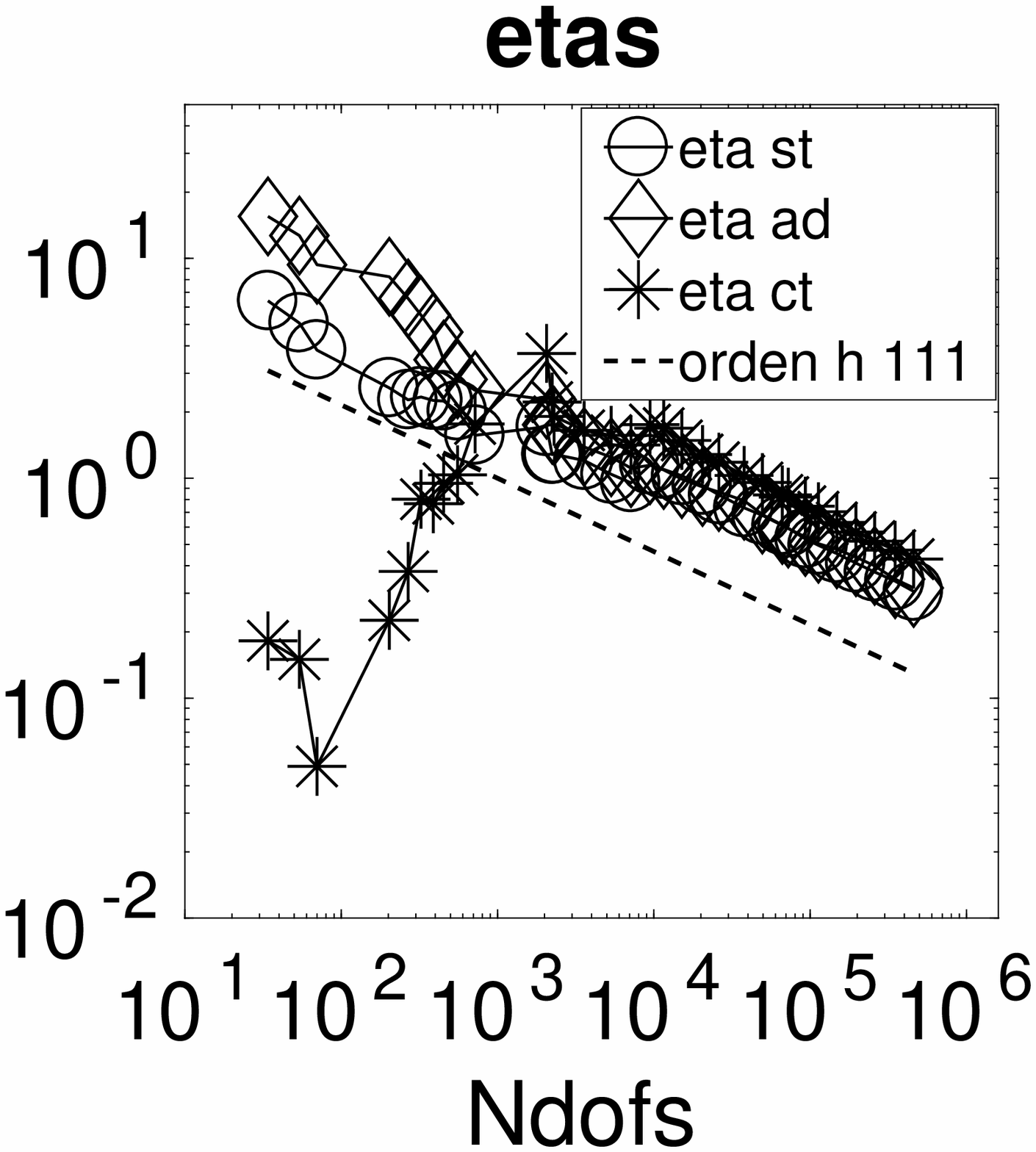}\\
\hspace{0.0cm}\tiny{(A.2)}\\
\includegraphics[trim={0 0 0 0},clip,width=3.4cm,height=3.4cm,scale=0.3]{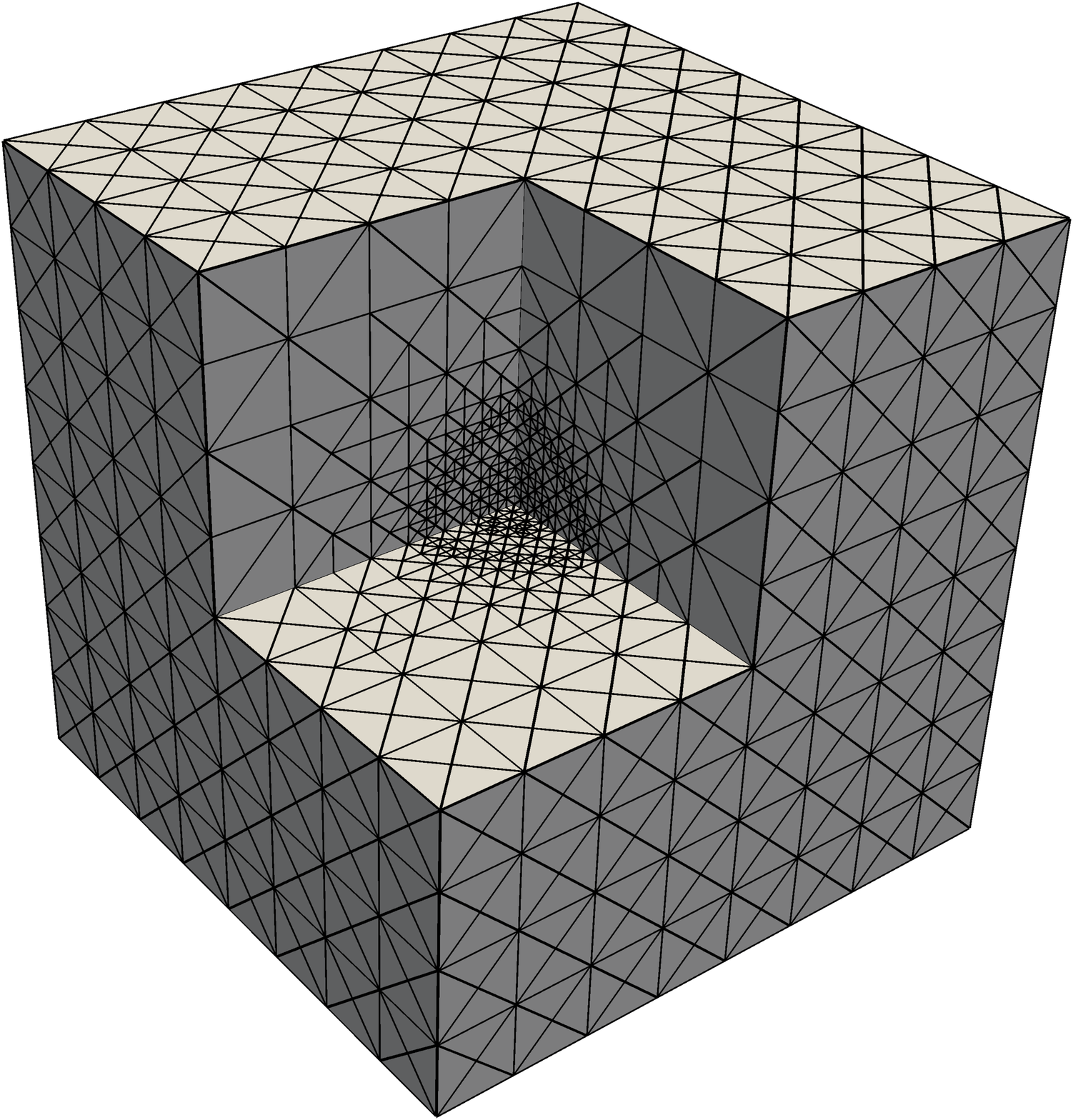}\\
\hspace{0.0cm}\tiny{(B.2)}\\
\end{minipage}
\begin{minipage}[c]{0.333\textwidth}\centering
$a = a_{3}$\\
\psfrag{etas}{\hspace{-1.7cm}\large{Estimator contributions}}
\includegraphics[trim={0 0 0 0},clip,width=4.35cm,height=3.8cm,scale=0.33]{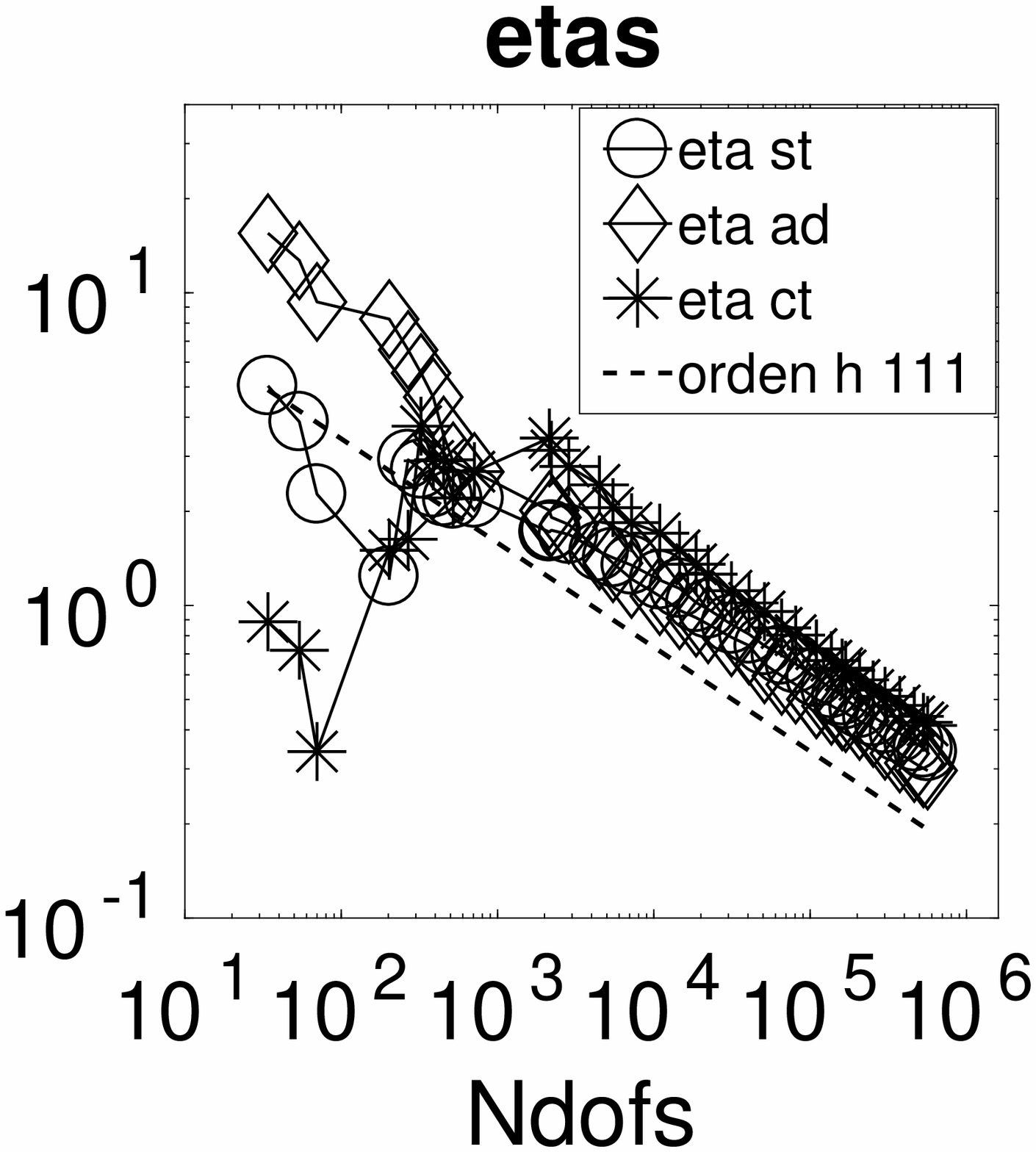}\\
\hspace{0.0cm}\tiny{(A.3)}\\
\includegraphics[trim={0 0 0 0},clip,width=3.4cm,height=3.4cm,scale=0.3]{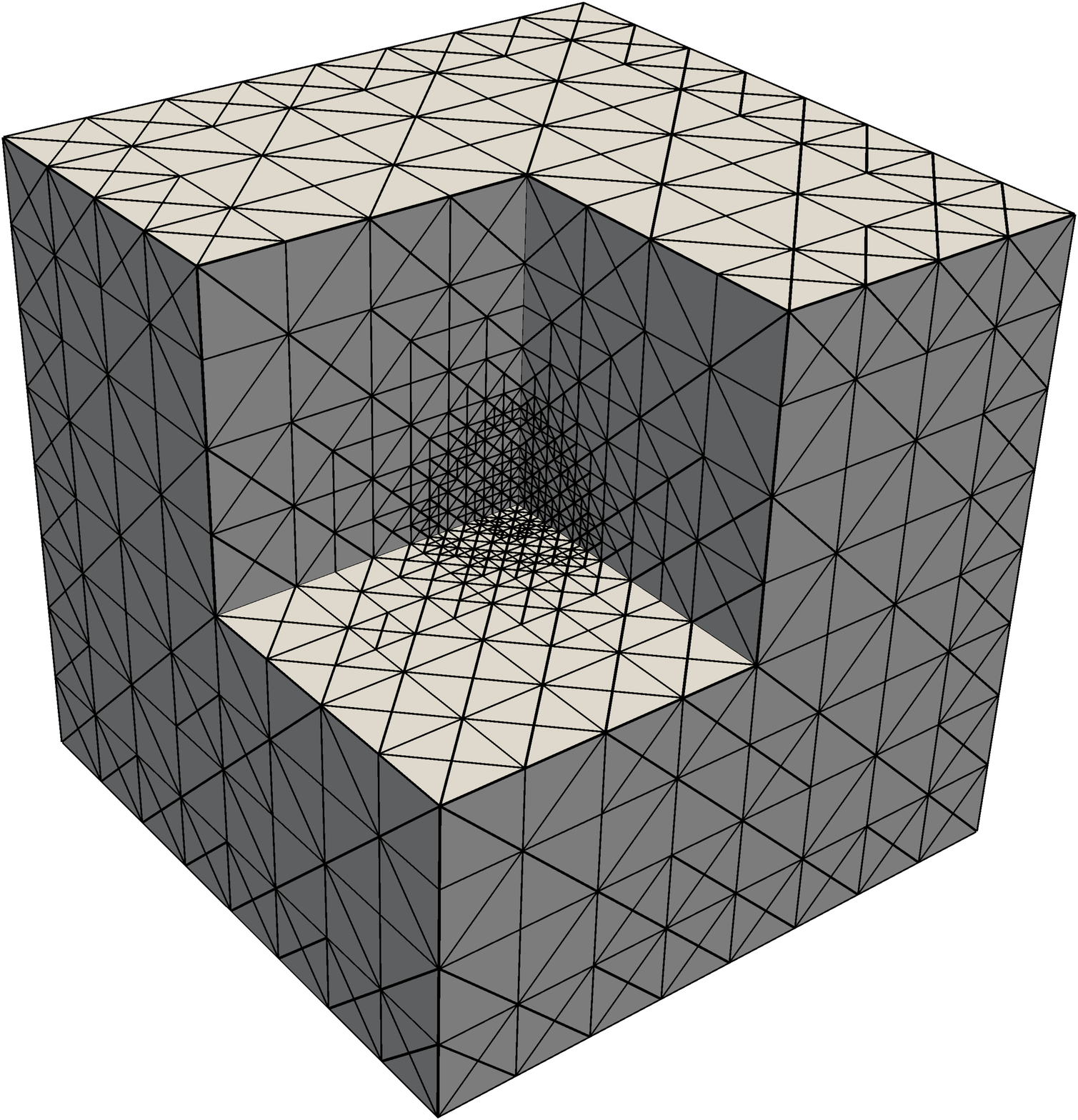}\\
\hspace{0.0cm}\tiny{(B.3)}\\
\end{minipage}
\caption{Example 2: Experimental rates of convergence for $\mathcal{E}_{st}$, $\mathcal{E}_{ad}$, and $\mathcal{E}_{ct}$ (A.1)--(A.3) and adaptively refined meshes obtained after 25 adaptive loops (B.1)--(B.3) with $\nu = 10^{-3}$. }
\label{fig:ex_2}
\end{figure}

\subsection{Conclusions}

We present the following conclusions:
~\\
$\bullet$ Most of the refinement occurs near to the interface of the control variable. This attests to the efficiency of the devised estimator. When the domain involves geometric singularities, refinement is also being performed in regions that are close to them. This shows a competitive performance of the a posteriori error estimator.
~\\
$\bullet$ All the individual contributions of the total error $\VERT e\VERT_\Omega$ exhibit optimal experimental rates of convergence for all the experiments and the nonlinear functions $a$ considered in the experiments that we have performed.
~\\
$\bullet$ The devised a posteriori error estimator, defined in \eqref{def:error_estimator_ocp}, is able to recognize the interface of $\bar{u}_{\T}$. This estimator also delivers, for all the numerical experiments that we have performed, optimal experimental rates of convergence. This is not the case when the error estimator \eqref{eq:chinese_total_indicator} is used in \textbf{Algorithm} \ref{Algorithm1}.


\bibliographystyle{siam}
\footnotesize
\bibliography{biblio}

\end{document}